\documentclass{amsart}
\usepackage[marginparwidth=1in]{geometry}
\usepackage{url}
\usepackage{amsthm}
\usepackage{amssymb}
\usepackage{latexsym}
\usepackage{amsfonts}
\usepackage{amsmath}
\usepackage{amstext}
\usepackage[alphabetic]{amsrefs}
\usepackage{accents}
\usepackage{cancel}
\usepackage{eucal}
\usepackage{amscd}
\usepackage{hyperref}
\usepackage{graphicx}
\usepackage{subcaption}
\usepackage{tikz}
\usepackage{mathrsfs}
\usepackage{enumitem}
\allowdisplaybreaks

\newtheorem{theorem}{Theorem}

\makeatletter
\newtheorem*{rep@theorem}{\rep@title}
\newcommand{\newreptheorem}[2]{%
\newenvironment{rep#1}[1]{%
 \def\rep@title{#2 \ref{##1}}%
 \begin{rep@theorem}}%
 {\end{rep@theorem}}}
\makeatother

\newreptheorem{theorem}{Theorem}

\newtheorem{conj}[theorem]{Conjecture}

\newtheorem{lemma}[theorem]{Lemma}

\theoremstyle{remark}
\newtheorem{remark}{Remark}

\newcommand{\dist}{\text{dist }}
\newcommand{\Q}{B}
\newcommand{\R}{\mathbb{R}}
\newcommand{\N}{\mathbb{N}}
\newcommand{\C}{\mathbb{C}}

\newcommand{\vn}[1]{\lVert#1\rVert}
\newcommand{\SL}{\mathcal{L}}

\newcommand{\cn}{\operatorname{cn}}
\newcommand{\sn}{\operatorname{sn}}
\newcommand{\dn}{\operatorname{dn}}
\newcommand{\AM}{\operatorname{AM}}
\newcommand{\rot}{\text{rot }}

\renewcommand{\k}{\ensuremath{\kappa}}

\newcommand{\E}{\mathcal{E}}

\newcommand{\myL}{\mathcal{L}}

\newcommand{\siE}{\hat{\mathcal{E}}}

\newcommand{\mykappa}{u}

\newcommand{\myu}{\rho}

\newcommand{\mya}{{b}}

\newcommand{\normal}{{\nu}}
\newcommand{\tangent}{{\tau}}

\renewcommand{\SS}{\mathcal{S}}
\newcommand{\SX}{\mathcal{X}}

\newcommand{\gamgrad}{\Gamma}

\title{On the basin of attraction for the free boundary free elastic flow}

\author[K.~Deckelnick]{Klaus Deckelnick}
\address{Institut f\"ur Analysis und Numerik,
Otto-von-Guericke-Universit\"at Magdeburg, Postfach 4120, 39016 Magdeburg, Germany}
\email{klaus.deckelnick@ovgu.de}

\author[H.-Ch.~Grunau]{Hans-Christoph Grunau}
\address{Institut f\"ur Analysis und Numerik,
Otto-von-Guericke-Universit\"at Magdeburg, Postfach 4120, 39016 Magdeburg, Germany}
\email{hans-christoph.grunau@ovgu.de}

\author[R.~N\"urnberg]{Robert N\"urnberg}
\address{Dipartimento di \textcolor{black}{Matematica}, Universit\`a di Trento,
38123 Trento, Italy}
\email{robert.nurnberg@unitn.it}

\author[G.~Wheeler]{Glen Wheeler}
\address{Institute for Mathematics and its Applications, University of
Wollongong, Northfields Avenue, Wollongong, NSW, 2522, Australia}
\email{glenw@uow.edu.au}

\author[V.-M.~Wheeler]{Valentina-Mira Wheeler}
\address{Institute for Mathematics and its Applications, University of
Wollongong, Northfields Avenue, Wollongong, NSW, 2522, Australia}
\email{vwheeler@uow.edu.au}

\date{\today}

\begin{document}

\begin{abstract}
The free boundary free elastic flow is the steepest descent gradient flow for the elastic energy of curves meeting parallel lines perpendicularly.
In this article we prove that the {straight} line has, measured in {Euler's} scale-invariant {bending} energy, a basin of attraction at least to the level $1.9615\, \pi$.
We show that {our} method of proof cannot be pushed to the {previously} conjectured level $2\pi$, and in addition present numerical evidence that this conjecture may in fact {be false}.
\end{abstract}

\maketitle

\section{Introduction}

To each smooth immersed curve $\gamma:[-1,1]\to\mathbb{R}^2$ we assign $\E[\gamma]$ and $\siE[\gamma]$, Euler's bending (or elastic) energy and Euler's scale-invariant bending energy respectively, defined as follows:
\[
\E[\gamma] := \frac12\int_{-1}^{1} \k^2\,ds\,,\quad\text{and}\quad
\siE[\gamma] := \myL[\gamma]\E[\gamma] = \frac12\int_{-1}^{1} \,ds\, \int_{-1}^{1} \k^2\,ds\,.
\]
Above, $\textcolor{black}{s(\myu) = \int_{-1}^{\myu} |\gamma'(r)|\,dr}$ is the arclength (with 
$ds = |\gamma'(\myu)|\,d\myu$ the arclength measure and $\partial_s = \frac{1}{|\gamma'(\myu)|}\, \partial_{\myu}$ the arclength derivative), $\myL[\gamma]$ is the length of $\gamma$, and $\k = \partial^2_s\gamma\cdot \normal$ is the scalar curvature.
We use $\normal$ and $\tangent$ for the normal and tangent vectors respectively.
We often use subscripts to denote differentiation.

Set $\eta_{\pm1}:\R\to\R^2$ to be parallel vertical lines with $\eta_{\pm1}(\myu) = (\pm1,\myu)$.\footnote{This choice of supporting lines is not essential; see Lemma \ref{lem:scaling}, {below}.}
The sets $\eta_{\pm1}(\R) = \{\eta_{\pm1}(\myu)\,:\,\myu\in\R\}$ are the admissible spatial location of the endpoints of $\gamma$.
\textcolor{black}{We impose further that the tangent vector of $\gamma$ at the endpoints is horizontal and that the derivative of the scalar curvature vanishes.}
Under these boundary conditions, we  consider the steepest descent $L^2(ds)$-gradient flow for $\E$.
Formally this evolution equation is posed on the space  
\[
\SX = \{\gamma:[-1,1]\to\R^2\,:\,
    \gamma\text{ smooth, immersed},\ 
    \gamma(\pm1) \in \eta_{\pm1}(\R),\ 
    \normal(\pm1)\cdot e_1 = 0,\ 
    \k_s(\pm1) = 0
    \}
\,.
\]
This is the \emph{free boundary free elastic flow}:
a one-parameter family $\gamma:[-1,1]\times[0,T)\to\R^2$ satisfying
\begin{equation}
\label{FEF}
\tag{FEF}
\begin{cases}
{\partial_t\gamma} = -\left( \k_{ss} + \frac12\k^3\right)\normal&\text{ in }(-1,1)\times(0,T)\,,\hspace{-1cm}
\\
\gamma(\cdot,t) \in \SX&\text{ for all }t\in[0,T)
\,.
\end{cases}
\end{equation}

\begin{figure}[t]
\centering
\begin{tikzpicture}[scale=0.8]

\begin{scope}[xshift=0cm]
  \draw (0,1.5) -- (0,5.5);
  \draw (4,1.5) -- (4,5.5);
  \node[above] at (0,5.5) {\small $\eta_{-1}$};
  \node[above] at (4,5.5) {\small $\eta_{+1}$};

  \draw[thick]
    (0,2.0)
      to[out=0,in=180]   (1.0,2.0)
      to[out=0,in=180]   (3.0,5.0)
      to[out=0,in=180]   (4.0,5.0);

  \node at (1.6,4.4) {$\gamma(\cdot,0)$};
\end{scope}

\draw[densely dotted,->] (5.25,3.5) -- (7.0,3.5);
  \node at (6.0,3.85) {$t\to\infty$};

\begin{scope}[xshift=8cm]
  \draw (0,1.5) -- (0,5.5);
  \draw (4,1.5) -- (4,5.5);
  \node[above] at (0,5.5) {\small $\eta_{-1}$};
  \node[above] at (4,5.5) {\small $\eta_{+1}$};

  \draw[thick] (0,3.5) -- (4,3.5);

  \node at (2.1,4.4) {$\gamma(\cdot,\infty)$};
\end{scope}

\end{tikzpicture}
\caption{A visualisation of the statement of Theorem \ref{TMmain}. Initial data $\gamma_0=\gamma(\cdot,0)$ \textcolor{black}{satisfying the hypotheses of Theorem~\ref{TMmain}} are driven by the flow to the horizontal line $\gamma(\cdot,\infty) := \lim_{t\to\infty}\gamma(\cdot,t)$.}
\end{figure}
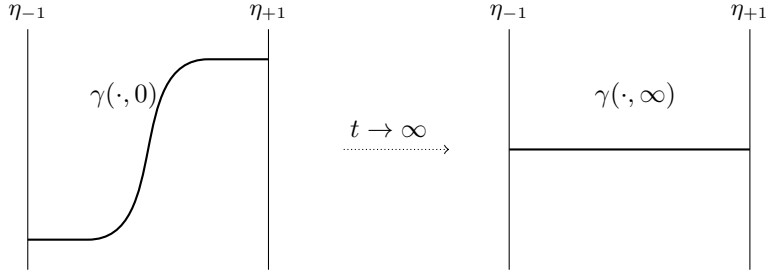

The main result of this paper is the following theorem.

\begin{theorem}
\label{TMmain}
The free boundary free elastic flow
$\gamma:[-1,1]\times[0,T)\to\R^2$ with initial data {$\gamma(\cdot,0)=\gamma_0\in\SX$} with turning number $0$, i.e.
\[
\int^1_{-1} \k \, ds=0,
\]
and satisfying 
\begin{equation}
\siE[\gamma_0] \le 1.9615\ \pi
\label{EQmtic}
\end{equation}
exists for all time, converging exponentially fast in the smooth topology as $t\to\infty$ to a straight line.
\end{theorem}

\textcolor{black}{The turning number is constant along any smooth solution of \eqref{FEF}, since it depends continuously on $t$ and, by the endpoint conditions, takes values in $\frac12\mathbb Z$. Hence the zero initial turning in Theorem~\ref{TMmain} is preserved for all $t\ge0$.}

\textcolor{black}{In \cite{WW24}, the conclusion of Theorem~\ref{TMmain} was established under the stronger condition $\siE[\gamma_0]\le\frac\pi2$. In view of the stationary solutions described below, it is natural to ask whether \eqref{EQmtic} can be replaced by $\siE[\gamma_0]<2\pi$.}
This is motivated by the fact that the space of stationary solutions $\SS$ consists precisely of:
\begin{enumerate}
\item the horizontal line with scale-invariant energy $\siE$ equal to zero;
\item $m$ half-periods of Euler's rectangular elastica (see Figure \ref{figrect}), with scale-invariant energy $\siE$ equal to $\textcolor{black}{2\pi m^2}$,
\end{enumerate}
and their images under vertical translation.
In terms of the functional $\siE$, the next threshold after zero is precisely $2\pi$.
In order to see this, it suffices to solve the Euler-Lagrange equation $\k_{ss} + \tfrac12\k^3=0$ for $\k$, and then determine which candidates satisfy the boundary conditions. 
This is classical, following Euler's famous investigation.
For a historical account, we refer the interested reader to \cite{Levien}.

Very recently, Dall’Acqua and Schlierf \cite{dallschlierf} studied the length-preserving elastic flow of curves with orthogonal free boundary on hypersurfaces in $\mathbb R^n$, proving short-time well-posedness, global existence and subconvergence to critical points, with stronger convergence results in certain flat planar cases. Their work treats a length-constrained flow and a general hypersurface free-boundary geometry.

The flow \eqref{FEF} is not a gradient flow for $\siE$.
The key difficulty is to establish monotonicity of $t\mapsto\siE[\gamma(\cdot,t)]$.
The time derivative of $\siE$ is
\begin{equation}
\frac{d}{dt}\siE[\gamma(\cdot,t)]
=  -\textcolor{black}{\E[\gamma(\cdot,t)]}\left(\int \k_s^2\,ds
- \frac12\int \k^4\,ds\right)
  -\textcolor{black}{\myL[\gamma(\cdot,t)]}\int \left(\k_{ss} + \frac12\k^3\right)^2\,ds 
  \,,
\label{EQsied}
\end{equation}
obtained by combining \eqref{EQenergy} and \eqref{EQlengthevo}, {below}.
As in \cite{WW24} we focus on the first term, specifically the quantity $\int \k_s^2\,ds
- \frac12\int \k^4\,ds$, which is the derivative of  $t\mapsto-\myL[\gamma(\cdot,t)]$.
A sufficient condition guaranteeing monotonicity of $\textcolor{black}{\siE[\gamma(\cdot,t)]}$ is that this term is non-negative, that is
\begin{equation}
\int \k^4\,ds \le 2\int \k_s^2\,ds
\,.
\label{EQsiesuff}
\end{equation}
Our goal is thus to discover conditions on $\siE[\gamma_0]$ that imply \eqref{EQsiesuff}.
(This will then be preserved by the flow due to the monotonicity so obtained.)
To this end we prove the following sharp interpolation inequality, which is of independent interest.

\begin{theorem}
\label{TMcritineqintro}
Let $L>0$. Then for any sufficiently smooth function $u:[0,L]\to \mathbb{R}$ satisfying $u'(0)=u'(L)=0$ and
\[
\int^L_0 u(s)\, ds=0
\]
one has
\begin{equation}
\label{eq:1.1}
\int^L_0 u^4(s)\, ds \le C_0\cdot L  \left( \int^L_0 u^2(s)\, ds \right)  
\left( \int^L_0 (u'(s))^2\, ds \right)
\end{equation}	
with $C_0=0.162278$. 
\end{theorem}

\begin{figure}[t]
\centering
\begin{subfigure}[t]{0.32\textwidth}
  \centering
  \includegraphics[width=\linewidth,trim={4cm 10cm 4cm 10cm},clip]{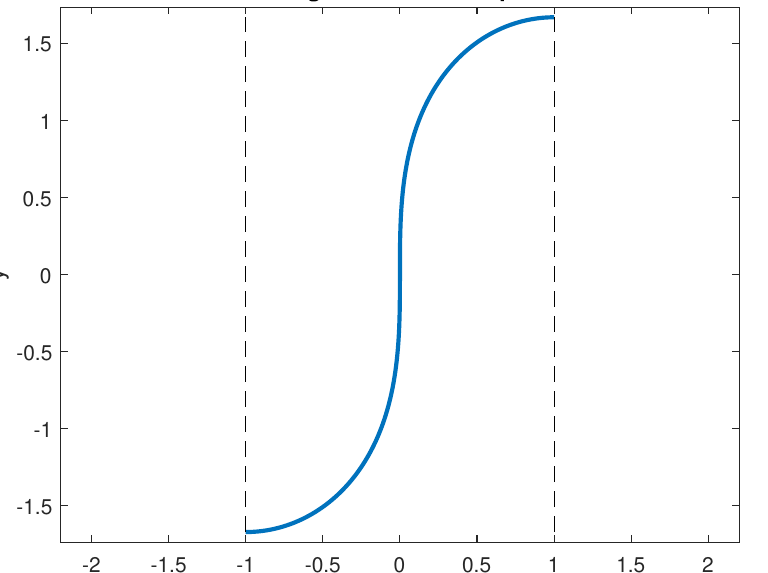}
  \subcaption{One half-period.}
\end{subfigure}\hfill
\begin{subfigure}[t]{0.32\textwidth}
  \centering
  \includegraphics[width=\linewidth,trim={4cm 10cm 4cm 10cm},clip]{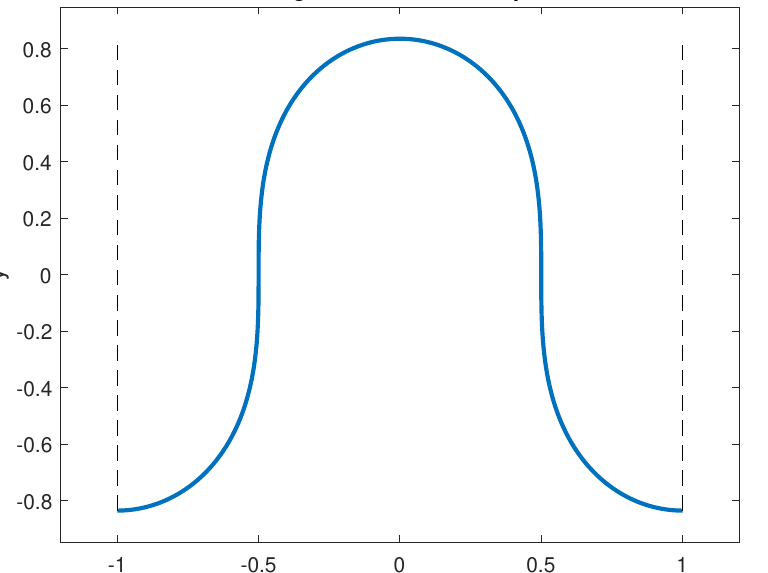}
  \subcaption{A full period.}
\end{subfigure}\hfill
\begin{subfigure}[t]{0.32\textwidth}
  \centering
  \includegraphics[width=\linewidth,trim={4cm 10cm 4cm 10cm},clip]{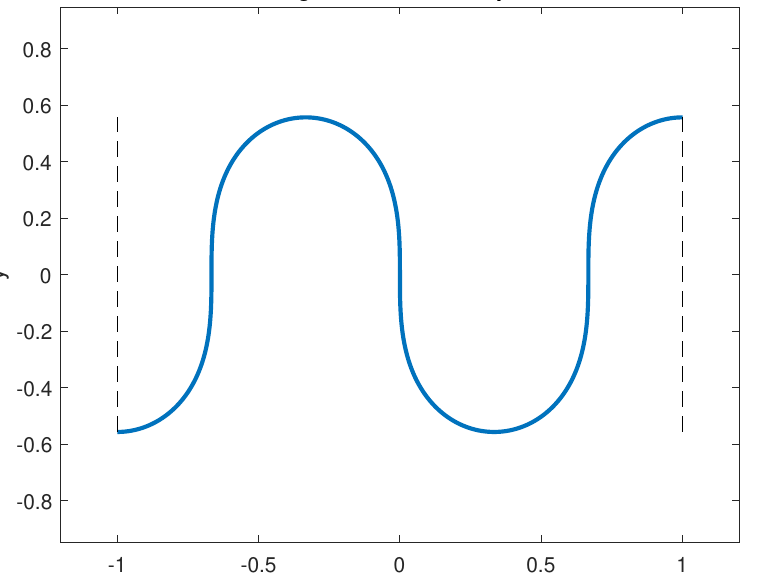}
  \subcaption{Three half-periods.}
\end{subfigure}

\caption{Euler's rectangular elastica in $\SX$ with one, two and three half-periods.
With modulus \(k=1/{\sqrt2}\), an arclength parametrisation $\gamma=(\gamma^1,\gamma^2):[0,L]\to\R^2$ is
\(
\gamma^1(s)=\frac{1}{\mya}\big(2E(\AM(\mya s,k),k)-\mya s\big)\,,
\gamma^2(s)=-\frac{\sqrt2}{\mya}\,\cn(\mya s,k),
\)
{where $\mya=m \Gamma(\tfrac34)^2/\sqrt{\pi}$ depends on the number of half-periods $m$, and where
each} curve has the same length \(L=\tfrac12\big(\Gamma(\tfrac14)/\Gamma(\tfrac34)\big)^2\approx 4.37688\)\textcolor{black}{.}
{See Appendix~\ref{AA} for the precise details.}
{For the scale-invariant energy it holds that $\siE[\gamma]=2\pi m^2$. In}
particular, the left image has {$\siE[\gamma]=2\pi$, and so} more energy than is allowed by the hypothesis \eqref{EQmtic} of Theorem \ref{TMmain}.
}
\label{figrect}
\end{figure}

The hypothesis \eqref{EQmtic} involves the mysterious number $1.9615$.
Theorem \ref{TMcritineqintro} reveals the meaning behind this figure.
It is chosen such that its product with $\pi$ is below $\frac{1}{C_0}$ (from \eqref{eq:1.1}).

The optimal value for $C_0$ in \eqref{eq:1.1} is smaller than $0.162278$ and given in \eqref{eq:2.3.9}, below. But as one can see there, this requires one to solve a complicated transcendental equation.
The solution is not available in closed form and a numerical approximation is required.
We take  $C_0 = 0.162278$ as a reliable upper bound for the optimal  constant in \eqref{eq:1.1}.
\textcolor{black}{For a flow satisfying the hypotheses of Theorem~\ref{TMmain}}, Theorem \ref{TMcritineqintro} applied with $u = \k$ yields \eqref{EQsiesuff}, and decay of $t\mapsto\siE[\gamma(\cdot,t)]$ follows.
This is the most important step in the proof of Theorem \ref{TMmain}.

The constant $C_0$ cannot be improved to $1/2\pi\simeq0.159155$ (if it were, \textcolor{black}{this would yield the threshold $2\pi$}).
Through careful analysis of the inequality \eqref{eq:1.1} and its maximising function, we are able to give explicitly an initial curve $\gamma_c$ such that
\begin{itemize}
\item \textcolor{black}{the scale-invariant energy satisfies $\siE[\gamma_c]<2\pi$, and}
\item the scale-invariant energy increases at least initially (the right-hand side of \eqref{EQsied} evaluated at $\gamma_c$ is positive).
\end{itemize}
This profile is visually indistinguishable from a half-period of Euler's rectangular elastica (which has scale-invariant energy equal to $2\pi$).
Numerical evidence indicates that the free boundary free elastic flow with initial data equal to $\gamma_c$ diverges as $t\to\infty$.
Since the profile in question is very close to the basin of attraction, {obtaining reliable numerical results is a challenging task}.
We refer the reader to Section \ref{sec:numerics} for further details.

The free boundary free elastic flow does not have any stable family of equilibria apart from the horizontal lines (see Appendix~\ref{AB}).
It typically inflates curves in its desire to minimise the elastic energy $\E$.
{On} rescaling the flow by setting $\hat\gamma(\cdot,t) = \frac{\gamma(\cdot,t)-\overline{\gamma}(t)}{\myL[\gamma(\cdot,t)]}$, where $\overline{\gamma}(t)$ is the barycentre of $\gamma(\cdot,t)$, {it is possible} to study the eventual qualitative shape of the flow.
{Observe that when} $\myL[\gamma(\cdot,t)]\to\infty$, this rescaling brings the {two supporting} lines together.
{This motivates the study of the following limiting case of the free boundary free elastic flow:
\begin{align} \label{eq:limitFEF}
& \text{\eqref{FEF} with $\SX$ replaced by} \nonumber \\ & \quad \{\gamma:[-1,1]\to\R^2\,:\,
    \gamma\text{ smooth, immersed},\ 
    \gamma(\pm1) \cdot e_1 = 0,\ 
    \normal(\pm1)\cdot e_1 = 0,\ 
    \k_s(\pm1) = 0
    \}\,.
\end{align}
}
There are two known families of expanders for the free elastic flow of closed curves, and these give two families of clear candidates for the asymptotic shape of {solutions to the flow \eqref{eq:limitFEF}:}
\begin{itemize}
\item {semicircles, circles and their multiple coverings}; and
\item {half-periods, full periods (and their multiple coverings) of} the Lemniscate of Bernoulli.
\end{itemize}
Note that both are self-similar expanders for the flow \eqref{eq:limitFEF}.
We make the following conjecture on the global behaviour of the free boundary free elastic flow.

\begin{conj}\label{conj:3}
Let $\gamma:[-1,1]\times[0,\infty)\to\R^2$ be a free boundary free elastic flow.
Set $\hat\gamma(\cdot,t) = \frac{\gamma(\cdot,t)-\overline{\gamma}(t)}{\myL[\gamma(\cdot,t)]}$.
Then $\hat\gamma(\cdot,t)$ converges as $t\to\infty$ exponentially fast in the smooth topology to a limit $\gamma_\infty$.
Furthermore, $\gamma_\infty$ is either
{
\begin{enumerate}
\item the horizontal line $\textcolor{black}{[-\tfrac12,\tfrac12]\times\{0\}}$; or
\item a rescaled Euler's rectangular elastica (with $m$ half-periods); or
\item a self-similar expander for the flow \eqref{eq:limitFEF}.
\end{enumerate}
}
\end{conj}
{%
We observe that the first case is guaranteed \textcolor{black}{under the hypotheses of our main theorem}.
We leave the analysis and the detailed numerical investigation of the above conjecture for future research. But we direct the reader to Section~\ref{sec:numerics}, where a half-period of the Lemniscate of Bernoulli is observed as the limiting shape $\gamma_\infty$ in several numerical experiments.}

The paper is organised as follows.
Section~\ref{Stwo} contains the basic properties of the flow, including local existence and evolution equations.
Section~\ref{Sthree} is concerned with the proof of the critical inequality Theorem \ref{TMcritineqintro}, and related issues such as its sharpness and the construction of $\gamma_c$.
In Section~\ref{Sconv} we conduct global analysis of the flow, including showing that long-time existence is generic and that an a-priori length bound is the main obstacle to convergence.
These are important ingredients for the proof of Theorem \ref{TMmain}, which is also in Section ~\ref{Sconv}.
In Section~\ref{sec:numerics} we state the numerical scheme for approximating the flow, and additionally present several visualisations of simulations of the flow with a variety of initial configurations.
Finally, we include an appendix, detailing Euler's rectangular elastica (Appendix~\ref{AA}) and the linearisation of the flow (Appendix~\ref{AB}).

\section{Analysis of the flow}
\label{Stwo}

\subsection{On the choice of supporting lines.}

It is not essential that the supporting lines are precisely $\eta_{\pm1}$; they may be any pair of parallel lines.
To see this, we present the following lemma, whose proof is a straightforward scaling argument.

\begin{lemma}
\label{lem:scaling}
Let $\ell_{-},\ell_{+}\subset\R^2$ be two parallel lines with
\(
\dist(\ell_{-},\ell_{+}) = {2\lambda}>0\,,
\)
perpendicular to the vector $e\in\R^2$, 
and consider the space
\[
\hat\SX = \{\gamma:[-1,1]\to\R^2\,:\,
    \gamma\text{ smooth, immersed},\ 
    \gamma(\pm1) \in \ell_{\pm},\ 
    \normal(\pm1)\cdot e = 0,\ 
    \k_s(\pm1) = 0
    \}
\,.
\]
Set $\Phi:\R^2\to\R^2$ to be the similarity transformation {$\Phi(x)=\frac1\lambda Qx + x_0$, where $Q\in\mathrm{SO}(2)$} and $x_0\in\R^2$,
determined by $\Phi(\ell_{\pm}) = \eta_{\pm1}(\R)$.

A one-parameter family $\gamma:[-1,1]\times[0,T)\to\R^2$ is a solution to \eqref{FEF}, {with $\SX$ replaced by $\hat\SX$}, if and only if the family {$\tilde\gamma:[-1,1]\times[0, T/\lambda^4)\to\R^2$, defined via $\tilde\gamma(\rho,t):=\Phi(\gamma(\rho,\lambda^4 t))$,} is a solution to \eqref{FEF}.
{Moreover, $\E[\gamma(\cdot,0)] = \frac1\lambda\textcolor{black}{\E[\tilde\gamma(\cdot,0)]}$ and $\siE[\gamma(\cdot,0)] = \siE[\tilde\gamma(\cdot,0)]$.}
\end{lemma}

\subsection{Local well-posedness.}

Local existence for the free boundary free elastic flow can be established via the following general procedure.
Let $\gamma_0\in\SX$ be arbitrary.
Now consider the $\delta$-tubular neighbourhood around $\gamma_0$, parameterised by coordinates $(\myu,\sigma)\mapsto\gamma_0(\myu)+\sigma \normal_0(\myu)$, where $\myu\in[-\myL_0/2,\myL_0/2]$ here is the arclength parameter on $\gamma_0$, $\normal_0$ is the normal to $\gamma_0$ and $\sigma\in(-\delta,\delta)$.
We may convert the free boundary free elastic flow to a scalar parabolic equation for a normal section $v:[-\myL_0/2,\myL_0/2]\times[0,T_\delta)\to\R$ for a short time $T_\delta = T_\delta(\gamma_0,\delta)$, via the requirement that $\gamma(\myu,t) = \gamma_0(\myu) + v(\myu,t)\normal_0(\myu)$.

\begin{lemma}
\label{LMnormalgraph}
Let $\gamma_0\in\SX$, and set ${\normal}_0$, $\k_0$, $\myL_0$ to be its {normal}, scalar curvature and length.
There exists a smooth function $\tilde F:\R^8\to\R$ with the following property.
Suppose there exists a solution $v:[-\myL_0/2,\myL_0/2]\times[0,T)\to\R$ to the PDE
\begin{equation}
\label{EQstepde}
\begin{cases}
    \partial_t v =
    -\frac{1}{((1-v\k_0)^2+v_\myu^2)^{2}}\,v_{\myu\myu\myu\myu}
  - \tilde F(\k_0,(\k_0)_\myu, 
  (\k_0)_{\myu\myu},
  (\k_0)_{\myu\myu\myu},
  v,
  v_\myu,
  v_{\myu\myu},
  v_{\myu\myu\myu})\,,
    \\
    v_{\myu}(\pm\myL_0/2,t) = v_{\myu\myu\myu}(\pm\myL_0/2,t) = 0\,,
    \\
    v(\myu,0) = 0\,,
\end{cases}
\end{equation}
for all $\myu\in(-\myL_0/2,\myL_0/2)$ and $t\in[0,T)$, where $T>0$ is a parameter.
Then $\gamma(\myu,t) := \gamma_0(\myu) + v(\myu,t)\normal_0(\myu)$ solves $\partial_t\gamma\cdot \normal = -(\kappa_{ss} + \frac12\kappa^3)$. After reparametrisation we obtain a free boundary free elastic flow with initial data given by $\gamma_0$.
\end{lemma}
\begin{proof}
Take $\myu$ to be the arclength parameter on $\gamma_0$ as above so that it is parameterised from $[-\myL_0/2,\myL_0/2]$.
The tangent and normal vectors along the family $\gamma(\cdot,t)$ are given by
\[
\tangent = \frac{(1-v\k_0)\tangent_0 + v_\myu \normal_0}{|(1-v\k_0)\tangent_0 + v_\myu \normal_0|}\,,\quad
\normal = \gamgrad^{-1}\rot ((1-v\k_0)\tangent_0 + v_\myu \normal_0)
= \gamgrad^{-1}((1-v\k_0)\normal_0 - v_\myu \tangent_0)
\,,
\]
where $\gamgrad := |(1-v\k_0)\tangent_0 + v_\myu \normal_0| = \sqrt{(1-v\k_0)^2+v_\myu^2}$, and $\rot (X,Y) = (-Y,X)$.
The curvature vector is
\begin{align*}
K &= \gamgrad^{-1}
\left(
\gamgrad^{-1}((1-v\k_0)\tangent_0 + v_\myu \normal_0)
\right)_\myu
\\&= \gamgrad^{-2}
\left(
(\k_0 - v\k_0^2 + v_{\myu\myu})\normal_0 - (2v_\myu\k_0+v(\k_0)_\myu) \tangent_0
\right)
+ (\gamgrad^{-1})_\myu \tangent
\,.
\end{align*}
So for the {scalar curvature} we have
\begin{align*}
\k &= K\cdot \normal
\\ &= \gamgrad^{-3}
\left(
(\k_0 - v\k_0^2 + v_{\myu\myu})\normal_0 - (2v_\myu\k_0+v(\k_0)_\myu) \tangent_0
\right)
\cdot((1-v\k_0)\normal_0 - v_\myu \tangent_0)
\\ &= \gamgrad^{-3}
\left(
(\k_0 - v\k_0^2 + v_{\myu\myu})(1-v\k_0)
+ (2v_\myu\k_0+v(\k_0)_\myu)v_\myu
\right)
\\ &= 
 \gamgrad^{-3}(1 - v\k_0)v_{\myu\myu}
+ \gamgrad^{-3}\left(
    2\k_0 v_\myu
    + v(\k_0)_\myu
\right)v_{\myu}
+ \gamgrad^{-3}\left(
     - 2\k_0^2 + v\k_0^3
\right) v
 + \gamgrad^{-3}\k_0\,,
\end{align*}
which we write as
\[
\k =  \gamgrad^{-3}(1 - v\k_0)v_{\myu\myu}
    + F_0(\k_0,(\k_0)_\myu,v,v_\myu)
\,,
\]
where $F_0$ is smooth (recall that we are working in a small tubular neighbourhood of $\gamma_0$ where the coordinate map is a diffeomorphism).
We further restrict the diameter of the tubular neighbourhood now as follows.
As $\gamma_0$ is smooth there exist $\delta, \varepsilon >0$ such that $1 - \delta\k_0 > \varepsilon > 0$ on $[-\myL_0/2,\myL_0/2]$.
Restrict the diameter of the domain (and thus the image $v([-\myL_0/2,\myL_0/2],t)$) to this $\delta$.
Note that this implies $\gamgrad = \sqrt{(1-v\k_0)^2 + v_\myu^2} < C(\k_0,v,v_\myu)$, so we obtain uniform strict positivity of the coefficient of the highest-order term on bounded subsets.

Returning to our calculation, we find
\[
\k_s = \gamgrad^{-1}\k_\myu
 = \gamgrad^{-4}(1 - v\k_0)v_{\myu\myu\myu}
    + F_1(\k_0,(\k_0)_\myu,(\k_0)_{\myu\myu},v,v_\myu,v_{\myu\myu})\,,
\]
where $F_1 =  v_{\myu\myu}\partial_\myu(\gamgrad^{-3}(1 - v\k_0))
+ \gamgrad^{-1}\partial_\myu(F_0(\k_0,(\k_0)_\myu,v,v_\myu))$ is smooth.

Now let us check that the boundary conditions given to $v$ imply that $\gamma$ satisfies the free boundary conditions.
First, as $\gamma_0\in\SX$, and $v$ is a normal displacement, the contact condition $\gamma(\pm\myL_0/2,t)\in\eta_{\pm1}(\R)$ is satisfied.
Second, since $v_\myu(\pm \myL_0/2,t) = 0$ (and using again $\gamma_0\in\SX$), we have
\begin{align*}
\normal(\pm\myL_0/2,t)\cdot e_1
&= \gamgrad^{-1}(\pm\myL_0/2,t)(\gamgrad^{-1}((1-v\k_0)\normal_0 - v_\myu \tangent_0))(\pm\myL_0/2,t)\cdot e_1
\\
&= \gamgrad^{-1}(\pm\myL_0/2,t)( - v_\myu(\pm\myL_0/2,t)) = 0
\end{align*}
as required.

The remaining boundary condition is $\k_s(\pm\myL_0/2,t)=0$.
To prepare, let us calculate
\begin{align*}
\partial_\myu\gamgrad(\pm\myL_0/2,t)
&= \tangent(\pm\myL_0/2,t)\cdot((-2v_\myu\k_0 - v(\k_0)_\myu)\tangent_0 + (\k_0-v\k_0^2+v_{\myu\myu})\normal_0)(\pm\myL_0/2,t)
= 0
\,.
\end{align*}
Above we used the fact that $v_\myu(\pm\myL_0/2,t) = 0$ and $\gamma_0\in\SX$.
Note also that $\partial_\myu((1 - v\k_0)v_{\myu\myu})(\pm\myL_0/2,t) = 0$, which uses additionally $v_{\myu\myu\myu}(\pm\myL_0/2,t) = 0$. 
We shall also need that
\[
 \partial_\myu\left(
    (2\k_0 v_\myu
    + v(\k_0)_\myu)v_{\myu}
\right)(\pm\myL_0/2,t) = 0\quad\text{ and }\quad
 \partial_\myu\left(
     (- 2\k_0^2 + v\k_0^3)v
\right)(\pm\myL_0/2,t) = 0\,.
\]
Using all of these, we calculate:
\begin{align*}
\k_s(\pm\myL_0/2,t)
 &= \gamgrad^{-1}\partial_\myu\Big(
 \gamgrad^{-3}(1 - v\k_0)v_{\myu\myu}
+ \gamgrad^{-3}\left(
    2\k_0 v_\myu
    + v(\k_0)_\myu
\right)v_{\myu}
\\ &\quad
+ \gamgrad^{-3}\left(
     - 2\k_0^2 + v\k_0^3
\right) v
 + \gamgrad^{-3}\k_0\Big)(\pm\myL_0/2,t)
\\
 &= \gamgrad^{-4}\partial_\myu\big(
  \k_0\big)(\pm\myL_0/2,t)
  = 0\,.
\end{align*}
Thus the family $\gamma(\cdot,t)$ is contained in $\SX$.

Computing further, we find
\[
\k_{ss} + \frac12\k^3
 = \gamgrad^{-5}(1 - v\k_0)v_{\myu\myu\myu\myu}
    + F_2(\k_0,(\k_0)_\myu, 
  (\k_0)_{\myu\myu},
  (\k_0)_{\myu\myu\myu},
  v,
  v_\myu,
  v_{\myu\myu},
  v_{\myu\myu\myu})\,,
\]
where again $F_2$ is smooth.
Finally, set
\begin{align*}
\tilde F(\k_0,(\k_0)_\myu, 
  (\k_0)_{\myu\myu},&
  (\k_0)_{\myu\myu\myu},
  v,
  v_\myu,
  v_{\myu\myu},
  v_{\myu\myu\myu}) \\&= \frac{\sqrt{(1-v\k_0)^2+v_\myu^2}}{1-v\k_0}
F_2(\k_0,(\k_0)_\myu, 
  (\k_0)_{\myu\myu},
  (\k_0)_{\myu\myu\myu},
  v,
  v_\myu,
  v_{\myu\myu},
  v_{\myu\myu\myu})
\,.
\end{align*}
Then, if $v$ satisfies the PDE \eqref{EQstepde}, we have
\[
\partial_t\gamma\cdot \normal
    = -\k_{ss} - \tfrac12\k^3.
\]
At this stage the velocity of $\gamma$ may still contain a tangential component.
Indeed, decomposing
\[
\partial_t\gamma = (\partial_t\gamma\cdot \normal)\normal + (\partial_t\gamma\cdot \tangent)\tangent
                 =: V_n\,\normal + {V_t}\,\tangent,
\]
the  free boundary free elastic flow prescribes only the normal speed 
\(V_n=-\k_{ss}-\tfrac12\k^3\),
while the tangential term~${V_t}$ corresponds to a reparametrisation of the curve.
Such tangential motions are admissible only if they vanish at the boundary (since otherwise the boundary points would move orthogonal to the supports).

To obtain a purely normal evolution while preserving the free boundary conditions,
we reparameterise the curve by solving for a family of diffeomorphisms
\[
\phi(\cdot,t):[-\myL_0/2,\myL_0/2]\to[-1,1],\qquad
\phi(\myu,0)=\frac{2}{\myL_0}\myu,\qquad
\phi(\pm\myL_0/2,t)=\pm 1,
\]
satisfying the transport equation
\begin{equation}\label{EQ:phi-transport}
\partial_t\phi(\myu,t)
   = -\,\frac{{V_t}(\phi(\myu,t),t)}{|\gamma_\myu(\phi(\myu,t),t)|}.
\end{equation}
Defining the reparameterised curve
\(\widetilde\gamma(\myu,t):=\gamma(\phi(\myu,t),t)\),
we compute
\[
{\partial_t\widetilde\gamma(\myu,t)
   = \partial_t\gamma(\phi(\myu,t),t)
     +\gamma_\myu(\phi(\myu,t),t)\,\partial_t\phi(\myu,t)
   = V_n(\phi(\myu,t),t)\,\normal(\phi(\myu,t),t),}
\]
so that $\partial_t\widetilde\gamma$ is purely normal.
Because $V_t/|\gamma_\myu|$ is continuous and bounded
and the initial and boundary data for~\eqref{EQ:phi-transport} fix the endpoints,
standard ODE theory gives a unique $\phi\in C^{1}$ in~$t$
with $\phi(\cdot,t)$ a diffeomorphism for each~$t$ small enough, because $\gamma_\myu(\myu,0)=1$.
Thus the reparameterised family $\widetilde\gamma$ satisfies
\[
\partial_t\widetilde\gamma\cdot \normal
   = -\k_{ss}-\tfrac12\k^3,
\]
with zero tangential velocity and fixed boundary.
Consequently $\widetilde\gamma$ (and hence $\gamma$ up to reparametrisation)
is a free boundary free elastic flow.
\end{proof}

From Lemma \ref{LMnormalgraph} and standard theory, local well-posedness follows.
We use the setting of maximal $L^p$-regularity, following \cite{PS16}.

\begin{remark}[Trace space for maximal $L^p$-regularity]
\label{rmk:trace-space}
Let $1<p<\infty$, $T>0$, $I=[-\myL_0/2,\myL_0/2]$, and set
\begin{align*}
E_p^1(0,T)&:=W^{1,p}(0,T;X_0)\cap L^p(0,T;X_1),\qquad
X_0:=L^p(I),\\
X_1&:=\{v\in W^{4,p}(I)\,:\,v'(\pm\myL_0/2)=v'''(\pm\myL_0/2)=0\}.
\end{align*}
For abstract maximal $L^p$-regularity problems, the initial trace space is
\[
\mathrm{tr}_{t=0}\,E_p^1(0,T)\;=\;(X_0,X_1)_{1-1/p,p},
\]
with continuous, surjective trace map and a bounded right inverse independent of $T$.
In our setting,
\[
(X_0,X_1)_{1-1/p,p}
\;=\;B^{\,4-4/p}_{p,p}(I)\cap\{v\,:\,v'(\pm\myL_0/2)=v'''(\pm\myL_0/2)=0\}.
\]
Since $4-4/p\notin\mathbb N$ for $p>5$, we have the identification
$B^{\,4-4/p}_{p,p}(I)=W^{4-4/p,p}(I)$ on the bounded interval~$I$,
and the boundary conditions are meaningful in the trace sense because
\[
4-\tfrac{4}{p} \;>\; 1+\tfrac{1}{p}\quad\text{and}\quad
4-\tfrac{4}{p} \;>\; 3+\tfrac{1}{p}
\qquad(\text{equivalently }p>5),
\]
which is exactly our threshold on~$p$.
Consequently,
\[
\mathrm{tr}_{t=0}\,E_p^1(0,T)
\;=\;W^{4-4/p,p}(I)\cap\{v\,:\,v'(\pm\myL_0/2)=v'''(\pm\myL_0/2)=0\},
\]
and our choice \(v(0)=0\in(X_0,X_1)_{1-1/p,p}\) is the natural (optimal) initial space for
Theorem~\ref{TMste}.  The same identification holds uniformly for frozen operators
$A(w)$ in a small neighbourhood, as used in the quasilinear fixed-point step (Step 4) of the proof of Theorem \ref{TMste}.
\end{remark}

\begin{theorem}
\label{TMste}
Let $\gamma_0\in\SX$ with curvature $\k_0$ and length $\myL_0$.
Fix $p>5$, set $I=[-\myL_0/2,\myL_0/2]$, and define
\[
X_0:=L^p(I),\qquad 
X_1:=\{v\in W^{4,p}(I)\,:\,v'(\pm\myL_0/2)=v'''(\pm\myL_0/2)=0\}.
\]
Write
\[
\gamgrad = \gamgrad(w) = \sqrt{(1-w\k_0)^2+w_\myu^2},\qquad
a=a(w)=\frac{1}{\gamgrad^4}=\frac{1}{\big((1-w\k_0)^2+w_\myu^2\big)^2}.
\]
Consider
\begin{equation}\label{E:PDE}
\begin{cases}
\partial_t v + A(v)v = G(v) & \text{in } I\times(0,T),\\[1mm]
v'(\pm\myL_0/2,t)=v'''(\pm\myL_0/2,t)=0&\text{on } (0,T),\\[1mm]
v(\myu,0)=0,
\end{cases}
\end{equation}
with
\[
A(w):=a(w)\,\partial_\myu^4,\qquad
G(w):=-\tilde F\big(\k_0,
(\k_0)_\myu,
(\k_0)_{\myu\myu},
(\k_0)_{\myu\myu\myu},
w,w_\myu,w_{\myu\myu},w_{\myu\myu\myu}\big),
\]
where $\tilde F$ is the smooth function from Lemma~\ref{LMnormalgraph}.
Then:

\begin{enumerate}[label=(\roman*)]
\item \textup{(Existence and uniqueness)}  There exists $T>0$ such that \eqref{E:PDE}
admits a unique solution
\[
v\in W^{1,p}(0,T;X_0)\cap L^p(0,T;X_1)\,.
\]

\item \textup{(Parabolic smoothing)} We have $v\in C^\infty(I\times(0,T])$.
Smoothness up to $t=0$ holds provided the full hierarchy of boundary
compatibility conditions generated from \eqref{E:PDE} 
and $v'(\pm\myL_0/2,t)=v'''(\pm\myL_0/2,t)=0$ 
is satisfied at $t=0$.

\item \textup{(Free elastic flow)} The curve
$\gamma(\myu,t):=\gamma_0(\myu)+v(\myu,t)\normal_0(\myu)$ defines, after reparametrisation, 
a free boundary free elastic flow on $[0,T)$.
\end{enumerate}
\end{theorem}

\begin{proof}
\emph{Step 1: Abstract setup and regularity thresholds.}
For $w\in W^{4-4/p,p}(I)$ we have
$w_{\myu\myu\myu}\in C^{0,\alpha}$ and $w_\myu\in L^\infty$ provided $p>5$, with
$\alpha=1-\tfrac{5}{p}\in(0,1)$.
This is the chief reason for the restriction $p>5$.

Hence $a(w,\cdot)\in C^{0,\alpha}\cap L^\infty$, and on a sufficiently small ball
$\mathcal U\subset W^{4-4/p,p}$ about $0$ we have uniform bounds
\[
0<a_* \le a(w,\myu)\le a^*<\infty,\qquad
\inf_{(\myu,w)\in I\times\mathcal U}\big(1-w\k_0(\myu)\big)\;>\;0,
\]
so $G:\mathcal U\to X_0$ is $C^\infty$ (the factor $(1-w\k_0)^{-1}$ is harmless on $\mathcal U$)
and $(i\xi)^4 a(w,\myu)$ is parameter-elliptic with any angle $\theta\in(0,\pi/2)$ (see \cite[Definition~6.1.1]{PS16}).

\emph{Step 2: Lopatinskii-Shapiro (LS) condition.}
The boundary operators are
$B_1v=v'$ and $B_2v=v'''$.
In one dimension the LS condition reduces to showing that the only
decaying solution of the model half‑line problem
\[
\lambda v + a_0\,\partial_y^{4}v = 0,\qquad y>0,\qquad
v'(0)=v'''(0)=0,
\tag{LS}
\]
where $\lambda\in\C$ with $\Re\lambda\ge0$ (and  $a_0>0$ is the frozen principal coefficient) is the trivial one.

Introduce
\(
\mu:=-\dfrac{\lambda}{a_0}
\)
and write the characteristic equation as
\begin{equation}\label{E:char}
\rho^{4}=\mu .
\end{equation}
If $\lambda=0$, the fundamental system consists of $\{1,t,t^2,t^3\}$, so the only decaying solution is $u\equiv0$, the trivial solution, which establishes LS in that case.

So, let $\rho_0\ne0$ be any fourth root of $\mu\ne0$ and set 
\(
\rho_k:=\rho_0\,e^{i\pi k/2}\;(k=0,1,2,3).
\)
Then $\bigl\{e^{-\rho_k y}\bigr\}_{k=0}^{3}$ is a fundamental system for the ODE, so
\begin{equation}\label{E:gensol}
v(y)=\sum_{k=0}^3 c_k\,e^{-\rho_k y},\qquad c_k\in\C .
\end{equation}
The assumption that the solution decays manifests as follows.
Since $\mu$ has nonpositive real part  (in particular, $\mu\notin(0,\infty)$),
none of its fourth roots lie on the imaginary axis.
Consequently, among the four roots $\{\rho_k\}_{k=0}^3$ there are exactly two with strictly positive real part
and exactly two with strictly negative real part.
Denote the two roots with $\Re\rho>0$ by $\rho_{+}$ and $\rho_{-}$.

A solution \eqref{E:gensol} decays as $y\to\infty$ iff $c_k=0$ for every root with $\Re\rho_k\le0$.
Hence the general decaying solution is given by
\begin{equation}\label{E:decay}
v(y)=c_{+}\,e^{-\rho_{+}y}+c_{-}\,e^{-\rho_{-}y}.
\end{equation}
We apply the boundary conditions $v'(0)=v'''(0)=0$ to \eqref{E:decay}:
\[
\begin{pmatrix}
\rho_{+} & \rho_{-}\\
\rho_{+}^{3} & \rho_{-}^{3}
\end{pmatrix}
\begin{pmatrix}
c_{+}\\ c_{-}
\end{pmatrix}=0.
\]
The determinant is
\(
\rho_{+}\rho_{-}(\rho_{-}^{2}-\rho_{+}^{2})\neq0
\)
because $\rho_{+}\neq\rho_{-}$ and neither is zero.  
Thus $c_{+}=c_{-}=0$ and $u\equiv0$.
Hence $(A(w),B_1,B_2)$ satisfies the hypotheses of
\cite[Theorem 6.3.2]{PS16}.

\emph{Step 3: Maximal $L^p$-regularity.}
We apply \cite[Theorem 6.3.2]{PS16} to obtain that
\begin{align*}
(\partial_t+A(w),\,B_1,\,B_2):
&W^{1,p}(0,T;X_0)\cap L^p(0,T;X_1)
\\&\longrightarrow
L^p(0,T;X_0)\times
L^p\bigl(0,T;W^{3-1/p,p}(\partial I)\bigr)\times
L^p\bigl(0,T;W^{1-1/p,p}(\partial I)\bigr)
\end{align*}
is an isomorphism for every frozen $w\in\mathcal U$, with a time $T>0$
independent of $w$.
Since $\partial I$ consists of two points, the boundary trace spaces above are canonically
isomorphic to $(\C^2)^2$, so the codomain can equivalently be written as
$L^p(0,T;X_0\times(\C^2)^2)$; that is, 
\[
(\partial_t+ A(w), B_1,B_2)
: W^{1,p}(0,T;X_0)\cap L^p(0,T;X_1)
  \longrightarrow L^p(0,T;X_0\times
                 (\C^2)^2)
                 \,.
\]
This setup is compatible with the nonlinearity $G$. 
Set 
\[
E_T:=W^{1,p}(0,T;X_0)\cap L^p(0,T;X_1),
\quad X_0=L^p(I),\ X_1=\{v\in W^{4,p}(I)\,:\,v'(\pm\tfrac{\myL_0}{2})=v'''(\pm\tfrac{\myL_0}{2})=0\}.
\]
By maximal $L^p$-regularity,
\[
E_T \hookrightarrow {C^0}\bigl([0,T];(X_0,X_1)_{1-1/p,p}\bigr)
= {C^0}\bigl([0,T];W^{4-4/p,p}(I)\bigr).
\]
Since $\tilde F\in C^\infty(\R^5)$ and $p>5$ (so $W^{4-4/p,p}(I)\hookrightarrow C^{3,\alpha}(I)$),
the associated map
\[
G:W^{4-4/p,p}(I)\to X_0,\qquad
G(w)=-\tilde F\big(\k_0,
(\k_0)_\myu,
(\k_0)_{\myu\myu},
(\k_0)_{\myu\myu\myu},
w,w_\myu,w_{\myu\myu},w_{\myu\myu\myu}\big),
\]
is $C^\infty$ and locally Lipschitz on bounded sets. Hence, for every $v\in E_T$,
$G(v(\cdot,t))\in X_0$ for a.e.\ $t$, and $G(v)\in L^p(0,T;X_0)$. Moreover, for each $R>0$
there exists $C_R$ such that
\[
\|G(v)-G(w)\|_{L^p(0,T;X_0)}
\ \le\ C_R\,\|v-w\|_{L^p(0,T;X_1)}
\qquad\text{whenever }\ \|v\|_{E_T},\|w\|_{E_T}\le R.
\]
Thus the right-hand side in \eqref{E:PDE} lies in the data space of the isomorphism above.
This underlies the fixed point  argument used in \cite[Theorem 5.1.1]{PS16}, which is itself used in Step 4 below.

\emph{Step 4: Quasilinear fixed point.}
The mappings 
\[
w\mapsto A(w)\in\mathcal L(X_1,X_0)\quad\text{ and }\quad w\mapsto G(w)\in X_0
\]
are $C^\infty$.
With maximal $L^p$-regularity at hand and initial datum $v(0)=0\in (X_0,X_1)_{1-1/p,p}=W^{4-4/p,p}(I)$
satisfying $B_1v(0)=B_2v(0)=0$, \cite[Theorem 5.1.1]{PS16} yields a unique solution
$v\in W^{1,p}(0,T;X_0)\cap L^p(0,T;X_1)$ of \eqref{E:PDE}.
Parabolic bootstrapping then implies (ii), with smoothness up to $t=0$ under the full
compatibility hierarchy; otherwise smoothing holds for $t>0$.
This establishes (i) and (ii).

\emph{Step 5: Normal‑graph reconstruction.}
The solution $v$ is initially zero and so, taking $T>0$ smaller if needed, $v$ will satisfy $|v\k_0| < 1$ for all $t\in[0,T)$.
Lemma \ref{LMnormalgraph} then applies, yielding that the map $w\mapsto \gamma_0+w\normal_0$ produces a solution to the free boundary free elastic flow.
Thus the theorem is proved.
\end{proof}

\begin{remark}
\label{RMreg}
Following the strategy in \cite{PS16}, if $\gamma_0$ were less regular, we could use a smooth curve $\tilde\gamma$ that is arbitrarily close to $\gamma_0$ in an appropriate norm to write the scalar parabolic PDE.
The initial data would be small but not zero.
Then, the regularity of $\gamma_0$ may be dramatically relaxed.
We have not pursued this here.
\end{remark}

Now we obtain local well-posedness for the flow.
Note that the statement below is only existence and uniqueness, however, smoothing and continuous dependence on initial data also clearly hold.

\begin{theorem}
\label{TMste2}
For every $\gamma_0\in\SX$ there exists a maximal time
\[
T_{\max}=T_{\max}(\gamma_0)\in(0,\infty]
\]
and a unique family of curves
\[
\gamma\;\in\;C^\infty\bigl([-1,1]\times(0,T_{\max})\bigr)\cap
{C^0}\bigl([-1,1]\times[0,T_{\max})\bigr)
\]
such that {$\gamma:[-1,1]\times[0,T_{\max})\to\R^2$ is a free boundary free elastic flow with
initial data $\gamma(\cdot,0)=\gamma_0$.}
Moreover, if $T_{\max}<\infty$ then
\begin{equation}
\label{EQcontcrit}
\limsup_{t\uparrow T_{\max}}
\Bigl(||\gamma||_{L^\infty}+\|\,\partial_s^p\k\|_{L^\infty}\Bigr)
      =\infty\,,
\end{equation}
for some $p\in\N_0$.
\end{theorem}
\begin{proof}
Theorem~\ref{TMste} gives a time
$T_1>0$ and a solution
$\gamma:[-1,1]\times[0,T_1)\to\R^2$ to the
 free boundary free elastic flow.
Let $\widetilde\gamma$ be any other, distinct, free elastic flow with the same initial data.  
Because $\gamma_0$ is the common reference curve and both families stay in the same normal tube of radius $\delta$ (from step 5 of the proof of Theorem \ref{TMste}), each can be written as a normal
graph \(\gamma_0+v\normal_0\) and \(\gamma_0+\widetilde v\normal_0\).
Both $v$ and $\widetilde v$ solve the same quasilinear
parabolic problem with the same initial datum~$0$; hence
$v=\widetilde v$ by uniqueness in Theorem~\ref{TMste}.
Consequently $\gamma=\widetilde\gamma$ as long as both exist, establishing uniqueness.

Assume the solution has been constructed on $[0,T_1)$
and satisfies
\(
\|v(\cdot,t)\|_{L^\infty}\le\delta/2
\)
for all $t<T_1$.
Letting $\gamma_1:=\gamma(\,\cdot,T_1/2)$
play the role of a new initial curve, we invoke
Theorem~\ref{TMste} once more to obtain a fresh existence time
interval $[T_1/2,T_1/2+T_2)$ on which the solution
remains a normal graph over $\gamma_{1}$.
Repeating this procedure as much as possible yields a maximal time of existence $T_{\max}\in(0,\infty]$.

If $T_{\max}<\infty$ then the sum of the $T_i$ must be bounded.
The two ways this can happen are
\begin{enumerate}[label=(\roman*)]
\item \emph{tubular neighbourhood width goes to zero:}          $\|\k(\cdot,t)\|_{L^\infty}\to\infty$;
\item \emph{vanishing at infinity:}
$\|\gamma(\cdot,t)\|_{L^\infty}\to\infty$;
\item \emph{loss of regularity:} $\|\partial_s^p\k(\cdot,t)\|_{L^\infty}\to\infty$ for some $p\in\N_0$.
\end{enumerate}
Observing that the criterion \eqref{EQcontcrit} contains all three of these possibilities finishes the proof.
\end{proof}

\begin{remark}
The continuation criterion can be substantially strengthened to
\[
T_{\max}<\infty
\quad\Longrightarrow\quad\limsup_{t\uparrow T_{\max}}
\Bigl(||\gamma||_{L^\infty}+\|\k\|_{L^\infty}\Bigr)
      =\infty\,,
\]
by following the regularity improvement outlined in Remark \ref{RMreg}.
There is possibly scope to go beyond this, however it would require a more subtle approach.
We will not need any improvements to the straightforward condition \eqref{EQcontcrit} here.
\end{remark}

\subsection{Evolution equations}
\label{Sevols}

The fundamental evolution equations are as follows.
Their proof can mostly be found in \cite[Lemmata 2.1 and 2.2]{WW24}, so we omit it.

\begin{lemma}
Let $\gamma:[-1,1]\times[0,T)\to\R^2$ be a free boundary free elastic flow.
Then
\begin{align*}
\partial_t \log |\gamma_\myu|
 &= \k\left(\k_{ss}+\frac12\k^3\right)\,,
\\
[\partial_t,\partial_s]
 &= \partial_t\partial_s - \partial_s\partial_t = -\k\left(\k_{ss}+\frac12\k^3\right)\partial_s\,,
\\
\partial_t\tangent & 
 = \left(-\k_{sss} - \frac32\k_s\k^2\right)\normal\,,
\\
\partial_t\normal &= \left(\k_{sss} + \frac32\k_s\k^2\right)\tangent\,,
\\
\partial_t\k &= -\k_{ssss} - \frac52\k_{ss}\k^2
- 3\k_{s}^2\k
- \frac12\k^5\,,
\\
\partial_t\k_s &= 
- \k_{s^5} 
- \frac52\k_{sss}\k^2
- 12\k_{ss}\k_s\k
- 3\k_{s}^3
- 3\k_s\k^4\,,
\end{align*}
and, {more generally for $\ell\in\N$},
\begin{align*}
\partial_t\k_{s^\ell}
 &= -\k_{s^{(\ell+4)}}
 + \sum_{q+r+u=\ell}\left(c^1_{qru}
    \k_{s^{(q+2)}}
    \k_{s^r}
    \k_{s^u}
    + c^2_{qru}
    \k_{s^{(q+1)}}
    \k_{s^{(r+1)}}
    \k_{s^u}\right)
\\
 &\qquad\qquad+ \sum_{q+r+u+v+w=\ell}c_{qruvw}
    \k_{s^q}
    \k_{s^r}
    \k_{s^u}
    \k_{s^v}
    \k_{s^w}
    \,,
\end{align*}
where $q,r,u,v,w\in\N_0$, {and where} $c^i_{qru}$ and $c_{qruvw}$ are suitable constants.
\label{LMevos}
\end{lemma}

From Lemma \ref{LMevos} we find, along a free boundary free elastic flow, 
\begin{align}
\notag\frac{d}{dt}\E[\gamma(\cdot,t)]
 &= \frac12\int \left(2\k\left(-\k_{ssss} - \frac52\k_{ss}\k^2
- 3\k_{s}^2\k
- \frac12\k^5\right) + \k^3\left(\k_{ss}+\frac12\k^3\right)\right)\,ds
\\
\notag
 &= -\int \left( \k_{ss}^2 + 2\k_{ss}\k^3
+ 3\k_{s}^2\k^2
+ \frac14\k^6
\right)\,ds
\\
\notag
 &= -\int \left(\k_{ss}^2 + \k_{ss}\k^3
+ \frac14\k^6
\right)\,ds
\notag
\\
 &= -\int \left(\k_{ss} + \frac12\k^3\right)^2\,ds
 = - ||\partial_t\gamma||_{L^2(ds)}^2,
\label{EQenergy}
\end{align}
where we used that $\gamma(\cdot,t)\in\SX$ and the identity (for $\gamma\in\SX$)
\[
\int \k_s^2\k^2\,ds
 = -\frac13\int \k_{ss}\k^3\,ds
 \,.
\]
The calculation \eqref{EQenergy} shows that the free boundary free elastic flow is the steepest descent $L^2(ds)$-gradient flow for $\E$.

We shall need the evolution of the norm of $\gamma$ itself in $L^2(ds)$.

\begin{lemma}
\label{LMgamevo}
Let $\gamma:[-1,1]\times[0,T)\to\R^2$ be a free boundary free elastic flow.
Then
\begin{align*}
\frac{d}{dt}\int |\gamma|^2\,ds
&= - {2\int \left(\k_{ss} + \frac12\k^3\right)\, \gamma\cdot \normal}\,ds
\\&\quad
+ \frac12\int |\gamma|^2\k^4\,ds
- \int \k_{s}^2|\gamma|^2\,ds
- 2\int \k_{s}\k \,{\gamma\cdot \tangent}\,ds
\,.
\end{align*}
\end{lemma}
\begin{proof}
The flow equation \eqref{FEF} and Lemma \ref{LMevos} imply
\begin{align*}
\frac{d}{dt}\int |\gamma|^2\,ds
&= \int {\gamma \cdot \left(-2\left(\k_{ss} + \frac12\k^3\right)\normal + \k\left(\k_{ss} + \frac12\k^3\right)\gamma\right)}\,ds
\\&= -{2\int \left(\k_{ss} + \frac12\k^3\right)\,\gamma\cdot \normal}\,ds
+ \frac12\int |\gamma|^2\k^4\,ds
+ \int \k_{ss}|\gamma|^2\k\,ds
\,.
\end{align*}
Now we integrate by parts, using $\gamma(\cdot,t)\in\SX$,  to obtain the desired result.
\end{proof}

In the following lemma we record the evolution of the length $\myL$ and the scale-invariant energy $\siE$.

\begin{lemma}
\label{LMlengthevo}
Let $\gamma:[-1,1]\times[0,T)\to\R^2$ be a free boundary free elastic flow.
Then
\begin{equation}
\label{EQlengthevo}
\frac{d}{dt}\myL[\gamma(\cdot,t)]
 = \int \k\left(\k_{ss}+\frac12\k^3\right)\,ds
 = -\int \k_s^2\,ds
    + \frac12\int \k^4\,ds
\end{equation}
and
\[
\frac{d}{dt}\siE[\gamma(\cdot,t)]
=  -\textcolor{black}{\E[\gamma(\cdot,t)]}\left(\int \k_s^2\,ds
- \frac12\int \k^4\,ds\right)
  -\textcolor{black}{\myL[\gamma(\cdot,t)]}\int \left(\k_{ss} + \frac12\k^3\right)^2\,ds 
  \,.
\]
\end{lemma}
\begin{proof}
Lemma~\ref{LMevos} together with integration by parts, using also that $\gamma(\cdot,t)\in\SX$, gives
\[
\frac{d}{dt}\myL[\gamma(\cdot,t)]
 = \int \k\left(\k_{ss}+\frac12\k^3\right)\,ds
 = -\int \k_s^2\,ds
    +\frac12\int \k^4\,ds
    \,.
\]
Combining the above with \eqref{EQenergy} yields the evolution of $\siE[\gamma(\cdot,t)]$.
\end{proof}

Regularity can be obtained by studying the evolution of integrals of derivatives of curvature.

\begin{lemma}
Let $\gamma:[-1,1]\times[0,T)\rightarrow\R^2$ be a free boundary free elastic flow.
Then
\begin{equation}
\label{EQevolksint}
\frac{d}{dt} \int \k_s^2\,ds
= 
- 2\int
    \k_{sss}^2
\,ds
+ 5\int 
  \k_{ss}^2\k^2  
\,ds
-\frac53\int
    \k_s^4
\,ds
-\frac{11}2\int 
    \k_s^2\k^4
\,ds
\,,
\end{equation}
and in general, {for $\ell\in\N$},
\begin{align*}
	\frac{d}{dt}\int \k_{s^\ell}^2\,ds
	&=
	- 2\int \k_{s^{(\ell+2)}}^2\,ds
	+ \sum_{q+r+u=\ell} \int \k_{s^\ell}\left(c^1_{qru}
    \k_{s^{(q+2)}}
    \k_{s^r}
    \k_{s^u}
    + c^2_{qru}
    \k_{s^{(q+1)}}
    \k_{s^{(r+1)}}
    \k_{s^u}\right) \,ds
	\\&\qquad
 	+ \sum_{q+r+u+v+w=\ell} c_{qruvw} \int \k_{s^\ell}\k_{s^q}\k_{s^r}\k_{s^u}\k_{s^v}\k_{s^w}\,ds
	\,,
\end{align*}
where $q,r,u,v,w\in\N_0$, {and where} $c^i_{qru}$ and $c_{qruvw}$ are suitable constants.
\label{LMhigherevos}
\end{lemma}
\begin{proof}
Lemma \ref{LMevos} and then integration by parts (with $\gamma(\cdot,t)\in\SX$ and \cite[Lemma 2.6]{WW24}) implies
\begin{align*}
\frac{d}{dt} \int \k_s^2\,ds
&= 
\int \k_s\left(-2\k_{s^5} 
- 5\k_{sss}\k^2
- 24\k_{ss}\k_s\k
- 6\k_{s}^3
- 6\k_s\k^4
\right)\,ds
\\&\quad
 + \int \k_s^2\k\left(\k_{ss}+\frac12\k^3\right)\,ds
\\
&= 
- 2\int
    \k_{sss}^2
\,ds
- 5\int 
  \k_{sss}\k_s\k^2  
\,ds
- 23\int
    \k_{ss}\k_s^2\k
\,ds
-6\int
    \k_s^4
\,ds
-\frac{11}2\int 
    \k_s^2\k^4
\,ds
\\
&= 
- 2\int
    \k_{sss}^2
\,ds
+ 5\int 
  \k_{ss}^2\k^2  
\,ds
- 13\int
    \k_{ss}\k_s^2\k
\,ds
-6\int
    \k_s^4
\,ds
-\frac{11}2\int 
    \k_s^2\k^4
\,ds
\\
&= 
- 2\int
    \k_{sss}^2
\,ds
+ 5\int 
  \k_{ss}^2\k^2  
\,ds
-\frac53\int
    \k_s^4
\,ds
-\frac{11}2\int 
    \k_s^2\k^4
\,ds
\,.
\end{align*}
For the general case we again use Lemma \ref{LMevos} for the evolution of {$\k_{s^\ell}$} and $\gamma(\cdot,t)\in\SX$, \cite[Lemma 2.6]{WW24} to eliminate boundary terms, finding
\begin{align*}
	\frac{d}{dt}\int \k_{s^\ell}^2\,ds
	&=
	\frac12\int \left(4\k_{s^\ell}\partial_t \k_{s^\ell} + \k_{s^\ell}^2(2\k\k_{ss} + \k^4)\right)\,ds
	\\&=
	\frac12\int \k_{s^\ell}^2(2\k\k_{ss} + \k^4)\,ds
	- 2\int \k_{s^\ell}\k_{s^{(\ell+4)}} \,ds
	\\&\qquad
	+ \sum_{q+r+u=\ell} \int \k_{s^\ell}\left(c^1_{qru}
    \k_{s^{(q+2)}}
    \k_{s^r}
    \k_{s^u}
    + c^2_{qru}
    \k_{s^{(q+1)}}
    \k_{s^{(r+1)}}
    \k_{s^u}\right) \,ds
	\\&\qquad
 	+ \sum_{q+r+u+v+w=\ell} c_{qruvw} \int \k_{s^\ell}\k_{s^q}\k_{s^r}\k_{s^u}\k_{s^v}\k_{s^w}\,ds
	\\&=
	- 2\int \k_{s^{(\ell+2)}}^2\,ds
	+ \sum_{q+r+u=\ell} \int \k_{s^\ell}\left(c^1_{qru}
    \k_{s^{(q+2)}}
    \k_{s^r}
    \k_{s^u}
    + c^2_{qru}
    \k_{s^{(q+1)}}
    \k_{s^{(r+1)}}
    \k_{s^u}\right) \,ds
	\\&\qquad
 	+ \sum_{q+r+u+v+w=\ell} c_{qruvw} \int \k_{s^\ell}\k_{s^q}\k_{s^r}\k_{s^u}\k_{s^v}\k_{s^w}\,ds
	\,,
\end{align*}
as required.
\end{proof}

\subsection{Estimates}
\label{Sests}
The proof of global existence in Theorem~\ref{TMlte}, {below}, relies on estimates for $|\gamma|$ and {$\k_{s^\ell}$} that are uniform on compact time intervals.

In this section we articulate the evolution equations from Section \ref{Sevols} to these a-priori estimates.

We first present an estimate for length.

\begin{lemma}
\label{LMlengthest}
Let $\gamma:[-1,1]\times[0,T)\rightarrow\R^2$ be a free boundary free elastic flow.
Then
\[
\myL[\gamma(\cdot,t)]
\le \myL[\gamma(\cdot,0)] 
+ {\frac{1+2t}2\E[\gamma(\cdot,0)]}
\,.
\]
\end{lemma}
\begin{proof}
The first equality in \eqref{EQlengthevo} and the Cauchy inequality {imply}
\[
\frac{d}{dt}\myL[\gamma(\cdot,t)]
 \le \frac12\int \k^2\,ds
 + \frac12\int \left(\k_{ss}+\frac12\k^3\right)^2\,ds\,.
\]
The gradient flow property and integration thus gives the estimate
\[
\myL[\gamma(\cdot,t)]
\le \myL[\gamma(\cdot,0)] + t\E[\gamma(\cdot,0)]
+ \frac12\int_0^t \int\left(\k_{ss}+\frac12\k^3\right)^2\,ds\,d\hat t
\,.
\]
Using again the gradient flow property, in particular the energy identity
\begin{equation}
\label{EQenrgyid}
\E[\gamma(\cdot,t)] + \int_0^t \int\left(\k_{ss}+\frac12\k^3\right)^2\,ds\,d\hat t = \E[\gamma(\cdot,0)],
\end{equation}
gives the estimate
\[
\myL[\gamma(\cdot,t)]
\le \myL[\gamma(\cdot,0)] 
+ t\E[\gamma(\cdot,0)]
+ \frac12\E[\gamma(\cdot,0)]
\,,
\]
which finishes the proof.
\end{proof}

\begin{lemma}
\label{LMksest}
Let $\gamma:[-1,1]\times[0,T)\rightarrow\R^2$ be a free boundary free elastic flow.
Then
\[
\int \k_s^2\,ds \le \int \k_s^2\,ds\bigg|_{t=0}
+ {Ct\E^7[\gamma(\cdot,0)]},
\]
where $C$ is an absolute constant.
\end{lemma}
\begin{proof}
Two applications of a general interpolation inequality (due to  Dziuk-Kuwert-Sch\"atzle \cite[Proposition 2.5]{DKS02} for closed curves and \cite[Lemma 4.1, Appendix C]{dall2014} for curve segments) {yield}
\[
\int \k_{ss}^2\k^2\,ds
 \le \delta \int \k_{sss}^2\,ds + C(\delta)\E^7[\gamma(\cdot,t)]
 \,.
\]
To see this, apply the H\"older inequality with $p=q=2$ and then interpolate, noting that our boundary conditions imply that we may estimate $||\k_{s^m}||^2_{L^2} \le ||\k||_{L^2}||\k_{s^{M}}||_{L^2}$ for any $0<m<M$, allowing us to keep only the lowest and highest seminorms.
The constant $C(\delta)$ depends only on $\delta$.
A uniform lower bound for length is needed, which we trivially have in our setting ($\myL[\gamma(\cdot,t)] \ge 2$).
Choosing $\delta = 1/5$ and using this to estimate the sole term with an unfavourable sign in \eqref{EQevolksint} yields
\[
\frac{d}{dt} \int \k_s^2\,ds
\le 
 C\E^7[\gamma(\cdot,0)]
\,,
\]
which upon integration gives the estimate.
\end{proof}

The estimate for $|\gamma|$ is below.

\begin{lemma}
\label{LMgamest}
Let $\gamma:[-1,1]\times[0,T)\rightarrow\R^2$ be a free boundary free elastic flow.
Then
\[
||\gamma||_{L^\infty}^2 \le
C\left(1+e^{Ct^5}\right)
\,,
\]
where $C$ is a constant depending only on $\|\gamma(\cdot,0)\|_{L^2}$, $\myL[\gamma(\cdot,0)]$, $\E[\gamma(\cdot,0)]$ and $\|\k_s(\cdot,0)\|_{L^2}$.
\end{lemma}
\begin{proof}
Estimating the evolution of $\gamma$ in $L^2(ds)$, Lemma~\ref{LMgamevo}, we find
\begin{align}
\frac{d}{dt}\int |\gamma|^2\,ds
&\le \frac12\int |\gamma|^2\,ds
    + 2\int \left(\k_{ss} + \frac12\k^3\right)^2\,ds
 + \int \k^2\,ds
 + \frac12\vn{{\k}}^4_{L^\infty}\int |\gamma|^2\,ds
\,.
\label{EQgamevo2}
\end{align}
We now require a curvature bound.
First, estimate (here {$\overline{\k}$} is the average of $\k$: $\overline{\k} = \frac1{\myL[\gamma(\cdot,t)]}\int \k\,ds$, and \textcolor{black}{$\omega:=\frac1{2\pi}\int\k\,ds$ is the turning number})
\[
||\k||_{L^\infty}
\le ||\k-\overline{\k}||_{L^\infty} + \textcolor{black}{|\overline{\k}|}
\le ||\k_s||_{{L^1}} + \textcolor{black}{\frac{2|\omega|\pi}{\myL[\gamma(\cdot,t)]}}
\,.
\]
As $\gamma(\cdot,t)\in\SX$, $\myL[\gamma(\cdot,t)]\ge2$.
Using this and Lemmata \ref{LMlengthest}, \ref{LMksest} yields
\[
||\k||_{L^\infty}
\le C\left(\myL[\gamma(\cdot,0)] 
+ {\frac{1+2t}2\E[\gamma(\cdot,0)]}\right)^\frac12
\left(\|\k_s(\cdot,0)\|_{L^2}^2 + C\,t\,\E^7[\gamma(\cdot,0)]\right)^\frac12
+ \textcolor{black}{|\omega|\pi}
\,,
\]
where $C$ is an absolute constant.

Applying the curvature estimate in \eqref{EQgamevo2} yields
\[
\frac{d}{dt}\int |\gamma|^2\,ds
\le (C_1+C_2t^4)\int |\gamma|^2\,ds
 + 2\E[\gamma(\cdot,0)]
+ 2\int \left(\k_{ss} + \frac12\k^3\right)^2\,ds,
\]
where $C_1$ and $C_2$ are  constants depending on $\myL[\gamma(\cdot,0)]$, $\E[\gamma(\cdot,0)]$, and $||\k_s||_{{L^2}}(0)$.
Integration and the energy identity yield
\[
\int |\gamma|^2\,ds
\le \int_0^t(C_1+C_2(\hat t)^4)\int |\gamma|^2\,ds\,d\hat t
 + \|\gamma(\cdot,0)\|_{L^2}^2
 + {2(2t+1)\E[\gamma(\cdot,0)]},
\]
from which the estimate
\begin{equation}
\int |\gamma|^2\,ds
\le (\|\gamma(\cdot,0)\|_{L^2}^2+2\E[\gamma(\cdot,0)])(2t+1)\exp\left(
    C_1t + \frac15C_2t^5
    \right)
\label{EQposvec}
\end{equation}
follows (due to the Gr\"onwall inequality).

To obtain the required $L^\infty$ control on $|\gamma|$ we apply a version of the Poincar\'e inequality (see \cite[Corollary 2.9]{WW24} for details) to
the components $\gamma^1$, $\gamma^2$ of $\gamma$.
We calculate
\begin{align*}
\vn{\gamma}_{L^\infty}
 &= \vn{\gamma^1e_1 + \gamma^2e_2}_{L^\infty}
\le
  \vn{\gamma^1}_{L^\infty} + \vn{\gamma^2}_{L^\infty}
\\&\le
\vn{\gamma^1 - \overline{\gamma^1} + \overline{\gamma^1}}_{L^\infty}
+ \vn{\gamma^2 - \overline{\gamma^2} + \overline{\gamma^2}}_{L^\infty}
\\&\le
\vn{\gamma^1 - \overline{\gamma^1}}_{L^\infty}
 + \vn{\overline{\gamma^1}}_{L^\infty}
+ \vn{\gamma^2 - \overline{\gamma^2}}_{L^\infty}
 + \vn{\overline{\gamma^2}}_{L^\infty}
\\&\le
	\bigg(
		\frac{2\myL[\gamma(\cdot,t)]}{\pi}\int |\tangent^1|^2 \,ds
	\bigg)^{\frac12}
	+ \bigg(
		\frac{2\myL[\gamma(\cdot,t)]}{\pi}\int |\tangent^2|^2 \,ds
	\bigg)^{\frac12}
+ \frac1{\myL[\gamma(\cdot,t)]}\int \left( |\gamma^1| + |\gamma^2|\right)\,ds
\\&\le
	2\myL[\gamma(\cdot,t)]
	\bigg(
		\frac{2}{\pi}
	\bigg)^{\frac12}
+ \frac{\sqrt 2}{\myL[\gamma(\cdot,t)]} \int  |\gamma|\,ds\,,\qquad\qquad\text{ since $|a + b| \le \sqrt2\sqrt{a^2 + b^2}$,}
\\&\le
	2\myL[\gamma(\cdot,t)]
	\bigg(
		\frac{2}{\pi}
	\bigg)^{\frac12}
+ \vn{\gamma}_{{L^2}}
\,.
\end{align*}
The claimed estimate follows now from this combined with  \eqref{EQposvec} and the length estimate Lemma~\ref{LMlengthest}.
\end{proof}

The higher derivative estimates are as follows.

\begin{lemma}
\label{LMkshighest}
Let $\gamma:[-1,1]\times[0,T)\to\R^2$ be a free boundary free elastic flow.
For each $\ell\in\N_0$ there exists $c_\ell=c_\ell\!\bigl(T,\myL[\gamma(\cdot,0)],\E[\gamma(\cdot,0)]\bigr)$ such that
\[
\int \k_{s^\ell}^2\,ds \;\le\; c_\ell\qquad\text{on }[0,T).
\]
\end{lemma}
\begin{proof}
Set $E_\ell(t):=\tfrac12\int \k_{s^\ell}^2\,ds$. 
Recalling Lemma~\ref{LMhigherevos} we note 
\begin{align}
\label{EQ:El-evo}
\frac{d}{dt}E_\ell(t)\;+\;\int \k_{s^{\ell+2}}^2\,ds
\;&\le\; \sum_{\substack{q+r+u=\ell}}
\!\int \!\!\k_{s^\ell}\Big(c^1_{qru}\,\k_{s^{q+2}}\k_{s^r}\k_{s^u}
+c^2_{qru}\,\k_{s^{q+1}}\k_{s^{r+1}}\k_{s^u}\Big)\,ds
\\&\qquad+\!\!\!\sum_{\substack{q+r+u+v+w=\ell}}\!\!\! c_{qruvw}\!\int\!\k_{s^\ell}\prod_{j\in\{q,r,u,v,w\}}\!\k_{s^j}\,ds .
\notag
\end{align}
We require uniform estimates for the reaction terms on compact time intervals, and will do so using a standard interpolation technique pioneered by Dziuk-Kuwert-Sch\"atzle \cite{DKS02}.

As the approach is well-known,  let us give only some representative estimates.
Consider the term with $(q,r,u)=(\ell,0,0)$ in the first sum:
\[
\int \big|\k_{s^\ell}\,\k_{s^{\ell+2}}\,\k^2\big|\,ds.
\]
By Cauchy-Schwarz,
\[
\int |\k_{s^{\ell+2}}|\;|\k_{s^\ell}\k^2|\,ds
\;\le\;\|\k_{s^{\ell+2}}\|_{{L^2}}\;\|\k_{s^\ell}\k^2\|_{{L^2}}.
\]
Choose exponents so that $\frac12=\frac1{p_0}+\frac2{p_1}$ with $(p_0,p_1)=(2,\infty)$, and interpolate with $m=\ell+2$:
\begin{align*}
\|\k_{s^\ell}\|_{L^{p_0}}
&\le C\,\|\k_{s^{\ell+2}}\|_{{L^2}}^{\theta_0}\,\|\k\|_{{L^2}}^{1-\theta_0},
\quad
\theta_0=\frac{\ell+\frac12-\frac1{p_0}}{\ell+2}
=\frac{\ell}{\ell+2},\\
\|\k\|_{L^{p_1}}
&\le C\,\|\k_{s^{\ell+2}}\|_{{L^2}}^{\theta_1}\,\|\k\|_{{L^2}}^{1-\theta_1},
\quad
\theta_1=\frac{\frac12-\frac1{p_1}}{\ell+2}=\frac{1}{2(\ell+2)}.
\end{align*}
Hence
\[
\|\k_{s^\ell}\k^2\|_{{L^2}}
\;\le\; C\,
\|\k_{s^{\ell+2}}\|_{{L^2}}^{\,\theta_0+2\theta_1}\;
\|\k\|_{{L^2}}^{\,1-\theta_0+2(1-\theta_1)}
= C\,
\|\k_{s^{\ell+2}}\|_{{L^2}}^{\,\frac{\ell+1}{\ell+2}}\;
\|\k\|_{{L^2}}^{\,\frac{2\ell+5}{\ell+2}},
\]
and therefore
\begin{equation}\label{EQ:danger-exponents}
\int \big|\k_{s^\ell}\,\k_{s^{\ell+2}}\,\k^2\big|\,ds
\;\le\; C\,
\|\k_{s^{\ell+2}}\|_{{L^2}}^{\,1+\frac{\ell+1}{\ell+2}}\;
\|\k\|_{{L^2}}^{\,\frac{2\ell+5}{\ell+2}}
\;=\;
C\,\|\k_{s^{\ell+2}}\|_{{L^2}}^{\,\frac{2\ell+3}{\ell+2}}\,
\|\k\|_{{L^2}}^{\,\frac{2\ell+5}{\ell+2}}.
\end{equation}
Note the exponent $\frac{2\ell+3}{\ell+2}=2-\frac{1}{\ell+2}<2$, so Young’s inequality gives, for any $\varepsilon>0$,
\[
\int \big|\k_{s^\ell}\,\k_{s^{\ell+2}}\,\k^2\big|\,ds
\;\le\; \varepsilon\,\|\k_{s^{\ell+2}}\|_{{L^2}}^2
\;+\; C(\varepsilon)\,\bigl(1+\|\k\|_{{L^2}}^{\,2\ell+5}\bigr).
\]

\smallskip
\emph{All other cubic terms.}
A generic cubic integrand from the first line of \eqref{EQ:El-evo} has the form
\(
\k_{s^\ell}\,\k_{s^{q+2}}\,\k_{s^r}\,\k_{s^u}
\)
or
\(
\k_{s^\ell}\,\k_{s^{q+1}}\,\k_{s^{r+1}}\,\k_{s^u}
\)
with $q+r+u=\ell$. Writing it as 
\(
\|\k_{s^{\ell+2}}\|_{{L^2}}\|\cdot\|_{{L^2}}
\)
after one integration by parts if needed, and estimating the $L^2$-factor by H\"older and the same interpolation inequality with $m=\ell+2$, one obtains the bound
\[
\int |\text{cubic term}|\,ds
\;\le\;
\varepsilon\,\|\k_{s^{\ell+2}}\|_{{L^2}}^2
\;+\;C(\varepsilon)\,\bigl(1+\|\k\|_{{L^2}}^{\,2\ell+5}\bigr),
\]
with the \emph{largest} power of $\|\k_{s^{\ell+2}}\|_{{L^2}}$ achieved by the term treated in \eqref{EQ:danger-exponents}.

\smallskip
\emph{Quintic terms.}
For
\(
\int \k_{s^\ell}\k_{s^q}\k_{s^r}\k_{s^u}\k_{s^v}\k_{s^w}\,ds
\)
with $q+r+u+v+w=\ell$, we again use H\"older/Gagliardo-Nirenberg with $m=\ell+2$; having more factors of $\k$ only \emph{reduces} the exponent of $\|\k_{s^{\ell+2}}\|_{{L^2}}$ below $\frac{2\ell+3}{\ell+2}$, so the same Young absorption applies and we obtain the identical right-hand side as above.

\smallskip
Combining all contributions in \eqref{EQ:El-evo} yields, for suitable $\varepsilon>0$,
\[
\frac{d}{dt}E_\ell(t) + (1-\varepsilon)\int \k_{s^{\ell+2}}^2\,ds
\;\le\; C_1\Bigl(1+\|\k\|_{{L^2}}^{\,2\ell+5}\Bigr),
\]
where $C_1$ depends on $T$, $\myL[\gamma(\cdot,0)]$ and $\E[\gamma(\cdot,0)]$ (the latter bounds $\|\k\|_{{L^2}}$).
By the (free-boundary) Poincar\'e inequality and Lemma~\ref{LMlengthest} (length control),
\[
\int \k_{s^{\ell+2}}^2\,ds \;\ge\; c(t)\int \k_{s^\ell}^2\,ds,
\qquad c(t)\ge c_*>0 \text{ on }[0,T).
\]
Hence
\begin{equation}\label{EQksesteqinpf}
\frac{d}{dt}E_\ell(t) + c_* E_\ell(t)\;\le\; C_2,
\end{equation}
and Gr\"onwall’s inequality gives the claimed uniform bound for $\int \k_{s^\ell}^2\,ds$ on $[0,T)$.
\end{proof}

\subsection{Global existence}
The main goal of this section is to establish global existence for generic initial data, that is, the following theorem.

\begin{theorem}
Let $\gamma:[-1,1]\times[0,T)\to\R^2$ be a free boundary free elastic flow with $T$ taken to be maximal.
Then $T=\infty$.
\label{TMlte}
\end{theorem}
\begin{proof}
With the estimates of the previous subsection in hand the argument is standard.
Suppose on the contrary that the maximal existence time is finite and equal to $T$.
Then, on $[0,T)$, the finiteness of $T$ implies all the estimates in Section \ref{Sests} are uniform.
This is in contradiction with \eqref{EQcontcrit} from Theorem \ref{TMste2}.
\end{proof}

\section{{The crucial} interpolation inequality}
\label{Sthree}

{In this section we prove Theorem~\ref{TMcritineqintro} and discuss possible refinements. We  further propose  possibly critical initial data for which  Section~\ref{sec:numerics} indicates that Theorem~\ref{TMmain} is presumably optimal. }

{
Let $L>0$. We recall that Theorem~\ref{TMcritineqintro} claims that for any sufficiently smooth function $u:[0,L]\to \mathbb{R}$ satisfying $u'(0)=u'(L)=0$ and
$
\int^L_0 u(s)\, ds=0
$
one has {that \eqref{eq:1.1} holds, i.e.}
\begin{equation*}
\int^L_0 u(s)^4\, ds \le {C_0} \cdot L \left( \int^L_0 u(s)^2\, ds \right)
\left( \int^L_0 u'(s)^2\, ds \right)
\end{equation*}	
with $C_0=0.162278$.}

\begin{remark}\label{rem:1}
	The constant $C_0=0.162278$ is calculated numerically and is not sharp.	
	However, the optimal constant is certainly larger than $0.162277>
	\frac{1}{2\pi}=0.159154\ldots$. That means that our method of proof will certainly fail when {investigating} whether the threshold in \eqref{EQmtic} in Theorem~\ref{TMmain}  could possibly be increased to $2\pi$.
	
	The optimal function is a Jacobian
	$\operatorname{cn}$-function, evaluated at an optimal parameter, which is close to $k=0.8803$. The optimal $C_0$ will be close to $0.1622778337$, see \eqref{eq:2.3.9} {below}.
\end{remark}

\begin{remark}\label{rem:2}
Remarks~\ref{rem:3} and \ref{rem:4} below show that there are initial data $\gamma_0$ with  $\siE[\gamma_0] < 2 \pi$, where {$\siE[\gamma (\cdot,t)]$, for the free boundary \textcolor{black}{free} elastic flow with $\gamma(\cdot,0)=\gamma_0$,}
is --at least initially-- strictly increasing. The long time behaviour of these initial data is numerically studied in Section~\ref{sec:numerics}.
\end{remark}

The proof of Theorem~\ref{TMcritineqintro} is performed by constrained maximisation. This results in studying differential equations of curvature type which can be explicitly solved by Jacobian elliptic functions. 

\subsection{Basics on Jacobian elliptic functions}
\label{Sbasics}

For the reader's convenience and in order to fix notation, we collect
some basic properties and definitions of Jacobian elliptic functions, one may see e.g. \cite[Sect. 4]{Mandel2}. One should have in mind that the notations are not uniform in the literature.

In what follows, let $k\in (0,1)$ be arbitrary but fixed.

\begin{align}
\mathbb{R} \ni \varphi \mapsto F(\varphi,k)&=\int^\varphi_0 \frac{1}{\sqrt{1-k^2\sin^2(\psi)}}\, d \psi,
\quad K(k)=F\left(\frac{\pi}{2},k\right),
\label{eq:2.1.01}\\
\mathbb{R} \ni \varphi \mapsto E(\varphi,k)&=\int^\varphi_0 \sqrt{1-k^2\sin^2(\psi)}\, d \psi,
\quad E(k)=E\left(\frac{\pi}{2},k\right) ,
\label{eq:2.1.02}\\
\mathbb{R}\ni s\mapsto \AM (s,k)& \quad \text{is the inverse function of}\quad 
\varphi \mapsto F(\varphi,k),\label{eq:2.1.03}\\
\mathbb{R}\ni s\mapsto \sn (s,k)&=\sin(\AM (s,k)),\qquad
\mathbb{R}\ni s\mapsto \cn (s,k)=\cos(\AM (s,k)),
\label{eq:2.1.04}\\
\mathbb{R}\ni s\mapsto \dn (s,k)&=\sqrt{1-k^2\sn^2(s,k)},
\label{eq:2.1.05}\\
\cn''(s,k) &:= \frac{\partial^2}{\partial s^2}\cn(s,k)=(2k^2 - 1)\cn(s,k) - 2k^2 \cn^3(s,k) ,\\ 
\sn''(s,k) &:=\frac{\partial^2}{\partial s^2}\sn(s,k) = - (1+k^2 )\sn(s,k)+2k^2 \sn^3(s,k) ,
\label{eq:2.1.06}\\
\frac{d}{dk}K(k)&= \frac{1}{k(1-k^2)}\left( E(k)-(1-k^2)K(k)\right)>0,
\label{eq:2.1.07}\\
\frac{d}{dk}E(k)&= \frac{1}{k}\left( E(k)-K(k)\right)<0.
\label{eq:2.1.08}
\end{align}
We calculate the following integrals for later use:
\begin{align}
\int^{K(k)}_0 \sn^2(s,k)\, ds &= \frac{1}{k^2}\left( K(k)-E(k)\right),\label{eq:2.1.1}\\
\int^{K(k)}_0 \cn^2(s,k)\, ds &= K(k)- \frac{1}{k^2}\left( K(k)-E(k)\right),\label{eq:2.1.2}\\
\int^{K(k)}_0 \sn^4(s,k)\, ds &=\frac{1}{k^4}\left(\frac13 k^2 K(k) -\frac23 k^2 E(k) +\frac23 K(k)-\frac23 E(k) \right), 
\label{eq:2.1.3}\\
\int^{K(k)}_0 \cn^4(s,k)\, ds &=\frac{1}{k^4}\left(k^4 K(k)-\frac53 k^2 K(k) +\frac43 k^2 E(k) +\frac23 K(k) -\frac23 E(k)\right).
\label{eq:2.1.4}
\end{align}

\subsection{Maximisers under constraints}

Thanks to scaling arguments it suffices to prove Theorem~\ref{TMcritineqintro} for $L=1$, and one may even restrict oneself to $\| u \|_{L^2(0,1)}=1$. The goal is to calculate
$$
\mathscr{C}_0:= \sup_{u\in \mathscr K} \frac{\| u \|^4_{L^4(0,1)}}{\| u_s \|^2_{L^2(0,1)}}
$$
{with}
$$
\mathscr K=
\{u:[0,1]\to \mathbb{R}\,{:}\, \int^1_0 u(s)\, ds=0, u'(0)=u'(1)=0,\| u \|^2_{L^2(0,1)}=1 \}. 
$$
However, one still has the $L^4$-norm as a free parameter. For this reason we introduce a second constraint. For arbitrary $D\in (0,\infty)$ we introduce
$$
\mathscr{K}_D:=\{u\in \mathscr K\,{:}\,\quad \| u \|^4_{L^4(0,1)}=D\}.
$$
Then, by standard variational techniques and regularity results  one finds maximisers 
$$
\exists u_D\in \mathscr{K}_D: \quad \frac{ D }{\| u_{D,s} \|^2_{L^2(0,1)}}
=\max_{u\in \mathscr{K}_D} \frac{ D }{\| u_{s} \|^2_{L^2(0,1)}}=:\mathscr{C}_{0,D}.
$$
Observe that 
\begin{equation}\label{eq:max_under_constraint}
    \mathscr{C}_0= \sup_{D\in (0,\infty)}\mathscr{C}_{0,D}.
\end{equation}
For suitable Lagrange parameters $\lambda,\mu \in \mathbb{R} $,
which depend on the variable $D\in (0,\infty)$, these are solutions of the following problem
\begin{equation}\label{eq:2.2.1}
-u_{D,ss}= \frac12 \lambda u_D^3 + \mu u_D\quad \text{in}\quad [0,1],
\quad u_{D,s}'(0)=u_{D,s}'(1)=0,\quad \int^1_0 u_{D}(s)\, ds
=0.
\end{equation}
Using $u_D$ as a testing function one finds the relation
$$
0<\int^1_0 u_{D,s}^2(s)\, ds=\frac12 \lambda D + \mu.
$$
This means that we need not consider the case where $\lambda \le 0$ and $\mu\le 0$.

Instead of working with the parameters $\lambda$ and $\mu$, we normalise them
to $0,+1,-1$ by scaling $s$ and $u$ and work on intervals with variable length $\tilde L$. So, we maximise \eqref{eq:2.3.1} below with respect to $\tilde L$ or equivalently with respect to the modulus of the elliptic functions instead of maximising \eqref{eq:max_under_constraint} with respect to $D$. See \eqref{eq:2.3.2} below.

Although the claim  \eqref{eq:1.1} looks like a general interpolation inequality, analysing its potential maximisers leads to studying the following curvature equations, see e.g.\ \cite{BryantGriffiths_1986,LangerSinger_1984a,Mandel1,Mandel2}.
We have to consider  five cases:
\begin{align}
& \mykappa_{ss}+\frac12 \mykappa^3-\mykappa=0, \label{2.2.01}\tag{\text{a}}\\
& \mykappa_{ss}+\frac12 \mykappa^3=0, \label{2.2.02}\tag{\text{b}}\\
& \mykappa_{ss}+\frac12 \mykappa^3+\mykappa=0, \label{2.2.03}\tag{\text{c}}\\
& \mykappa_{ss} +\mykappa=0, \label{2.2.04}\tag{\text{d}}\\
& \mykappa_{ss}-\frac12 \mykappa^3+\mykappa=0. \label{2.2.05}\tag{\text{e}}
\end{align}

\subsection{Reducing the problem to discussing a particular real function}
\label{sect:2.3}

According to \cite{BryantGriffiths_1986,LangerSinger_1984a,Mandel1,Mandel2}
and after some elementary calculations we have to consider the following solutions:
\begin{align}
 \mykappa_k(s) =\frac{2k}{\sqrt{2k^2-1}}\cn\left( \frac{s}{\sqrt{2k^2-1}},k\right),   \quad
s\in [0,\tilde L(k)], \quad \tilde L(k)=2 \sqrt{2k^2-1} K(k),&\tag{\text{a}}\\ 
\quad k\in (1/\sqrt2,1),&\nonumber\\
  \mykappa_{1/\sqrt2}(s) =\sqrt2\cn\left( s,1/\sqrt2\right),   \quad
s\in [0,\tilde L(1/\sqrt2)], \quad \tilde L(1/\sqrt2)=2  K(1/\sqrt2),& 
 \tag{\text{b}}\\
   \mykappa_k(s) =\frac{2k}{\sqrt{1-2k^2}}\cn\left( \frac{s}{\sqrt{1-2k^2}},k\right),   \quad
s\in [0,\tilde L(k)], \quad \tilde L(k)=2 \sqrt{1-2k^2} K(k),&\tag{\text{c}}\\ \quad k\in (0,1/\sqrt2),&\nonumber
\\
 \mykappa_0(s)=\cos(s),  \quad
s\in [0,\tilde L(0)], \quad \tilde L(0)=\pi,& \tag{\text{d}}\\
 \mykappa_k(s) =\frac{2k}{\sqrt{k^2+1}}\sn\left( \frac{s}{\sqrt{k^2+1}},-k\right),   \quad
s\in [-\tilde L(k)/2,\tilde L(k)/2], & \tag{\text{e}}\\
 \tilde L(k)=2 \sqrt{k^2+1} K(-k),\quad k\in (-1,0).&\nonumber 
\end{align}
The choice of the intervals of definition is made such that:
\begin{itemize}
	\item The functions have mean value $0$.
	\item They obey homogeneous Neumann boundary conditions.
	\item Their minimal length in order to satisfy the above conditions is
	a {half-period}. When working on $m$-th multiples of {half-periods}, the relevant quantity we want to maximise  decreases by the factor $1/m^2$.
\end{itemize}
The quantity we have to consider (taking into account the antisymmetry of $\mykappa_k$) is
\begin{equation}\label{eq:2.3.1}
(-1,1)\ni k\mapsto Q(k):=
\frac{\int^{\tilde L(k)/2}_0 \mykappa_k(s)^4\, ds }{ 2 \tilde L(k) \cdot  \left( \int^{\tilde L(k)/2}_0 \mykappa_k (s)^2\, ds \right)\cdot  
\left( \int^{\tilde L(k)/2}_0 \mykappa_{k,s}(s)^2\, ds \right)}.
\end{equation}
Theorem~\ref{TMcritineqintro} is then proved by calculating
\begin{equation}\label{eq:2.3.2}
\mathscr{C}_0=\sup_{k\in (-1,1)} Q(k)
\end{equation}
or, more precisely, by calculating a reliable upper bound for it.

Now we calculate the quantities of {interest} in the five cases as 
mentioned above.

\bigskip
\noindent 
\underline{Case (a): $k\in (1/\sqrt2,1)$.} We recall that here, we have to consider
$$
\mykappa_k(s) =\frac{2k}{\sqrt{2k^2-1}}\cn\left( \frac{s}{\sqrt{2k^2-1}},k\right),   \quad
s\in [0,\tilde L(k)], \quad \tilde L(k)=2 \sqrt{2k^2-1} K(k).
$$
Making use of \eqref{eq:2.1.2} and of \eqref{eq:2.1.4} we find:
\begin{align*}
\int^{\tilde L(k)/2}_0 \mykappa_k (s)^2\, ds &=\frac{4k^2}{{2k^2-1}}\int^{\sqrt{2k^2-1} K(k) }_0 
\cn^2\left( \frac{s}{\sqrt{2k^2-1}},k\right)\, ds\\
&=\frac{4k^2}{\sqrt{2k^2-1}}\int^{ K(k) }_0 
\cn^2\left( s,k\right)\, ds
=\frac{4}{\sqrt{2k^2-1}}\left( k^2 K(k) -K(k)+E(k)\right),
\\
\int^{\tilde L(k)/2}_0 \mykappa_k (s)^4\, ds &=\frac{16k^4}{{(2k^2-1)^2}}\int^{\sqrt{2k^2-1} K(k) }_0 
\cn^4\left( \frac{s}{\sqrt{2k^2-1}},k\right)\, ds\\
&=\frac{16k^4}{{(2k^2-1)^{3/2}}}\int^{ K(k) }_0 
\cn^4\left( s,k\right)\, ds\\
&=\frac{16}{{(2k^2-1)^{3/2}}}
\left(k^4 K(k)-\frac53 k^2 K(k) +\frac43 k^2 E(k) +\frac23 K(k) -\frac23 E(k)\right).
\end{align*}
In view of the differential equation
$\mykappa_{ss}+\frac12 \mykappa^3-\mykappa=0$, the boundary conditions, and the symmetry  we finally see that 
$$
 \int^{\tilde L(k)/2}_0 \mykappa_{k,s}(s)^2\, ds =
\frac12 \int^{\tilde L(k)/2}_0 \mykappa_k(s)^4\, ds -\int^{\tilde L(k)/2}_0 \mykappa_k(s)^2\, ds, 
$$
hence
\begin{align*}
Q(k)&=
\frac{\int^{\tilde L(k)/2}_0 \mykappa_k(s)^4\, ds }{ 2 \tilde L(k) \cdot  \left( \int^{\tilde L(k)/2}_0 \mykappa_k (s)^2\, ds \right)\cdot  
	\left( \int^{\tilde L(k)/2}_0 \mykappa_{k,s}(s)^2\, ds \right)}\\
&=\frac{1}{\tilde L(k)\cdot  \left( \int^{\tilde L(k)/2}_0 \mykappa_k (s)^2\, ds \right)\cdot \left(1- 2\frac{ \int^{\tilde L(k)/2}_0 \mykappa_k (s)^2\, ds}{\int^{\tilde L(k)/2}_0 \mykappa_k(s)^4\, ds} \right)}.
\end{align*}
Putting all together we come up with
\begin{align}
Q(k)&=\frac{1}{8K(k)\cdot \left(  k^2 K(k) -K(k)+E(k)\right)}
\cdot \frac{1}{1- (k^2-1/2)\frac{ k^2 K(k) -K(k)+E(k)}{k^4 K(k)-\frac53 k^2 K(k) +\frac43 k^2 E(k) +\frac23 K(k) -\frac23 E(k)}}.\label{eq:2.3.3}
\end{align}

\bigskip
\noindent 
\underline{Case (b): $k =1/\sqrt2$.} We recall that here, we have to consider
$$
\mykappa_{1/\sqrt2}(s) =\sqrt{2}\cn\left( s,1/\sqrt2\right),   \quad
s\in [0,\tilde L(1/\sqrt2)], \quad \tilde L(1/\sqrt{2} )=2  K(1/\sqrt2).
$$
Here  the differential equation is just 
$\mykappa_{ss}+\frac12 \mykappa^3=0$ so that 
$$
\int^{\tilde L({1/\sqrt{2}}   )/2}_0 \mykappa_{{1/\sqrt{2}}   ,s}(s)^2\, ds =
\frac12 \int^{\tilde L({1/\sqrt{2}}   )/2}_0 \mykappa_{1/\sqrt{2}}   (s)^4\, ds ,
$$
hence
\begin{align*}
Q({1/\sqrt{2}}   )&=
\frac{\int^{\tilde L({1/\sqrt{2}}   )/2}_0 \mykappa_{1/\sqrt{2}}   (s)^4\, ds }{ 2 \tilde L({1/\sqrt{2}}   ) \cdot  \left( \int^{\tilde L({1/\sqrt{2}}   )/2}_0 \mykappa_{1/\sqrt{2}}    (s)^2\, ds \right)\cdot  
	\left( \int^{\tilde L({1/\sqrt{2}}   )/2}_0 \mykappa_{{1/\sqrt{2}}   ,s}(s)^2\, ds \right)}\\
&=\frac{1}{\tilde L({1/\sqrt{2}}   )\cdot  \left( \int^{\tilde L({1/\sqrt{2}}   )/2}_0 \mykappa_{1/\sqrt{2}}    (s)^2\, ds \right)}.
\end{align*}
In this particluar case $k =1/\sqrt2$ we find
$$
 \int^{\tilde L({1/\sqrt{2}}   )/2}_0 \mykappa_{1/\sqrt{2}}    (s)^2\, ds
 =2 \int_0^{K({1/\sqrt{2}}   )} \cn^2 (s,{1/\sqrt{2}}   )\, ds=2( 2E({1/\sqrt{2}}   )-K({1/\sqrt{2}}   )),
$$
hence 
$$
Q(1/\sqrt2)=\frac{1}{4K(1/\sqrt2) (2 E(1/\sqrt2)-K(1/\sqrt2) )}.
$$
One may observe that this expression can be obtained from \eqref{eq:2.3.3}, when putting
there formally $k=1/\sqrt2$.

\bigskip
\noindent 
\underline{Case (c): $k\in (0,1/\sqrt2)$.} Formally this case is rather similar to Case~(a). However, we have a different sign in the differential equation. Here we have to consider
$$
\mykappa_k(s) =\frac{2k}{\sqrt{1-2k^2}}\cn\left( \frac{s}{\sqrt{1-2k^2}},k\right),   \quad
s\in [0,\tilde L(k)], \quad \tilde L(k)=2 \sqrt{1-2k^2} K(k).
$$
Making use of \eqref{eq:2.1.2} and of \eqref{eq:2.1.4} we find:
\begin{align*}
\int^{\tilde L(k)/2}_0 \mykappa_k (s)^2\, ds &=\frac{4k^2}{{1-2k^2}}\int^{\sqrt{1-2k^2} K(k) }_0 
\cn^2\left( \frac{s}{\sqrt{1-2k^2}},k\right)\, ds\\
&=\frac{4k^2}{\sqrt{1-2k^2}}\int^{ K(k) }_0 
\cn^2\left( s,k\right)\, ds
=\frac{4}{\sqrt{1-2k^2}}\left( k^2 K(k) -K(k)+E(k)\right),
\\
\int^{\tilde L(k)/2}_0 \mykappa_k (s)^4\, ds &=\frac{16k^4}{{(1-2k^2)^2}}\int^{\sqrt{1-2k^2} K(k) }_0 
\cn^4\left( \frac{s}{\sqrt{1-2k^2}},k\right)\, ds\\
&=\frac{16k^4}{{(1-2k^2)^{3/2}}}\int^{ K(k) }_0 
\cn^4\left( s,k\right)\, ds\\
&=\frac{16}{{(1-2k^2)^{3/2}}}
\left(k^4 K(k)-\frac53 k^2 K(k) +\frac43 k^2 E(k) +\frac23 K(k) -\frac23 E(k)\right).
\end{align*}
In view of the boundary conditions, the symmetry and the differential equation
$\mykappa_{ss}+\frac12 \mykappa^3+\mykappa=0$ we finally see that 
$$
\int^{\tilde L(k)/2}_0 \mykappa_{k,s}(s)^2\, ds =
\frac12 \int^{\tilde L(k)/2}_0 \mykappa_k(s)^4\, ds +\int^{\tilde L(k)/2}_0 \mykappa_k(s)^2\, ds, 
$$
hence
\begin{align*}
Q(k)&=
\frac{\int^{\tilde L(k)/2}_0 \mykappa_k(s)^4\, ds }{ 2 \tilde L(k) \cdot  \left( \int^{\tilde L(k)/2}_0 \mykappa_k (s)^2\, ds \right)\cdot  
	\left( \int^{\tilde L(k)/2}_0 \mykappa_{k,s}(s)^2\, ds \right)}\\
&=\frac{1}{\tilde L(k)\cdot  \left( \int^{\tilde L(k)/2}_0 \mykappa_k (s)^2\, ds \right)\cdot \left(1+ 2\frac{ \int^{\tilde L(k)/2}_0 \mykappa_k (s)^2\, ds}{\int^{\tilde L(k)/2}_0 \mykappa_k(s)^4\, ds} \right)}.
\end{align*}
Observe the difference in the sign, when compared with Case~(a).
Putting all together we come up with
\begin{align}
Q(k)&=\frac{1}{8K(k)\cdot \left(  k^2 K(k) -K(k)+E(k)\right)}
\cdot \frac{1}{1+ (1/2-k^2)\frac{ k^2 K(k) -K(k)+E(k)}{k^4 K(k)-\frac53 k^2 K(k) +\frac43 k^2 E(k) +\frac23 K(k) -\frac23 E(k)}},\label{eq:2.3.4}
\end{align}
which, somehow surprisingly,  coincides  precisely with \eqref{eq:2.3.3}.

\bigskip
\noindent 
\underline{Case (d): $k=0$.} This case is simple:
$\mykappa_0(s)=\cos(s)$,  $s\in [0,\tilde L(0)]$, $ \tilde L(0)=\pi$.
We calculate
\begin{align*}
\int^{\tilde L(0)/2}_0 \mykappa_0 (s)^2\, ds =&\int^{\pi/2}_0 \cos^2(s) \, ds=\frac{\pi}{4},
\quad 
\int^{\tilde L(0)/2}_0 \mykappa_0 (s)^4\, ds =\int^{\pi/2}_0 \cos^4(s) \, ds=\frac{3\pi}{16},\\
\int^{\tilde L(0)/2}_0 \mykappa_{0,s} (s)^2\, ds =&\int^{\pi/2}_0 \sin^2(s) \, ds=\frac{\pi}{4}.
\end{align*}
This gives
\begin{equation}\label{eq:3.3.6}
Q(0)= 
\frac{\int^{\tilde L(0)/2}_0 \mykappa_0(s)^4\, ds }{ 2 \tilde L(0) \cdot  \left( \int^{\tilde L(0)/2}_0 \mykappa_0 (s)^2\, ds \right)\cdot  
	\left( \int^{\tilde L(0)/2}_0 \mykappa_{0,s}(s)^2\, ds \right)}
=\frac{\frac{3\pi}{16}}{\frac{\pi^3}{8}}=\frac{3}{2\pi^2}.
\end{equation}

\bigskip
\noindent 
\underline{Case (e): $k\in (-1,0)$.} We recall that here, we have to consider
$$
\mykappa_k(s) =\frac{2k}{\sqrt{k^2+1}}\sn\left( \frac{s}{\sqrt{k^2+1}},-k\right),   \quad
s\in [-\tilde L(k)/2,\tilde L(k)/2], \quad \tilde L(k)=2 \sqrt{k^2+1} K(-k).
$$
Making use of \eqref{eq:2.1.1} and of \eqref{eq:2.1.3} we find:
\begin{align*}
\int^{\tilde L(k)/2}_0 \mykappa_k (s)^2\, ds &=\frac{4k^2}{{k^2+1}}\int^{\sqrt{k^2+1} K(-k) }_0 
\sn^2\left( \frac{s}{\sqrt{k^2+1}},-k\right)\, ds\\
&=\frac{4k^2}{\sqrt{k^2+1}}\int^{ K(-k) }_0 
\sn^2\left( s,-k\right)\, ds
=\frac{4}{\sqrt{k^2+1}}\left(K(-k) -E(-k) \right),
\\
\int^{\tilde L(k)/2}_0 \mykappa_k (s)^4\, ds &=\frac{16k^4}{{(k^2+1)^2}}\int^{\sqrt{k^2+1} K(-k) }_0 
\sn^4\left( \frac{s}{\sqrt{k^2+1}},-k\right)\, ds\\
&=\frac{16k^4}{{(k^2+1)^{3/2}}}\int^{ K(k) }_0 
\sn^4\left( s,-k\right)\, ds\\
&=\frac{16}{{(k^2+1)^{3/2}}}
\left(\frac13 k^2 K(-k) -\frac23 k^2 E(-k) +\frac23 K(-k)-\frac23 E(-k)  \right).
\end{align*}
In view of the boundary conditions, the symmetry and the differential equation
$\mykappa_{ss}-\frac12 \mykappa^3+\mykappa=0$ we finally see that 
$$
\int^{\tilde L(k)/2}_0 \mykappa_{k,s}(s)^2\, ds =
-\frac12 \int^{\tilde L(k)/2}_0 \mykappa_k(s)^4\, ds +\int^{\tilde L(k)/2}_0 \mykappa_k(s)^2\, ds, 
$$
hence
\begin{align*}
Q(k)&=
\frac{\int^{\tilde L(k)/2}_0 \mykappa_k(s)^4\, ds }{ 2 \tilde L(k) \cdot  \left( \int^{\tilde L(k)/2}_0 \mykappa_k (s)^2\, ds \right)\cdot  
	\left( \int^{\tilde L(k)/2}_0 \mykappa_{k,s}(s)^2\, ds \right)}\\
&=\frac{1}{\tilde L(k)\cdot  \left( \int^{\tilde L(k)/2}_0 \mykappa_k (s)^2\, ds \right)\cdot \left(-1+ 2\frac{ \int^{\tilde L(k)/2}_0 \mykappa_k (s)^2\, ds}{\int^{\tilde L(k)/2}_0 \mykappa_k(s)^4\, ds} \right)}.
\end{align*}
Putting all together we come up with
\begin{align}
Q(k)&=\frac{1}{8K(-k)\left(K(-k) -E(-k) \right)}\cdot 
\frac{1}{-1+ \frac12 (k^2+1)  \frac{\left(K(-k) -E(-k) \right)}{\frac13 k^2 K(-k) -\frac23 k^2 E(-k) +\frac23 K(-k)-\frac23 E(-k) }}.
\label{eq:2.3.7}
\end{align}

\subsection{{Concluding the p}roof of Theorem~\ref{TMcritineqintro}}
\label{sect:proof_thm1}

For a plot, putting all the five cases together, see Figure~\ref{figure:1}.
By means of some asymptotic or computer assisted calculations one  can prove that $Q$   can   be continuously extended to the interval $[-1,1]$ with  $Q(-1)=0$, $Q(0)=\frac{3}{2\pi^2}$, $Q(1/\sqrt{2})=\frac{1}{2\pi}$,
 and $Q(1)=0$.
This yields a ``graphical'' proof of Theorem~\ref{TMcritineqintro}.
\begin{figure}[h] 
	\centering 
	\includegraphics[width=.45\textwidth]{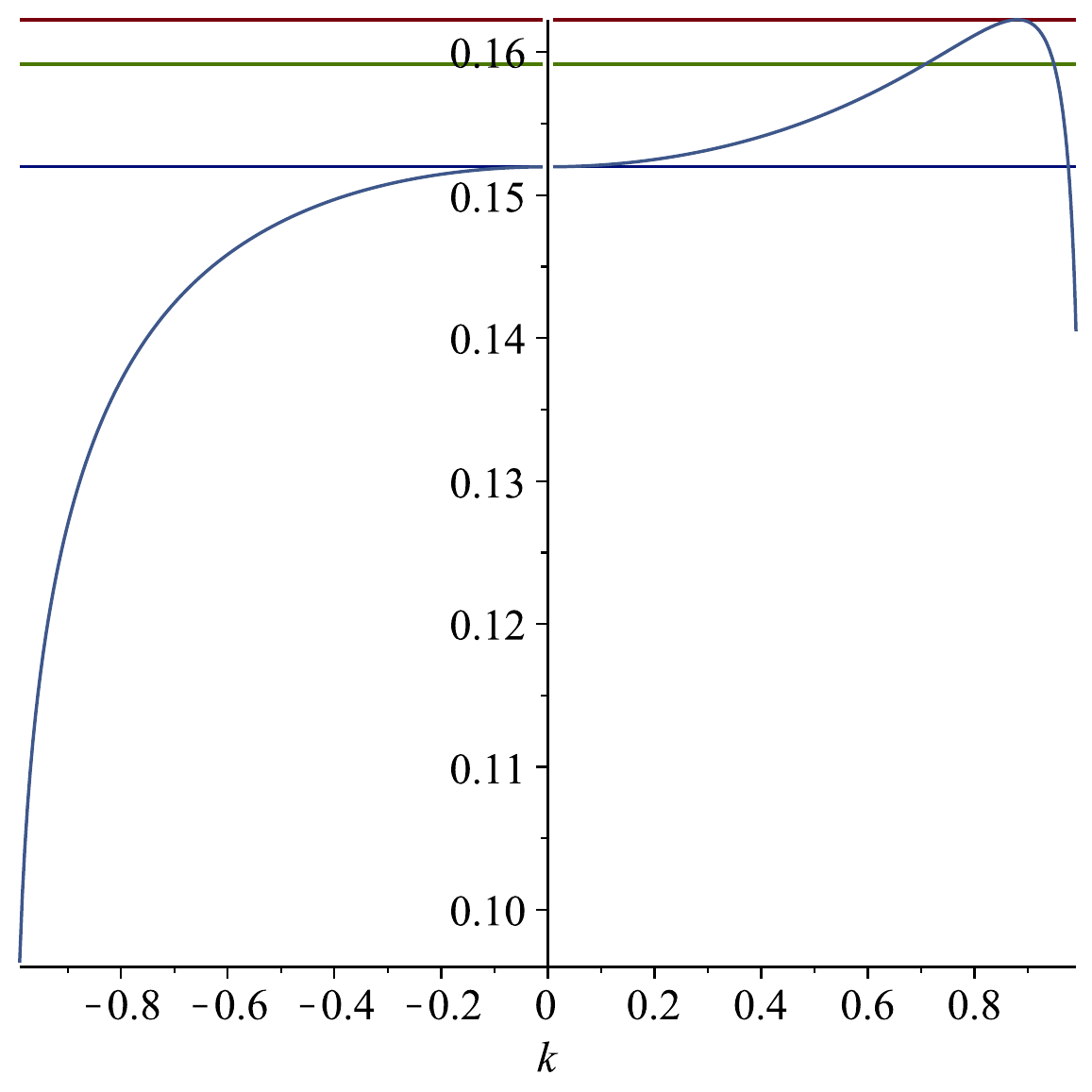}
	\caption{The function {$Q$}, together with the constants $\frac{3}{2\pi^2}$, $\frac{1}{2\pi}$, and $0.162278$.}
	\label{figure:1} 
\end{figure} 

For the relevant interval $k\in (1/\sqrt{2},1)$ we provide some analytical computations. Instead of maximising {$Q$} we may minimise
$$
(1/\sqrt{2},1)\ni k\mapsto \frac{1}{Q(k)}=F_1(k)\cdot F_2(k)
$$
with 
\begin{align*}
F_1(k)&:=8K(k)\cdot (E(k)-(1-k^2)K(k))>0,\\
F_2(k)&:=\frac{(1-k^2)K(k)+(2k^2-1)E(k)}{2(1-k^2)(2-3k^2)K(k)+4(2k^2-1)E(k)}>0.
\end{align*}
We calculate further 
\begin{align*}
F_1'(k)&= \frac{8}{k(1-k^2)} (E(k)-(1-k^2)K(k))^2+8kK(k)^2>0,\\
F_2'(k)&=\frac32 k\cdot  \frac{(1-k^2)K(k)^2+(2k^2-1)E(k)^2-2k^2E(k)K(k)}{\left( (1-k^2)(2-3k^2)K(k)+2(2k^2-1)E(k) \right)^2}<0.
\end{align*}
The critical value of $k$, where $1/Q$ is minimised (or where $Q$ is maximised), is given as solution of the equation
$$
0\stackrel{!}{=}F_1'(k)\cdot F_2(k)+F_1(k)\cdot F_2'(k)
\quad \Leftrightarrow \quad \frac{F_1'(k)}{F_1(k)}=-\frac{F_2'(k)}{F_2(k)},
$$
which is in our situation equivalent to 
\begin{align}
&\frac{K(k)-E(k)}{k(E(k)-(1-k^2)K(k))}+ \frac{E(k)}{k(1-k^2)K(k)}\nonumber\\
&\quad =3k\cdot \frac{-(1-k^2)K(k)^2+2k^2E(k)K(k)
	-(2k^2-1)E(k)^2}{\left( (1-k^2)(2-3k^2)K(k)+2(2k^2-1)E(k) \right)\cdot \left( (1-k^2)K(k)+(2k^2-1)E(k)\right)}.
\label{eq:2.3.8}
\end{align}
This is solved by
$$
k_{\max}=0.8802924038863\ldots 
$$
with
\begin{equation}\label{eq:2.3.9}
\mathscr{C}_0=Q(k_{\max})=0.162277833628\ldots.
\end{equation}
We take $C_0=0.162278$ as a reliable upper bound for $\mathscr{C}_0$.
\hfill$\square$

\subsection{{Discussing possible refinements of Theorem~\ref{TMcritineqintro} and possibly critical initial data}}
\label{sect:discussing_proof_thm1}

\begin{remark}\label{rem:3}
	In view of the application of the second claim of Lemma~\ref{LMlengthevo} in the proof of Theorem~\ref{TMmain}
	(see also \cite[Lemma 3.3]{WW24}) one could hope that
	the best constant in
	\[
	\int^L_0\! u(s)^4\, ds \le C_0\cdot L \cdot  \left( \int^L_0 \! u(s)^2\, ds \right)\cdot  
	\left( \int^L_0 \! u'(s)^2\, ds \right)
	+\frac{4L\cdot\int^L_0 \!\left( u''(s) + \frac12 u(s)^3\right)^2 ds}{ \left( \int^L_0 u(s)^2\, ds \right)}
	\]
	could be $C_0=\frac{1}{2\pi}$.
	
	In order to check whether this could possibly be true one may consider the functions from (a) in Section~\ref{sect:2.3}. There one simply has
	$$
	\frac{4L}{ \left( \int^L_0 u(s)^2\, ds \right)}\cdot\int^L_0 \left( u''(s) + \frac12 u(s)^3\right)^2\, ds=4L.
	$$
	In the affirmative, the modified quantity 
	\begin{align*}
	\tilde Q(k)&=
	\frac{\int^{\tilde L(k)/2}_0 \mykappa_k(s)^4\, ds-2 \tilde L(k)}{ 2 \tilde L(k) \cdot  \left( \int^{\tilde L(k)/2}_0 \mykappa_k (s)^2\, ds \right)\cdot  
		\left( \int^{\tilde L(k)/2}_0 \mykappa_{k,s}(s)^2\, ds \right)}\\
	&=\frac{1-\frac{2\tilde L(k)}{\int^{\tilde L(k)/2}_0 \mykappa_k(s)^4\, ds}
	}{\tilde L(k)\cdot  \left( \int^{\tilde L(k)/2}_0 \mykappa_k (s)^2\, ds \right)\cdot \left(1- 2\frac{ \int^{\tilde L(k)/2}_0 \mykappa_k (s)^2\, ds}{\int^{\tilde L(k)/2}_0 \mykappa_k(s)^4\, ds} \right)}\\
&=\frac{1-\frac{(2k^2-1)^2K(k)}{4\left( k^4 K(k)-\frac53 k^2 K(k) +\frac43 k^2 E(k) +\frac23 K(k) -\frac23 E(k)\right)}
}{8K(k)\cdot \left(  k^2 K(k) -K(k)+E(k)\right)}\nonumber\\ & \quad
\cdot \frac{1}{1- (k^2-1/2)\frac{ k^2 K(k) -K(k)+E(k)}{k^4 K(k)-\frac53 k^2 K(k) +\frac43 k^2 E(k) +\frac23 K(k) -\frac23 E(k)}}
\end{align*}
should stay below $\frac{1}{2\pi}$. Figure~\ref{figure:2} shows that this conjecture is false. One should carefully observe that this is just an observation
on a special family of functions. This does not show anything about possible improvements of Theorem~\ref{TMcritineqintro}.
\begin{figure}[h] 
	\centering 
	\includegraphics[width=.45\textwidth]{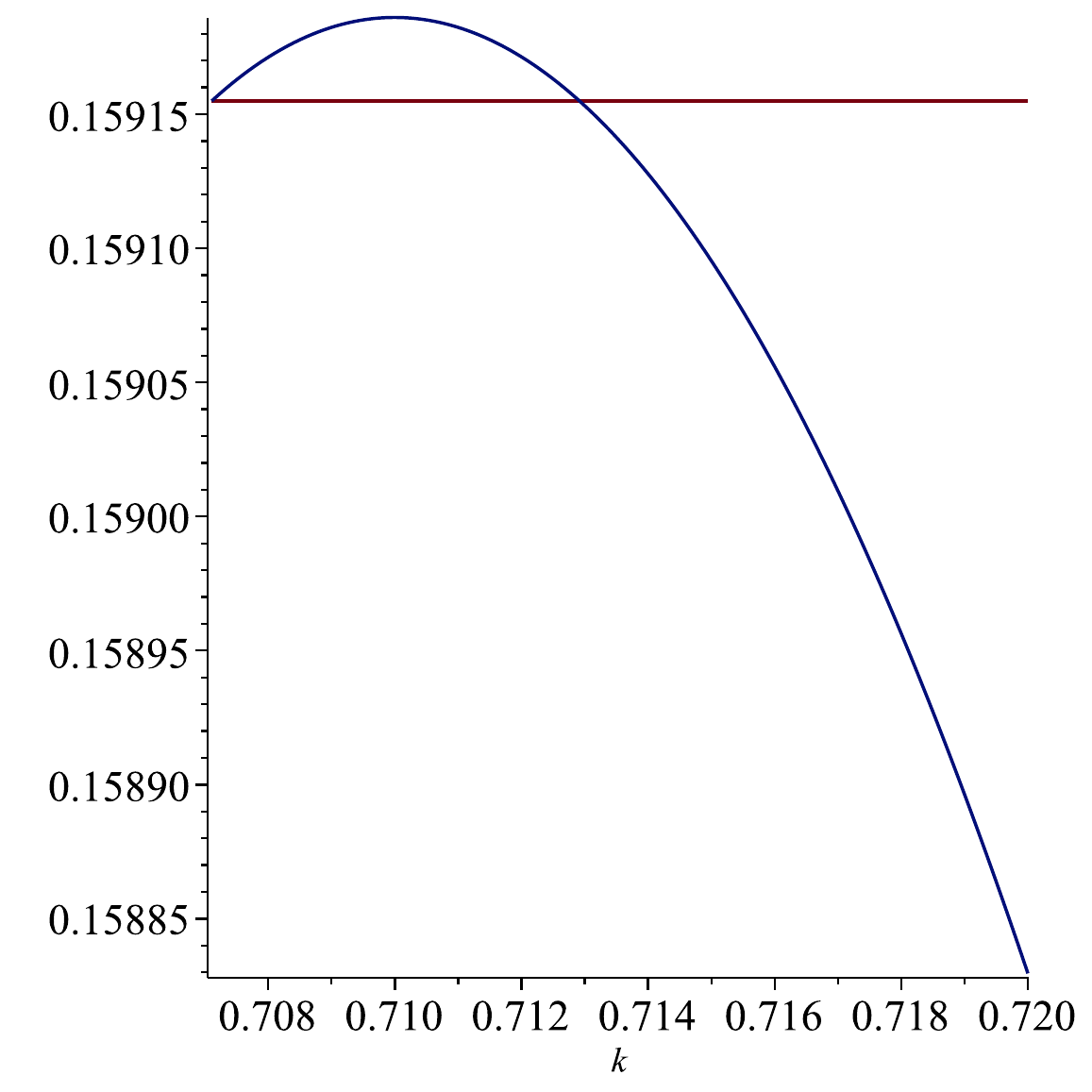}
	\caption{The function {$\tilde Q$}, together with the constant $\frac{1}{2\pi}$.}
	\label{figure:2} 
\end{figure} 
\end{remark}

\begin{remark}\label{rem:4}
	In order to possibly improve 
	 the application of the second claim of Lemma~\ref{LMlengthevo} in the proof of Theorem~\ref{TMmain}
	(see also \cite[Lemma 3.3]{WW24}) under an 
	optimal assumption, one may think of assuming
	\begin{equation}\label{eq:rem3.1}
	 L \cdot  \left( \int^L_0 u(s)^2\, ds \right)\le 4\pi.
	\end{equation}
	In view of the nonlinear character of the inequality in Remark~\ref{rem:3} one could hope that under the constraint \eqref{eq:rem3.1} the constant could be possibly improved to $C_0=\frac{1}{2\pi}$. In other words: The question is whether
	under the assumptions \eqref{eq:rem3.1} and --of course-- $\int^L_0 u(s)\, ds=0$ the following inequality is true or not:
	\begin{equation}\label{eq:rem3.2}
	LHS:=\left( \int^L_0 u(s)^2\, ds \right)\cdot
	\left( -  \int^L_0 u'(s)^2\, ds  +\frac12   \int^L_0 u(s)^4\, ds  \right)
	-2L \int^L_0 \left( u''(s) + \frac12 u(s)^3\right)^2\, ds \le 0?
	\end{equation}
Motivated by Remark~\ref{rem:3} we consider 
\begin{equation}\label{eq:rem3.5}
u(s):= a\cn (s,k) \quad \mbox{on}\quad [0,L(k)], \quad L(k)=2 K(k),
\end{equation}
with
\begin{equation}\label{eq:rem3.6}
k:=0.71\quad \text{and}\quad a:=\sqrt{\frac{\pi}{K(k)^2-\frac{1}{k^2}(K(k)^2-E(k)K(k))}}
=1.41233776\ldots
\end{equation}
so that 
$$
\int^{L(k)}_0 u(s)^2\, ds =\frac{4\pi}{L(k)}=\frac{2\pi}{K(k)}.
$$
We calculate further: 
\begin{align*}
\int^{L(k)}_0 u(s)^4\, ds &=2a^4\int_0^{K(k)}\cn^4(s,k)\, ds\\
&=2\frac{a^4}{k^4}\left( k^4 K(k)-\frac53 k^2 K(k)+\frac43 k^2 E(k) +\frac23 K(k) -\frac23 E(k)\right)\\
&=4.92029245\ldots\\
u'(s)^2&=a^2\sn^2 (s,k)\cdot \left( 1-k^2 \sn^2 (s,k) \right) \\
\int^{L(k)}_0 u'(s)^2\, ds &=2a^2\int_0^{K(k)}\sn^2(s,k)\, ds -
2a^2k^2\int_0^{K(k)}\sn^4(s,k)\, ds\\
&=\frac23 a^2 \left( \frac{1}{k^2}(K(k)-E(k))-K(k)+2E(k)\right)
=2.45917590\ldots \\
u''(s) &=a(2k^2-1) \cn (s,k)-2a\, k^2\, \cn^3(s,k)\\
\int^{L(k)}_0 \left( u''(s) + \frac12 u(s)^3\right)^2\, ds &=
2\int_0^{K(k)} \left( a(2k^2-1) \cn (s,k)+\left(\frac12 a^3-2a\, k^2\right)\, \cn^3(s,k)\right)^2\, ds\\
&=0.0000273294\ldots
\end{align*}
Putting all together we find that
$$
LHS=0.00307904\ldots >0,
$$
which means that our conjecture \eqref{eq:rem3.2} is false.

Based on these observations  we construct an initial ``possibly critical'' curve as follows.
We first define a curve parameterised by arclength having {the} particular {$u$} from \eqref{eq:rem3.5} and \eqref{eq:rem3.6} as a curvature function:
\begin{equation} \label{eq:gammac}
[0,L(k)]\ni s\mapsto \tilde{\gamma}(s,0):= \begin{pmatrix}
	\tilde{\gamma}^1(s)\\ \tilde{\gamma}^2(s) 
\end{pmatrix}
:=\begin{pmatrix}
	\int^s_0 \cos \left( \int^\sigma_0 u({\varrho})\, d{\varrho}\right)\, d\sigma\\
	\int^s_0 \sin \left( \int^\sigma_0 u({\varrho})\, d{\varrho}\right)\, d\sigma
\end{pmatrix} {.}
\end{equation}
\textcolor{black}{We take a truncated value of $a$ from \eqref{eq:rem3.6} in order to slightly decrease the scaling invariant energy below $2\pi$ while still keeping the expression $LHS>0$.}
This curve is then reparameterised and rescaled to be defined on $[-1,1]$, to fit between the straight lines $\eta_{\pm1}$ 
and to satisfy the boundary conditions:
$$
[-1,1]\ni \rho\mapsto \gamma(\rho,0):= \frac{2}{\tilde{\gamma}^1(L(k))}\begin{pmatrix}
	\tilde{\gamma}^1((\rho+1)\cdot L(k)/2)\\ \tilde{\gamma}^2((\rho+1)\cdot L(k)/2)\end{pmatrix}  -\begin{pmatrix} 1\\0\end{pmatrix}.
$$

Observing that {$\gamma_c:=\gamma(\cdot,0)$} and hence its whole evolution {$\gamma(\cdot,t)$} have turning number zero, the quantity {$K_{\operatorname{osc}}(\gamma(\cdot,t))$} used in \cite[Lemma 3.3]{WW24} \textcolor{black}{equals $2\siE[\gamma(\cdot,t)]$}. Hence, this lemma or --equivalently-- Lemma~\ref{LMlengthevo} shows  for the scaling invariant energy that
$$
{\siE[\gamma(\cdot,0)]}\textcolor{black}{<}2\pi,\qquad 
\frac{d}{dt}{\siE[\gamma(\cdot,t)]}|_{t=0}>0.
$$
This shows analytically that initially the flow of {$\gamma_c$} points away from the locally stable straight lines. The long time behaviour and the question, whether the flow of {$\gamma_c$} is outside the basin of attraction of the straight lines, is studied in Section~\ref{sec:numerics}. 
For a plot of this ``possibly critical'' curve {$\gamma_c$}, see Figure~\ref{figure:3}.

\begin{figure}[h] 
	\centering 
	\includegraphics[width=.3\textwidth]{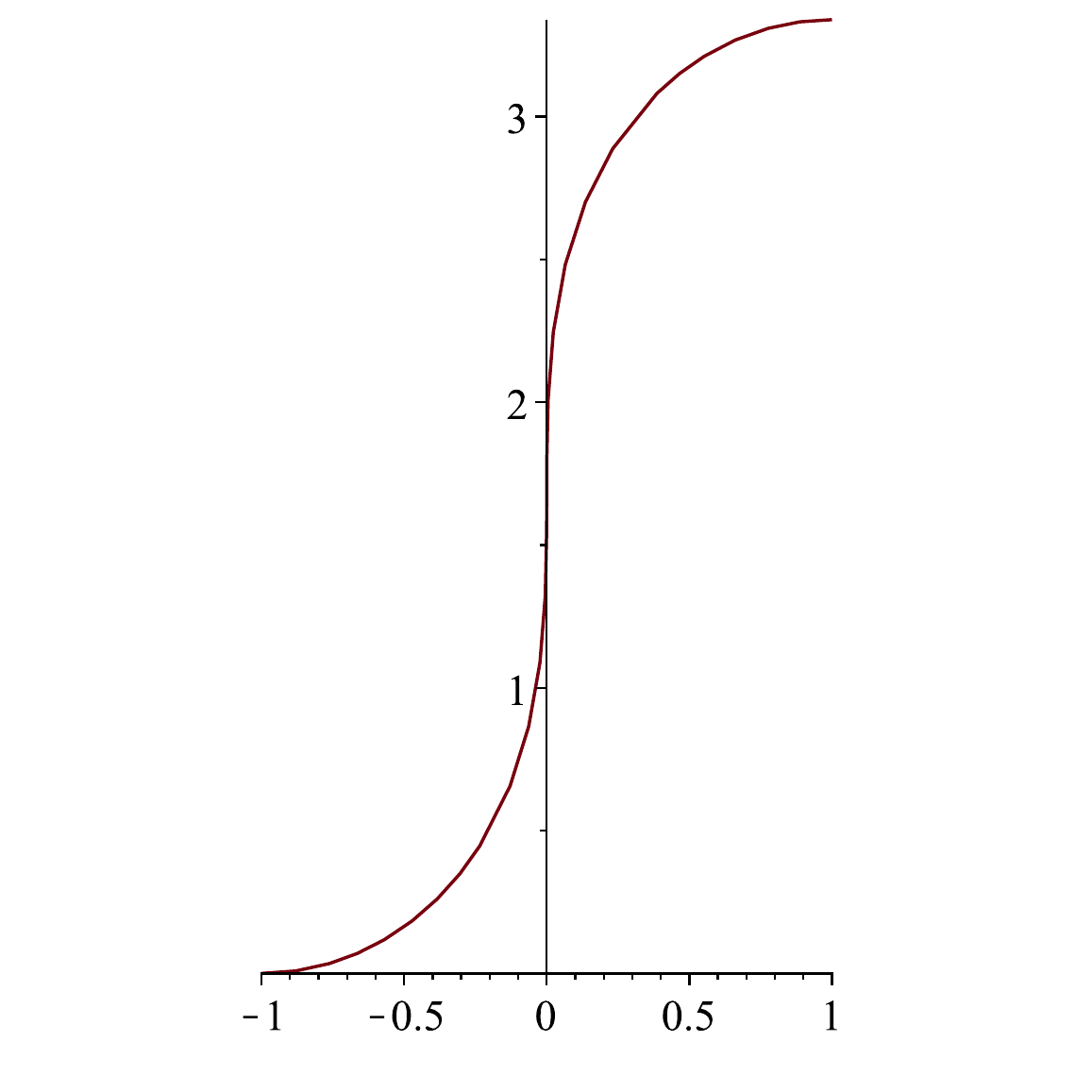}
	\caption{{The} initial curve {$\gamma_c$}, which is possibly not attracted by the straight lines.}
	\label{figure:3} 
\end{figure}
This shows that {the condition} \eqref{eq:rem3.1} does not suffice to carry out 
the proof of Theorem~\ref{TMmain} of convergence to a straight line.
This holds true also for slightly smaller initial data.
Whether the initial data described here might still be attracted by a straight line or are already outside their basin of attraction is an interesting and challenging open question. Numerical indications are given in Section~\ref{sec:numerics}.

We compare this  possibly critical profile with a suitably scaled {half-period}
of Euler's elastica, for a definition of the latter see the caption of Figure~\ref{figrect}, both parameterised proportional to their arclength. See Figure~\ref{figure:4} for the difference of Euler's curve minus our curve {$\gamma_c$}: Both coincide almost up to a relative error of $10^{-4}$, 
i.e. both plots look identical.
\begin{figure}[h] 
	\centering 
	\includegraphics[width=.3\textwidth]{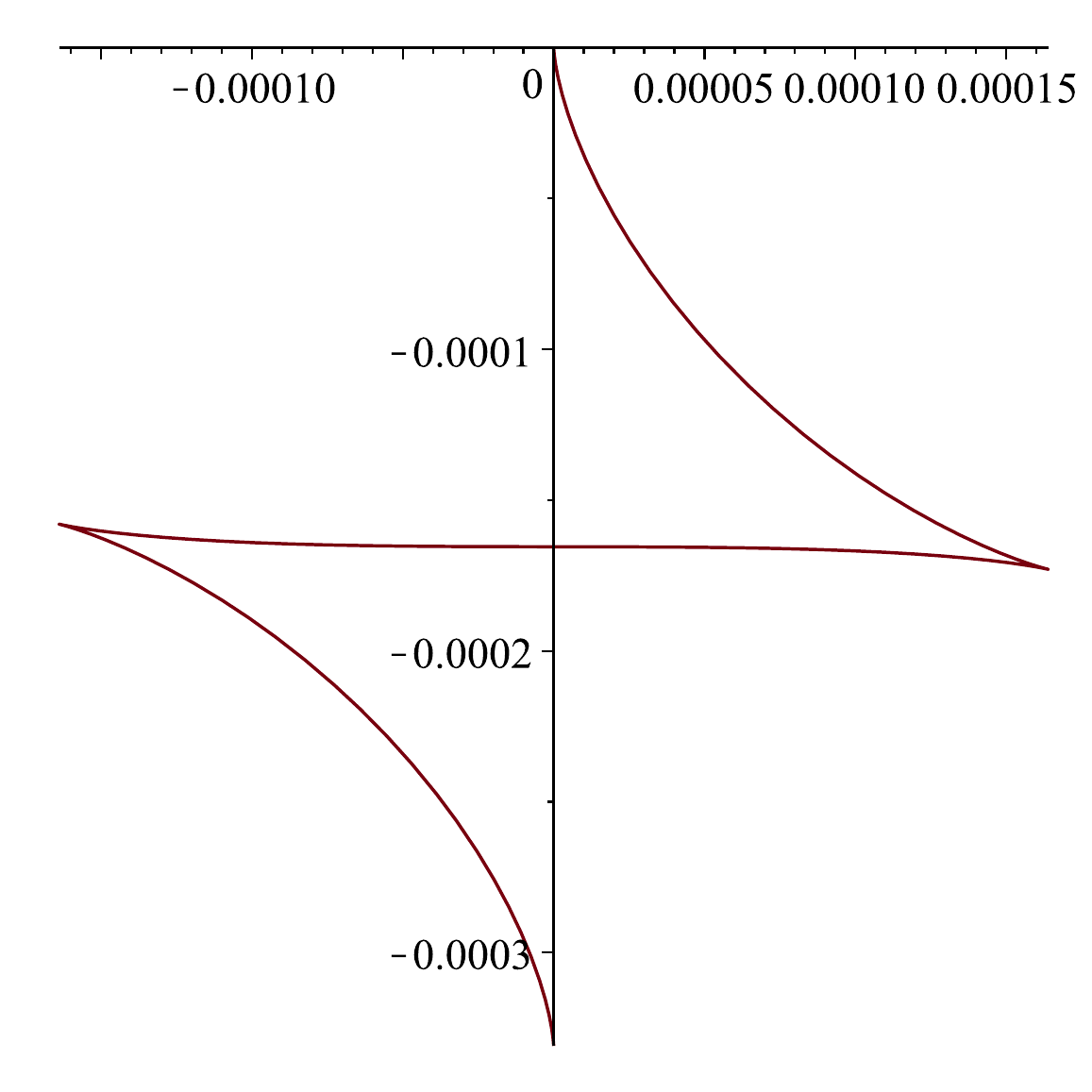}
	\caption{Rescaled Euler elastica minus {the} possibly critical initial curve {$\gamma_c$}.}
	\label{figure:4} 
\end{figure}
Figure~\ref{figure:5} displays the initial velocity of the flow in direction of the upward normal.
\begin{figure}[h] 
	\centering 
	\includegraphics[width=.3\textwidth]{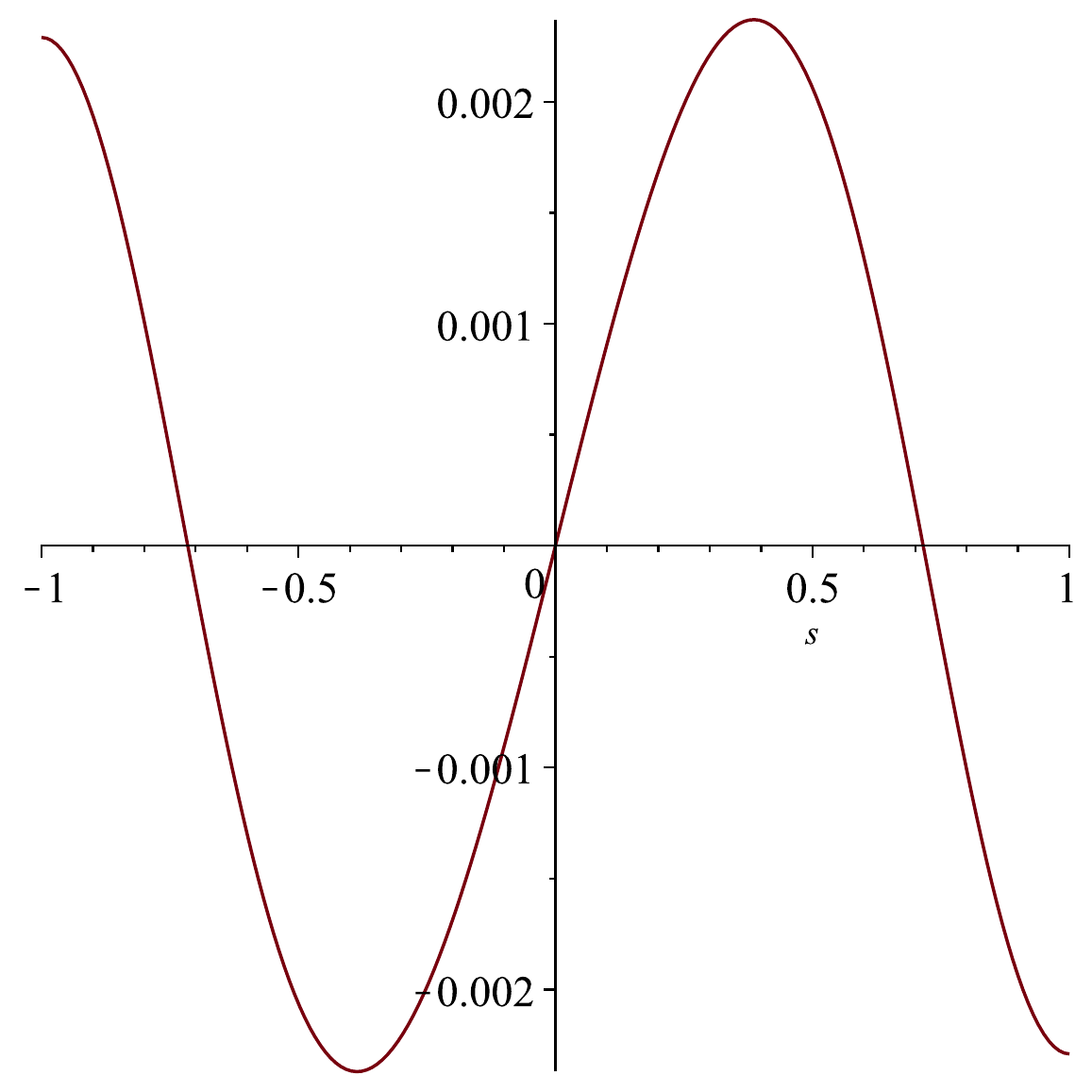}
	\caption{{$\frac{\partial}{\partial t}\gamma_c(\cdot,0)$} in direction of the upward normal.}
	\label{figure:5} 
\end{figure}
\end{remark}

\section{Convergence and Proof of Theorem \ref{TMmain}}
\label{Sconv}

Lemma~\ref{LMlengthest} shows that length can only grow to infinity as $t\to\infty$.
Clearly, if the flow is asymptotic to a stationary shape (a straight line segment or a multiple of half-periods of Euler's elastica), then the length does not grow to become unbounded.
It turns out that the presence of a uniform bound for evolving length is critical.
A bound for length implies uniform bounds for all derivatives of curvature, and so precompactness of trajectories (modulo translation).

\begin{theorem}
\label{TMsubconv}
Let $\gamma:[-1,1]\times[0,\infty)\rightarrow\R^2$ be a free boundary free elastic flow.
Suppose that there exists a constant $C$ such that $\myL[\gamma(\cdot,t)]<C$.
Then there are sequences $t_j\to\infty$ and $c_j\in\R$ such that
\[
\gamma(\cdot,t_j) - (0,c_j) \to \gamma_\infty\,,
\]
where the convergence is in the smooth topology on $\SX$ and $\gamma_\infty$ is an equilibrium, that is, a horizontal line segment or one of Euler's rectangular elastica.
\end{theorem}
\begin{proof}
The hypothesis on length implies that the following uniform version of \eqref{EQksesteqinpf} holds, recalling the notation $E_\ell(t):=\tfrac12\int \k_{s^\ell}^2\,ds$: 
\begin{equation*}
	\frac{d}{dt}\int  \k_{s^{{\ell}}}^2\,ds
	+ c\int  \k_{s^{{\ell}}}^2\,ds
	\le \hat{C}.
\end{equation*}
This in turn implies that the estimate of Lemma \ref{LMkshighest} holds uniformly on $[0,\infty)$.
Let $T_c(x,y) = (x,y+c)$ be any vertical translation.
The uniform estimates on length and all derivatives of curvature are thus enough to obtain precompactness of trajectories  $\{[\gamma(\cdot,t)]\}_{t\ge0}$ in $\SX/T_c$.

The energy identity \eqref{EQenrgyid} implies the existence of $t_j\to\infty$ such that $||\nabla\E[\gamma(\cdot,t_j)]||_{L^2}\to0$.
The smooth convergence then implies that $\nabla\E[\gamma_\infty]=0$ which finishes the proof.
\end{proof}

\begin{remark}
Since all trajectories are automatically immortal, and Euler's elastica are each unstable, a reasonably clear impression of the dynamics of the flow in the presence of a length bound emerges: \textcolor{black}{one may expect that a dense set of trajectories with bounded length converges to a horizontal line segment.}
We do not pursue this further here.
\end{remark}

\begin{remark}
If we assume that length is bounded, we may obtain convergence to a horizontal line segment with an initial condition on the energy $\E$ instead of the scale-invariant energy $\siE$.
The statement is as follows:

\medskip

\emph{Let $\gamma:[-1,1]\times[0,\infty)\rightarrow\R^2$ be a free boundary free elastic flow.
Suppose
\[
\E[\gamma(\cdot,0)] < \frac2\pi\Gamma\left(\frac34\right)^4\,,
\]
and that there exists a constant $C$ such that $\myL[\gamma(\cdot,t)]<C$.
Then $\gamma(\cdot,t)\to\gamma_\ell$, where $\gamma_\ell$ is a horizontal line segment.
The convergence is exponentially fast and in the smooth topology on $\SX$.}

\medskip

To prove this, note that Theorem \ref{TMsubconv} implies that the flow subconverges smoothly to an equilibrium.
Equilibria consist precisely of the horizontal line segment $[\gamma_\ell]$, which has $\E[\gamma_\ell] = 0$, and {$m$} half-periods of Euler's rectangular elastica $[\gamma_e^{(m)}]$, which has 
$\siE[\gamma_e^{(m)}] = 2\pi m^2$, 
$\myL[\gamma_e^{(m)}] = \frac12\left(\Gamma(\frac14)/\Gamma(\frac34)\right)^2$, hence
$\E[\gamma_e^{(m)}] = m^2\frac2\pi\Gamma\left(\frac34\right)^4$.
The initial energy condition combined with monotonicity of $\E[\gamma(\cdot,t)]$ implies that only the horizontal line segment $[\gamma_\ell]$ is a candidate limit for the trajectory.

As the convergence is smooth, we thus have $\E[\gamma(\cdot,t_j)]\to0$ and furthermore, due again to the length bound, $\siE[\gamma(\cdot,t_j)]\to0$.
In particular at some time $t_N$ the hypothesis of \cite[Theorem 1.4]{WW24} is satisfied.
This then yields the desired smooth exponential convergence to a particular limit in $[\gamma_\ell]$ and concludes the proof.
\end{remark}

Now let us prove the main result.

\begin{proof}[Proof of Theorem \ref{TMmain}]
To begin, Theorem \ref{TMlte} implies $T=\infty$.
As mentioned right after Theorem \ref{TMmain} the turning number of $\gamma(\cdot,t)$ is zero for all $t\ge 0$.

Therefore $\int \k\,ds = 0$, and (again because of $\gamma(\cdot,t)\in\SX$) we have that $\k_s$ vanishes at $s=0$ and $s={\myL[\gamma]}$.
Thus, we may apply Theorem \ref{TMcritineqintro} for each $t\in[0,\infty)$ with $u=\k(\cdot,t)$, obtaining
\begin{equation}
\int \k^4\,ds
    \le C_0\cdot\myL[\gamma(\cdot,t)]\int \k^2\,ds\ \int \k_s^2\,ds
    = 2C_0\siE[\gamma(\cdot,t)]\int \k_s^2\,ds
    \,.
\label{EQcritineqk}
\end{equation}
Lemma \ref{LMlengthevo} and \eqref{EQcritineqk} imply
\[
	\frac{d}{dt}\myL(\gamma(\cdot,t)) 
	= -\int \k_s^2\,ds + \frac12\int \k^4\,ds
    \le -
        \left(1-C_0\siE[\gamma(\cdot,t)]\right)
        \int \k_s^2\,ds
        \,.
\]
\textcolor{black}{Since $1.9615\pi<1/C_0$, condition \eqref{EQmtic} gives $\siE[\gamma(\cdot,0)]<1/C_0$. The preceding inequality and \eqref{EQsied} show that $\myL[\gamma(\cdot,t)]$ and $\siE[\gamma(\cdot,t)]$ are non-increasing whenever $\siE[\gamma(\cdot,t)]<1/C_0$. A continuity argument therefore gives $\siE[\gamma(\cdot,t)]\le\siE[\gamma(\cdot,0)]<1/C_0$ and $\myL[\gamma(\cdot,t)]\le\myL[\gamma(\cdot,0)]$ for all $t\ge0$.}

This places us in the setting of Theorem \ref{TMsubconv}.
Applying that result yields the convergence of $\gamma(\cdot,t_j)-(0,c_j)$ to an equilibrium.
The scale invariant energy of the limit is less than $2\pi$, which rules out every family of candidates apart from vertical translations of the horizontal line segment.
As the convergence is smooth, this implies $\siE[\gamma(\cdot,t_j)]\to0$.
And so, for some $t_N$, the curve $\gamma(\cdot,t_N)$ satisfies $\textcolor{black}{\siE[\gamma(\cdot,t_N)]\le\frac\pi2}$, the hypothesis of \cite[Theorem 1.4]{WW24}.
That result then applies, yielding the smooth, full convergence to a single horizontal line segment (not modulo translation).
This finishes the proof.
\end{proof}

\section{Numerics}
\label{sec:numerics}

\newcommand{\myF}{\mathcal F}

The numerical simulations in this section are based on an approach developed in \cite{DN25}
for the evolution of closed curves by the free elastic flow. The evolving curves
are described 
with the help of a mapping $x : \overline{I} \times [0,T) \to \mathbb R^2$ ($I=(-1,1)$), whose velocity $x_t$ has, unlike $\gamma_t$ in \eqref{FEF},
a suitable tangential component. 

Under discretisation, this component leads to an almost uniform 
distribution of grid points along the curve. We refer the reader to
\cite{DN25} for further details.
In order to describe the corresponding scheme it is
convenient to formulate the evolution equation as a system of two second order equations for $x$ and the
second variable 
\begin{equation} \label{eq:y}
y = \frac{x_{\rho\rho}}{|x_\rho|^2} , \quad \mbox{ so that } \kappa = y \cdot \normal.
\end{equation}
It is shown in \cite[\S2.2]{DN25} that the system
\begin{equation} \label{eq:elasticxy}
x_t = - \frac{1}{|x_\rho |^2} y_{\rho \rho} + 
\frac{1}{|x_\rho |^2} \myF(x_\rho,y,y_\rho) y,
\end{equation}
where $\myF(a,b,c) \in \mathbb R^{2 \times 2}$ is given by $\myF(a,b,c)=\myF_1(a,b,c)+\myF_2(a,b,c)+\myF_3(a,b)$ with
\begin{align*}
\myF_1(a,b,c) &= \bigl( 2 a \cdot c + | a |^2 | b |^2  \bigr) I_2,  \\
\myF_2(a,b,c) &= 2 \bigl( c \otimes a -  a \otimes c \bigr) + 2 a \cdot b \bigl( a \otimes b - b \otimes a \bigr), \\
\myF_3(a,b)& = - \tfrac12 \bigl( |a|^2 |b|^2- (a \cdot b)^2 \bigr) I_2, 
\end{align*}
where $I_2 \in \mathbb R^{2\times2}$ denotes the identity matrix, has the required property that
\begin{displaymath}
x_t \cdot \normal = -\kappa_{ss}- \frac{1}{2} \kappa^3 \qquad \mbox{ in } I \times (0,T). 
\end{displaymath}

The boundary conditions in \eqref{FEF} can be encoded via $x(\cdot,0)= x_0 \in \SX$ and
\begin{subequations} \label{eq:bc}
\begin{alignat}{2} 
x_t(\pm1,t) \cdot e_1 & = 0 
\quad & t \in (0,T), \label{eq:bca} \\
x_\rho(\pm1,t) \cdot e_2 & = 0 
\quad  & t \in  (0,T), \label{eq:bcb} \\
y(\pm1,t) \cdot e_1 & = 0 
\quad  & t \in  (0,T), \label{eq:bcc} \\
y_\rho(\pm1,t)  \cdot e_2 & = 0 \quad & t \in  (0,T). \label{eq:bcd}
\end{alignat}
\end{subequations}
On recalling that $\kappa = y \cdot \normal$, we obtain
\begin{displaymath}
\kappa_s = \frac{1}{| x_\rho |} \kappa_\rho = \frac{1}{| x_\rho |} \bigl( y_\rho \cdot \normal + y \cdot \normal_\rho \bigr) = \frac{1}{| x_\rho |} \bigl( y_\rho \cdot \normal - \kappa y \cdot x_\rho \bigr),
\end{displaymath}
so that \eqref{eq:bcb}--\eqref{eq:bcd} imply that $\kappa_s=0$ on $\partial I \times (0,T)$. \\
In order to derive a weak formulation of \eqref{eq:elasticxy}, \eqref{eq:bc} we set  $\underline{V} = [H^1(I)]^2$ and define
$\underline{V}_\partial = \{ \xi \in \underline{V}\,:\,  \xi \cdot e_1 = 0 \text{ on } \partial I\}$. 
We then aim to find $x, y : \overline{I} \times [0,T) \rightarrow \mathbb R^2$ such 
that $x(\cdot,0)= x_0 \in \SX$ and $x_t, y \in \underline{V}_\partial, t \in (0,T)$ as well as
\begin{subequations} \label{eq:weak}
\begin{align} 
& \int_I x_t \cdot \chi |x_\rho|^2 \,\,d\rho - \int_I y_\rho \cdot \chi_\rho \,\,d\rho
=  \int_I \myF(x_\rho,y,y_\rho)  y \cdot \chi \,\,d\rho
\qquad \forall\ \chi \in \underline{V}_\partial, \label{eq:weaka} \\
& \int_I y \cdot \xi |x_\rho|^2 \,\,d\rho + \int_I x_\rho \cdot \xi_\rho \,\,d\rho
= 0
\qquad  \qquad \qquad \qquad \qquad  \quad \; \; \forall\ 
\xi \in \underline{V}_\partial. \label{eq:weakb}
\end{align}
\end{subequations}
We remark that in the above variational formulation \eqref{eq:bcb} and \eqref{eq:bcd} arise as natural boundary conditions. Furthermore, \eqref{eq:bca} and \eqref{eq:bcc} are a consequence of $x_t \in \underline{V}_\partial$ and
$y \in \underline{V}_\partial$, respectively. \\
In order to discretise in space, we choose
a partition $-1=q_0 < q_1< \ldots < q_{J-1} < q_J=1$ of $\overline{I} = [-1,1]$ into 
intervals $I_j=[q_{j-1},q_j]$ and let $h_j= q_j-q_{j-1}$ as well as $h=\max_{j=1,\ldots,J} h_j$. We define
the finite element space
\begin{displaymath}
V^h = \{\chi \in C^0(\overline{I})\,:\, \chi\!\mid_{I_j} \mbox{ is affine},\ 
j=1,\ldots, J\}
\end{displaymath}
and set $\underline{V}^h = [V^h]^2$ and $\underline{V}^h_\partial = \underline{V}^h \cap \underline{V}_\partial$. \\
For a time step  $\Delta t>0$ let $t_m= m \Delta t$ and denote by $(x^m_h, y^m_h) \in \underline{V}^h \times \underline{V}^h$ the
approximation of $(x(\cdot,t_m),y(\cdot,t_m))$. The algorithm 
we propose adapts  (5.4) from \cite{DN25} to open
curves. \\
Given $(x^m_h, y^m_h) \in \underline{V}^h \times \underline{V}^h$, find 
$(x^{m+1}_h, y^{m+1}_h) \in \underline{V}^h \times \underline{V}^h_\partial$, such that 
$x^{m+1}_h - x^m_h \in \underline{V}^h_\partial$ and
\begin{subequations} \label{eq:fdel}
\begin{align}
& \int_I \frac{x^{m+1}_h-x^m_h}{\Delta t} \cdot \chi |x^m_{h,\rho}|^2 \,\,d\rho
- \int_I y^{m+1}_{h,\rho} \cdot \chi_\rho \,\,d\rho \nonumber \\ & \quad
= 2 \int_I (y^{m+1}_{h,\rho} \cdot x^m_{h,\rho}) y^m_h \cdot \chi \,\,d\rho
+ \int_I | x^m_{h,\rho} |^2 (y^m_h \cdot y^{m+1}_h) y^m_h \cdot \chi \,\,d\rho
\nonumber \\ & \qquad
+ \int_I \myF_2(x^m_{h,\rho},y^m_h,y^m_{h,\rho}) y^{m+1}_h \cdot \chi \,\,d\rho + \int_I \myF_3(x^m_{h,\rho},y^m_h) y^m_h \cdot \chi \,\,d\rho
\qquad \forall\ \chi \in \underline{V}^h_\partial, \label{eq:fdela} \\
& \int_I y^{m+1}_h \cdot \xi |x^m_{h,\rho}|^2 \,\,d\rho
+ \int_I x^{m+1}_{h,\rho} \cdot \xi_\rho \,\,d\rho = 0
\qquad \forall\ \xi \in \underline{V}^h_\partial. \label{eq:fdelb}
\end{align}
\end{subequations}
Note that \eqref{eq:fdel} represents a square linear system for $(x^{m+1}_h,y^{m+1}_h)$, for
which we have the following existence and uniqueness result. 
\begin{lemma} If $| x^m_{h,\rho} | > 0$ in $I$, then \eqref{eq:fdel} has
a unique solution $(x^{m+1}_h,y^{m+1}_h) \in \underline{V}^h \times \underline{V}^h_\partial$.
\end{lemma}
\begin{proof} It is sufficient to prove that the corresponding homogeneous system only has the trivial solution. Hence, let
$(X_h,Y_h) \in \underline{V}^h_\partial \times \underline{V}^h_\partial$ be such that
\begin{subequations} 
\begin{align}
\int_I \frac{X_h}{\Delta t} \cdot \chi |x^m_{h,\rho}|^2 \,\,d\rho 
- \int_I Y_{h,\rho} \cdot \chi_\rho \,\,d\rho    
& = 2 \int_I (Y_{h,\rho} \cdot x^m_{h,\rho}) y^m_h \cdot \chi \,\,d\rho 
+ \int_I | x^m_{h,\rho} |^2 (y^m_h \cdot Y_h) y^m_h \cdot \chi
\,\,d\rho \nonumber \\ & \qquad
+ \int_I \myF_2(x^m_{h,\rho},y^m_h,y^m_{h,\rho}) Y_h \cdot \chi \,\,d\rho  
\qquad \forall\ \chi \in \underline{V}^h_\partial, \label{eq:fd0a} \\
\int_I Y_h \cdot \xi |x^m_{h,\rho}|^2 \,\,d\rho 
+ \int_I X_{h,\rho} \cdot \xi_\rho \,\,d\rho  & = 0
\qquad \forall\ \xi \in \underline{V}^h_\partial. \label{eq:fd0b}
\end{align}
\end{subequations}
Choosing $\chi = -\Delta t Y_h$ in \eqref{eq:fd0a}, $\xi = X_h$ in 
\eqref{eq:fd0b} and summing the two gives, upon observing that
$\myF_2(a,b,c) v \cdot v=0$ for any $v \in \mathbb R^2$, that
\begin{align*}
0 & = \int_I |X_{h,\rho}|^2 \,\,d\rho 
+ \Delta t \int_I |Y_{h,\rho}|^2 \,\,d\rho  
+ 2 \Delta t \int_I (Y_{h,\rho} \cdot x^m_{h,\rho}) y^m_h \cdot Y_h 
\,\,d\rho 
+ \Delta t \int_I | x^m_{h,\rho} |^2 (y^m_h \cdot Y_h)^2 \,d\rho  \\ & 
= \int_I |X_{h,\rho}|^2 \,d\rho  +
\Delta t \int_I | Y_{h,\rho} + (y^m_h \cdot Y_h) x^m_{h,\rho} |^2 
\,d\rho .
\end{align*}
It follows that $X_{h,\rho} = 0$ in $I$, and so first
\eqref{eq:fd0b} together with the assumption on $| x^m_{h,\rho} |$ implies that $Y_h=0$, after which we infer from \eqref{eq:fd0a}  that $X_h=0$. 
\end{proof}

We observe that the scheme \eqref{eq:fdel} is agnostic to the exact location of the two vertical support lines: they are implicitly defined through the positions of the endpoints of $x^m_h$, and hence of $x^0_h$. For simplicity, our presented numerical simulations do not directly approximate \eqref{FEF} with the support lines $\eta_{\pm1}(\R)$, but rather \eqref{FEF} with $\SX$ replaced by $\hat\SX$ from Lemma~\ref{lem:scaling}, with the choice of support lines $\ell_\pm = \{ (x^0_h(\pm1) \cdot e_1, \rho) : \rho \in \R\}$. For the case $\ell_- \not=\ell_+$ the lemma defines a rescaled equivalent flow for \eqref{FEF}, which crucially does not affect the condition \eqref{EQmtic} from Theorem~\ref{TMmain}. In the case $x^0_h(-1) \cdot e_1=x^0_h(1) \cdot e_1$ our numerical simulations approximate the flow \eqref{eq:limitFEF}.

We implemented \eqref{eq:fdel} within the
finite element toolbox Alberta, \cite{Alberta}, using
the sparse factorization package UMFPACK, see \cite{Davis04},
for the solution of the linear systems of equations arising at each time level.
For all our numerical simulations we use a uniform partitioning of $[-1,1]$,
so that $q_j = -1 + jh$, $j=0,\ldots,J$, with $h = \frac 2J$.
Unless otherwise stated, we use the discretization parameters $J=4096$
and $\Delta t = 10^{-4}$. When plotting several discrete curves in a single
figure, we use the following default colour convention: 
\textcolor{blue}{\bf blue} for the first curve,
\textcolor{red}{\bf red} for the last curve and
\textcolor{black}{\bf black} for any intermediate curves.
For later use, we define the discrete energies
\begin{equation} \label{eq:Em}
\E^m = \tfrac12 \int_I |P^m_h y^m_h |^2 |x^m_{h,\rho}| \,d\rho,
\end{equation}
where $P^m_h = I_2 - \frac{x^m_{h,\rho}}{|x^m_{h,\rho}|} \otimes
\frac{x^m_{h,\rho}}{|x^m_{h,\rho}|}$, as well as
\begin{equation} \label{eq:tildeEm}
\siE^m = \E^m \int_I |x^m_{h,\rho}| \,d\rho,
\end{equation}
recall \eqref{EQmtic}. 

\subsection{The curve from Figure~\ref{figure:3}} \label{sec:figure3}

Our first set of numerical simulations is for the initial curve depicted in
Figure~\ref{figure:3}. We begin by stating an explicit construction of a
possible parametrisation of this curve. Let
\begin{equation} \label{eq:a}
k:=0.71\quad \text{and}\quad a = a^\star :=\sqrt{\frac{\pi}{K(k)^2-\frac{1}{k^2}(K(k)^2-E(k)K(k))}}
=1.4123377600\ldots,
\end{equation}
recall \eqref{eq:rem3.6}, and choose $\mykappa$ as in \eqref{eq:rem3.5}.
We then define
\begin{equation} \label{eq:theta}
\theta(s) := \int_0^s \mykappa(\varrho)\, d\varrho = \frac{a}k \arctan \frac{k \sn(s,k)}{\dn(s,k)}
\quad \mbox{on}\quad [0,L(k)],
\end{equation}
and let, compare with \eqref{eq:gammac},
\begin{equation} \label{eq:x0special}
x_0(\rho) := \int_0^{\frac{1+\rho}2 L(k)} \binom{\cos \theta(s)}{\sin \theta(s)}\,ds
\quad \mbox{on}\quad \overline{I}.
\end{equation}
The discrete initial data is then defined as $x^0_h = \pi^h x_0$,
where $\pi^h:C^0(\overline{I},\mathbb R^2)\to \underline{V}^h$ 
is the standard Lagrange interpolation operator, defined
for any $v \in C^0(\overline{I},\mathbb R^2)$ via
$(\pi^h v)(\rho_j) = v(\rho_j)$, $j=0,\ldots,J$.
Here, for the calculation of the integrals in 
\eqref{eq:x0special}, we employ a Romberg integration method, and in 
addition make use of the symmetry
\begin{equation*} 
x_0(\rho) = 2 x_0(0) - x_0(- \rho) \quad \mbox{on}\quad [0,1],
\end{equation*}
so that the length of the domain of integration is never longer than 
$\frac12 L(k)$. For the computation of the upper integration limit and the 
integrand in \eqref{eq:x0special}, recall \eqref{eq:theta} and \eqref{eq:a}, we make use of the GNU Scientific Library 
(GSL) for accurate implementations of $E(k)$, $K(k)$, $\sn(\cdot, k)$ and
$\dn(\cdot, k)$.
In addition, for the discrete initial data $y^0_h$ we choose
$y^0_h = \pi^h y_0$, where $y_0 = \frac{x_{0,\rho\rho}}{|x_{0,\rho}|^2}$
is defined by
\begin{equation} \label{eq:grunau_y0}
y_0(\rho) := \hat y_0(\tfrac{1+\rho}2 L(k))
\quad \mbox{on}\quad \overline{I}, \qquad \text{with}\quad
\hat y_0(s) := \mykappa(s) 
\binom{-\sin\theta(s)}{\cos\theta(s)}
\quad \mbox{on}\quad [0,L(k)].
\end{equation}

We begin with a consistency test for the chosen discrete initial data
$(x_h^0, y_h^0)$, and investigate the error of the discrete energy
$\siE^0$, see \eqref{eq:tildeEm}, compared to the true value $2\pi$, where we
recall that
\[
\siE[x_0] = 
\tfrac12 L(k) \int^{L(k)}_0 u(s)^2\, ds = 2\pi.
\]
The results are shown in Table~\ref{tab:ediff} and confirm a consistent
approximation of the discrete energy.
\begin{table}
\center
\begin{tabular}{r|c|c|c|c}
$J$ & 1024 & 4096 & 16384 & 65536 \\
\hline
$\siE^0 - 2\pi$ & -4.8248e-05 & -3.0155e-06 & -1.8847e-07 & -1.1779e-08 \\
\end{tabular}
\caption{Numerical approximations of $\siE[x_0] = 2\pi$ for
\eqref{eq:tildeEm} with \eqref{eq:x0special}, \eqref{eq:grunau_y0}.}
\label{tab:ediff}
\end{table}%

Next we are interested in the evolution under the free boundary free elastic flow for this chosen 
initial data. To this end, we integrate the flow on the time interval 
$[0,T]$ for $T=10$, using our algorithm \eqref{eq:fdel}. As expected, the
evolution is very slow, and so for now we only consider the behaviour of the
discrete energy $\siE^m$ from \eqref{eq:tildeEm} at relatively early times. For a discussion of the
(long time) behaviour of the curve itself we refer to later parts of this
section. For now we show on the left of Figure~\ref{fig:siEm} 
the evolution of the discrete
energy $\siE^m$ for the four sets of discretization parameters from
Table~\ref{tab:ediff},
where for the time step sizes we fix $\Delta t = 0.2048 h$.
We can see that the three runs with the finer
discretization parameters agree that the curve begins to grow. In fact, the
finer the parameters, the clearer the growth. This indicates
that the chosen initial data does not go to a straight line under the flow.
\begin{figure}
\center
\includegraphics[angle=-0,width=0.45\textwidth]{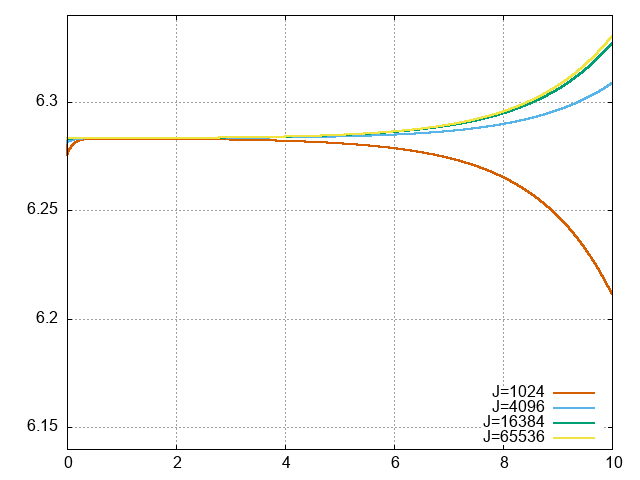}\qquad
\includegraphics[angle=-0,width=0.45\textwidth]{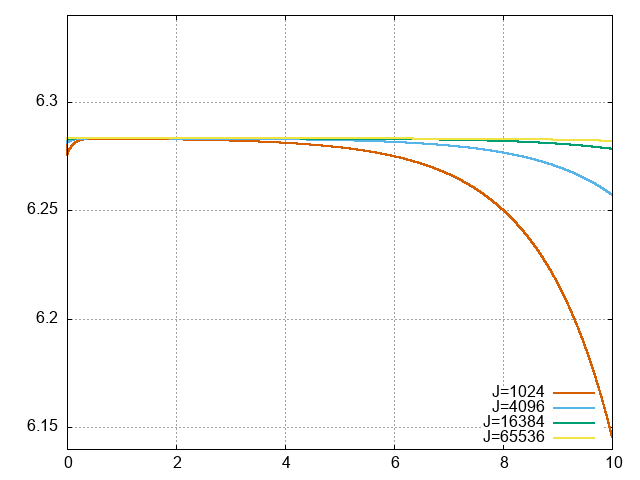}
\caption{The evolution of $\siE^m$ for the initial data \eqref{eq:x0special},
\eqref{eq:grunau_y0} (left), and for the initial data 
\eqref{eq:ere}, \eqref{eq:ere_y0} (right).
}
\label{fig:siEm}
\end{figure}%

\subsection{Euler's rectangular elastica}
As a comparison to the computations in the previous subsection, we now consider
a known stationary solution for the free boundary free elastic flow, that is, a half-period of  
Euler's rectangular elastica. We begin by stating a possible parametrisation of
this curve, compare with Appendix~\ref{AA}. Let
\[
k := \frac1{\sqrt{2}} \quad \text{ and } \quad L(k) := 2K(k),
\]
and choose
\begin{equation} \label{eq:ere}
x_0(\rho) := \hat x_0(\tfrac{1+\rho}2 L(k))
\quad \text{ on } \quad \overline{I},
\qquad \text{with} \quad
\hat x_0(s) = 
\begin{pmatrix} 2E(\AM(s,k),k) - s\\ -\sqrt2 \cn(s,k) \end{pmatrix} 
\quad \text{ on } \quad [0, L(k)].
\end{equation}
Correspondingly, we define $y_0 = \frac{x_{0,\rho\rho}}{|x_{0,\rho}|^2}$ as
\begin{equation} \label{eq:ere_y0}
y_0(\rho) := \hat y_0(\tfrac{1+\rho}2 L(k))
\quad \text{ on } \quad \overline{I},
\quad \text{with} \quad
\hat y_0(s) := 
\begin{pmatrix} -2\sn(s,k)\cn(s,)\dn(s,k) \\ \sqrt{2}\cn^3(s,k) 
\end{pmatrix} \quad \text{ on } \quad [0, L(k)].
\end{equation}
For the numerical experiments we then choose $x_h^0 = \pi^h x_0$ and
$y_h^0 = \pi^h y_0$ as before.
Similarly to Table~\ref{tab:ediff}, we confirm the consistency of the discrete
energy \eqref{eq:tildeEm} in Table~\ref{tab:ediff2}.
\begin{table}
\center
\begin{tabular}{r|c|c|c|c}
$J$ & 1024 & 4096 & 16384 & 65536 \\
\hline
$\siE^0 - 2\pi$ & -4.8216e-05 & -3.0135e-06 & -1.8834e-07 & -1.1771e-08\\
\end{tabular}
\caption{Numerical approximations of $\siE[x_0] = 2\pi$ for
\eqref{eq:tildeEm} with \eqref{eq:ere}, \eqref{eq:ere_y0}.}
\label{tab:ediff2}
\end{table}%
In addition, on the right of Figure~\ref{fig:siEm} we plot the evolution
for the discrete energy $\siE^m$ when using our algorithm \eqref{eq:fdel}
for the initial data \eqref{eq:ere}, \eqref{eq:ere_y0}. Of course, the
continuous initial data is a known stationary solution. However, as
discussed in Appendix~\ref{AB}, it is linearly unstable. Hence for a numerical
algorithm, due to the presence of numerical noise and rounding errors,
it is to be expected that the discrete approximations eventually move away from
the initial data. This is exactly what can be observed on the right of
Figure~\ref{fig:siEm}. However, unlike in the previous example (shown on the
left of the same figure), here the change in the energy gets smaller and
smaller, the finer the discretization parameters become. In other words, for finer
discretization parameters the algorithm
\eqref{eq:fdel} captures the stationarity of the continuous initial data
better.

\subsection{Curves close to the curve from Figure~\ref{figure:3}}

In this section we consider the same initial data as in {Section~}\ref{sec:figure3},
but now instead of \eqref{eq:a} we just fix $k=0.71$, and let
$a$ vary. We then still define the initial curve via \eqref{eq:x0special} with
\eqref{eq:theta}, and let $y_0$ be defined as in \eqref{eq:grunau_y0}. 
If we choose $a = 1.4 < a^\star$, then the curve very quickly evolves
to a straight line, as can be seen in Figure~\ref{fig:a14}.
\begin{figure}
\center
\includegraphics[angle=-0,width=0.25\textwidth]{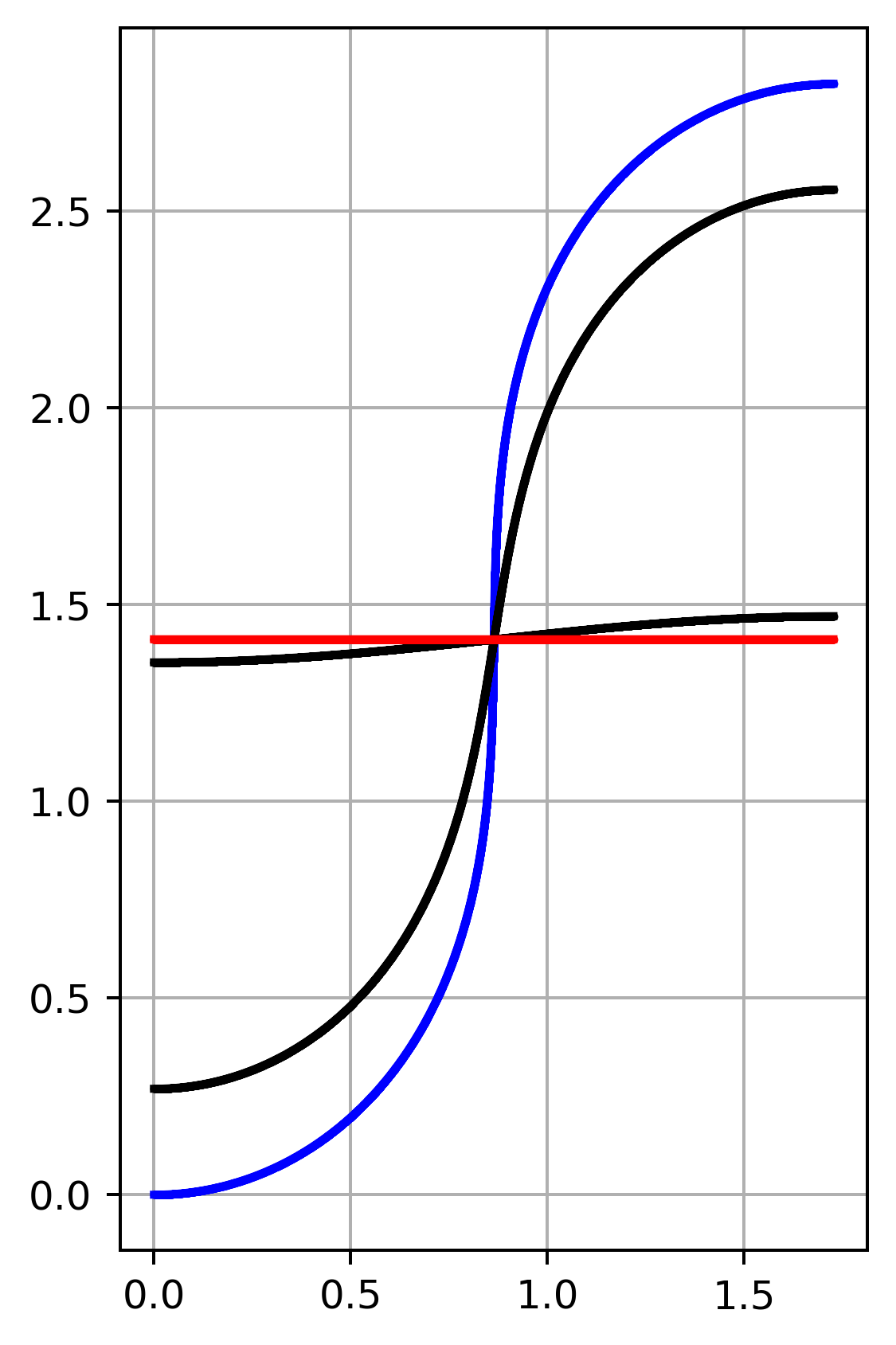}
\qquad
\includegraphics[angle=-0,width=0.45\textwidth]{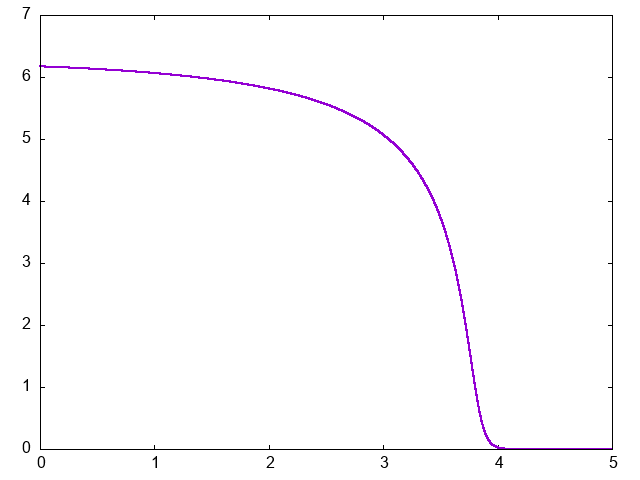}
\caption{Free boundary free elastic flow for a curve with $a=1.4$. We show
$x_h^m$ at times $t=0,3,4,5$, and a plot of $\siE^m$ over time.
}
\label{fig:a14}
\end{figure}%
However, if we let $a = 1.5 > a^\star$, then the curve starts to grow
indefinitely, as can be seen in Figure~\ref{fig:a15}.
\begin{figure}
\center
\includegraphics[angle=-0,width=0.12\textwidth]{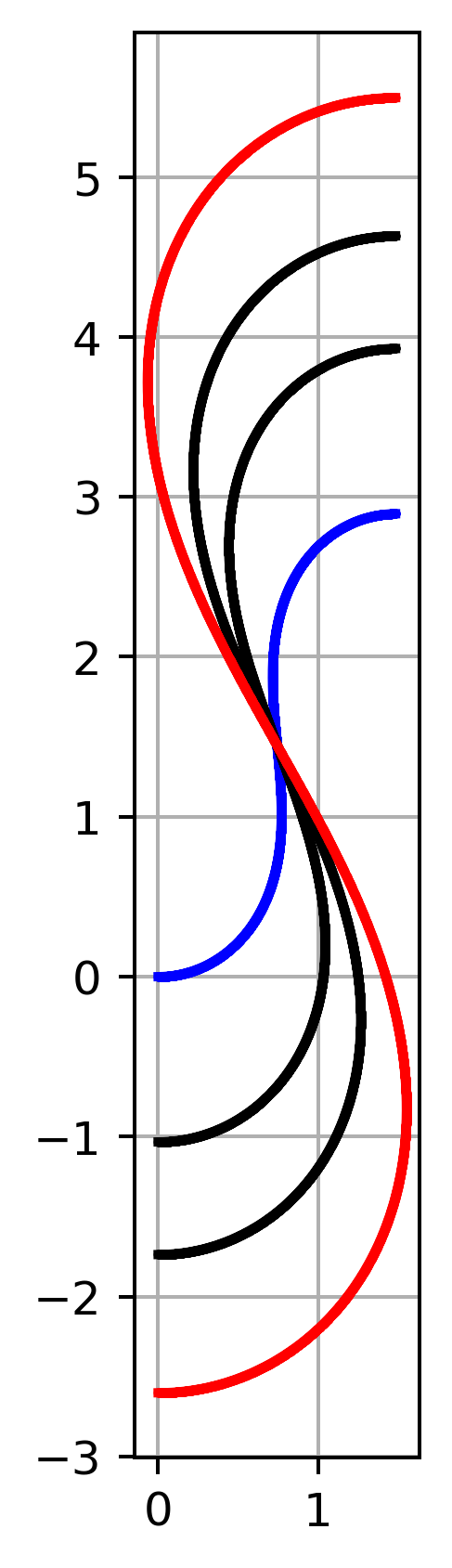}
\includegraphics[angle=-0,width=0.24\textwidth]{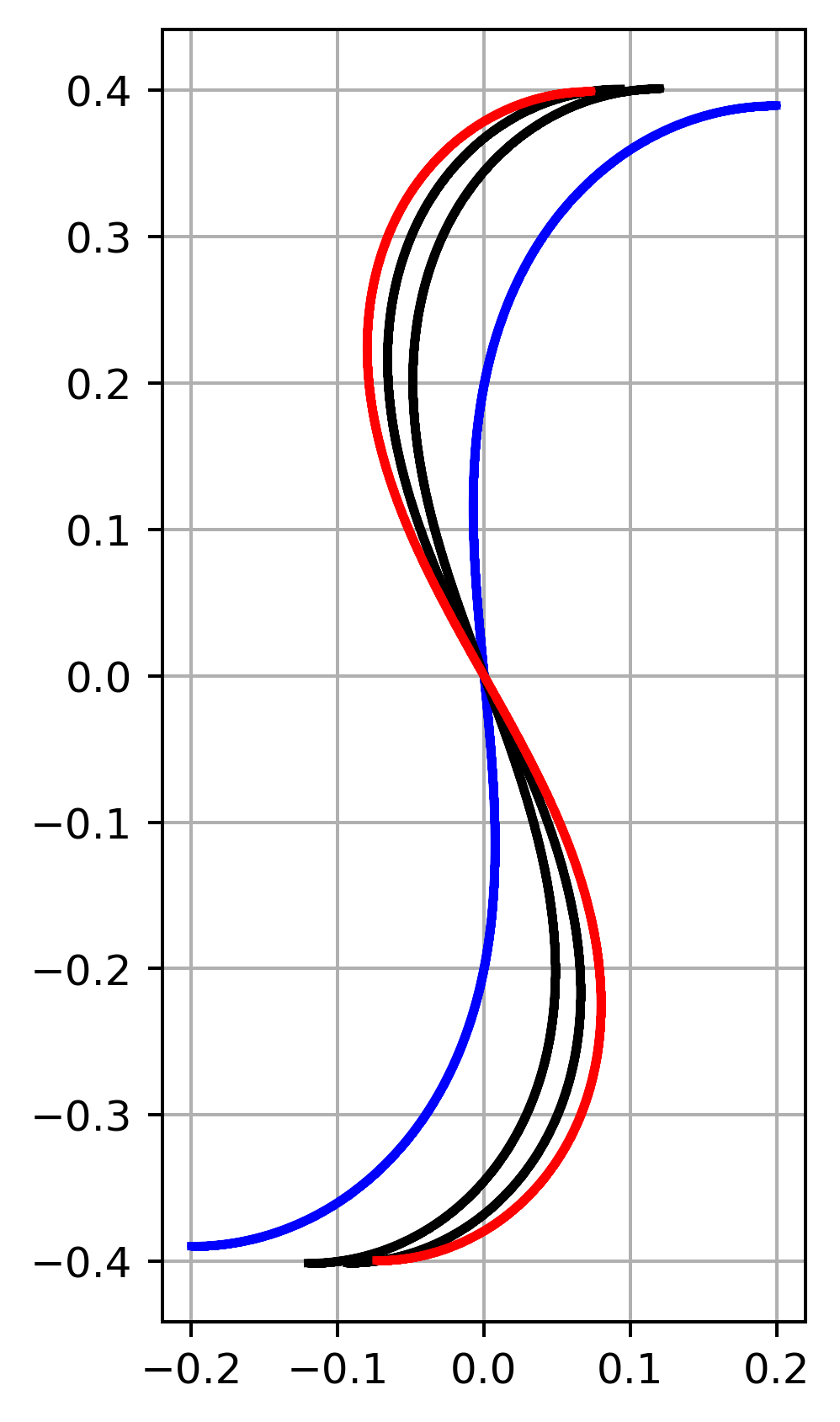}
\qquad
\includegraphics[angle=-0,width=0.45\textwidth]{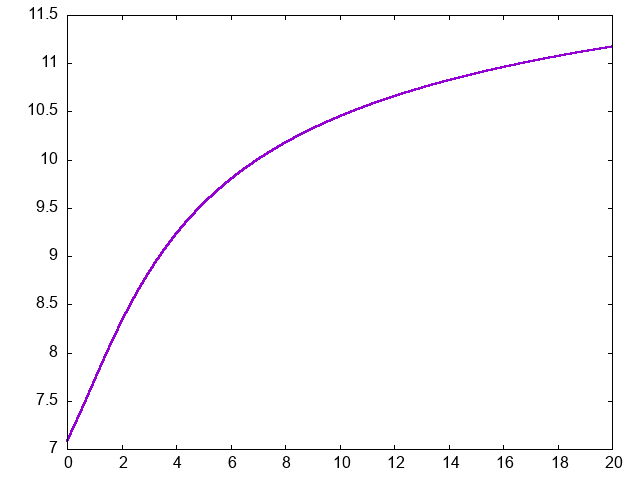}
\caption{Free boundary free elastic flow for a curve with $a=1.5$. We show
$x_h^m$ at times $t=0,5,10,20$, the same curves rescaled by length,
and a plot of $\siE^m$ over time.
}
\label{fig:a15}
\end{figure}%

Inspired by these two simulations, we now attempt to find reliable lower and
upper bounds for the critical value $a_c$ of $a$ at the border between 
``going to a straight line'' and ``growing indefinitely''. To this end we use
our scheme \eqref{eq:fdel} for the described initial data, and observe whether
at large times the curve has evolved to a straight line, or whether it
continues to grow. 
The numerical values reported in Table~\ref{tab:ac}, using the same
discretization parameters as in Figure~\ref{fig:siEm}, suggest
that the critical value $a_c$ lies in the interval
$(1.41233, 1.41234)$.
\begin{table}
\center
\begin{tabular}{r|c|c}
$J$ &
lower bound: $a$ and $\siE^0/(2\pi)$ & upper bound: $a$ and $\siE^0/(2\pi)$ \\
\hline
1024 & 1.4123: 9.9994e-01 & 1.4124: 1.0001e-00 \\ 
4096 & 1.41233: 9.9999e-01 & 1.41234: 1.000e-00 \\ 
16384& 1.41233: 9.9999e-01 & 1.41234: 1.000e-00 \\ 
65536& 1.41233: 9.9999e-01 & 1.41234: 1.000e-00 \\
\end{tabular}
\caption{Numerical estimates for lower and upper bounds on $a_c$.
}
\label{tab:ac}
\end{table}%

\subsection{Long time behaviour}

Looking at Figure~\ref{fig:a15}, a natural question arises: what is the
limiting shape of the rescaled curves? 
Recall also Conjecture~\ref{conj:3}.
In this section we investigate the long
time behaviour for various interesting initial data. Here we will always define
some initial discrete parametrisation $x_0^h$, and then compute $y_0^h \in
\underline{V}^h_\partial$ as the unique solution to
\begin{equation} \label{eq:y0}
\int_I y^{0}_h \cdot \xi |x^0_{h,\rho}|^2\, d\rho 
+ \int_I x^{0}_{h,\rho} \cdot \xi_\rho\,d\rho = 0
\qquad \forall\ \xi \in \underline{V}^h_\partial,
\end{equation}
compare with \eqref{eq:fdelb}.

We begin with an initial curve that is made up of two straight line segments
and two unit circles. Note that since the curvature undergoes two jumps, this
is not a possible minimizer or stationary state. Yet, one could imagine that the
curve finds some (rescaled) stationary solution close by. Instead, the
simulation shown in Figure~\ref{fig:8} demonstrates an elaborate untangling,
with the growing curve eventually attaining the limiting shape already seen
in Figure~\ref{fig:a15}. To enable this extremely long computation, we have
employed an adaptive time step strategy for this simulation.
\begin{figure}
\center
\includegraphics[angle=-0,width=0.25\textwidth]{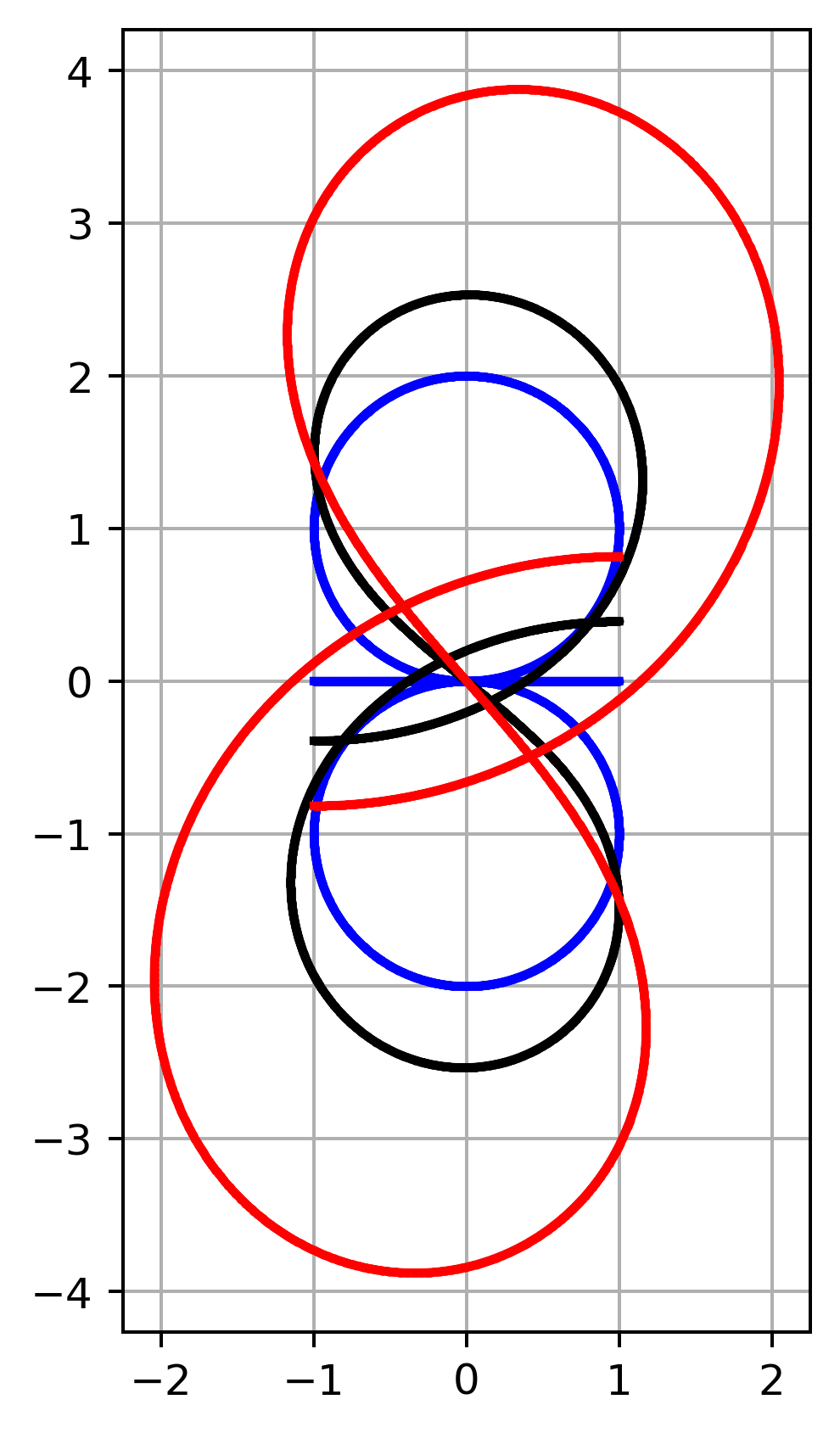}\
\includegraphics[angle=-0,width=0.26\textwidth]{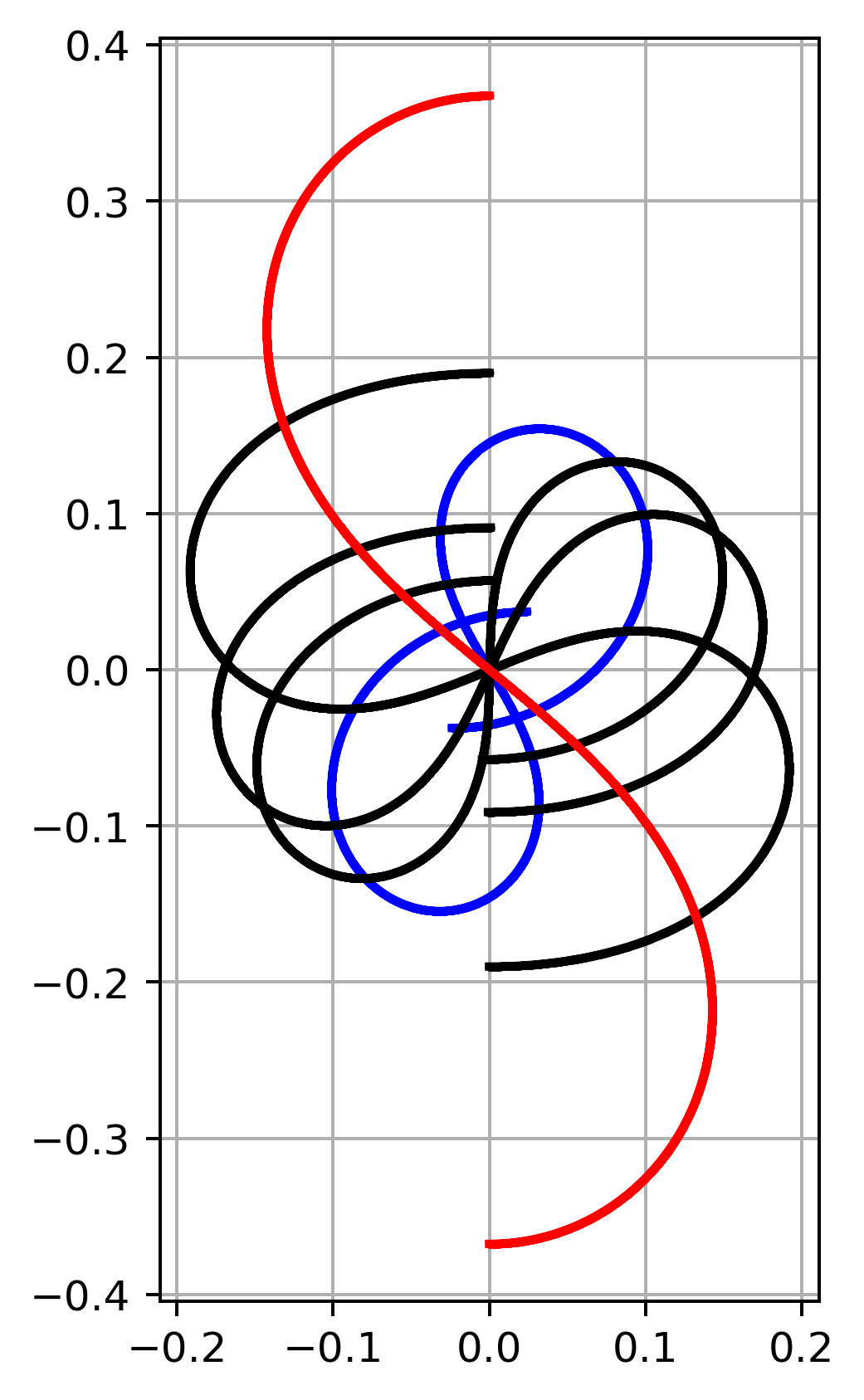}
\caption{Free boundary free elastic flow for a figure-eight curve made up of two straight line
segments and two unit circles. We show
$x_h^m$ at times $t=0,1,10$ (left), 
and at times $t=100,10^5, 10^8, 10^{11}, 10^{12}$
rescaled by length (right).
}
\label{fig:8}
\end{figure}%

The observed simulation suggests that eventually the (rescaled) curve 
approaches a half-period of Bernoulli's lemniscate, a well-known self-similar expander
for the flow \eqref{eq:limitFEF}. With the next experiment we want to
investigate whether starting with a full period of Euler's rectangular elastica might lead
to a different long term evolution. Here we observe that the initial curve has
both ends attached to the $y$-axis, but that it does
not satisfy the curvature boundary conditions, and so it is not a steady state solution
of \eqref{eq:limitFEF}. 
However, the numerical simulation shown in Figure~\ref{fig:elastica} indicates
that also in this case the shape of half a lemniscate is
approached at large times. In particular, we observe that the rescaled shapes
at times $t=10$ and $t=20$ overlap perfectly to the naked eye.
\begin{figure}
\center
\includegraphics[angle=-0,width=0.26\textwidth]{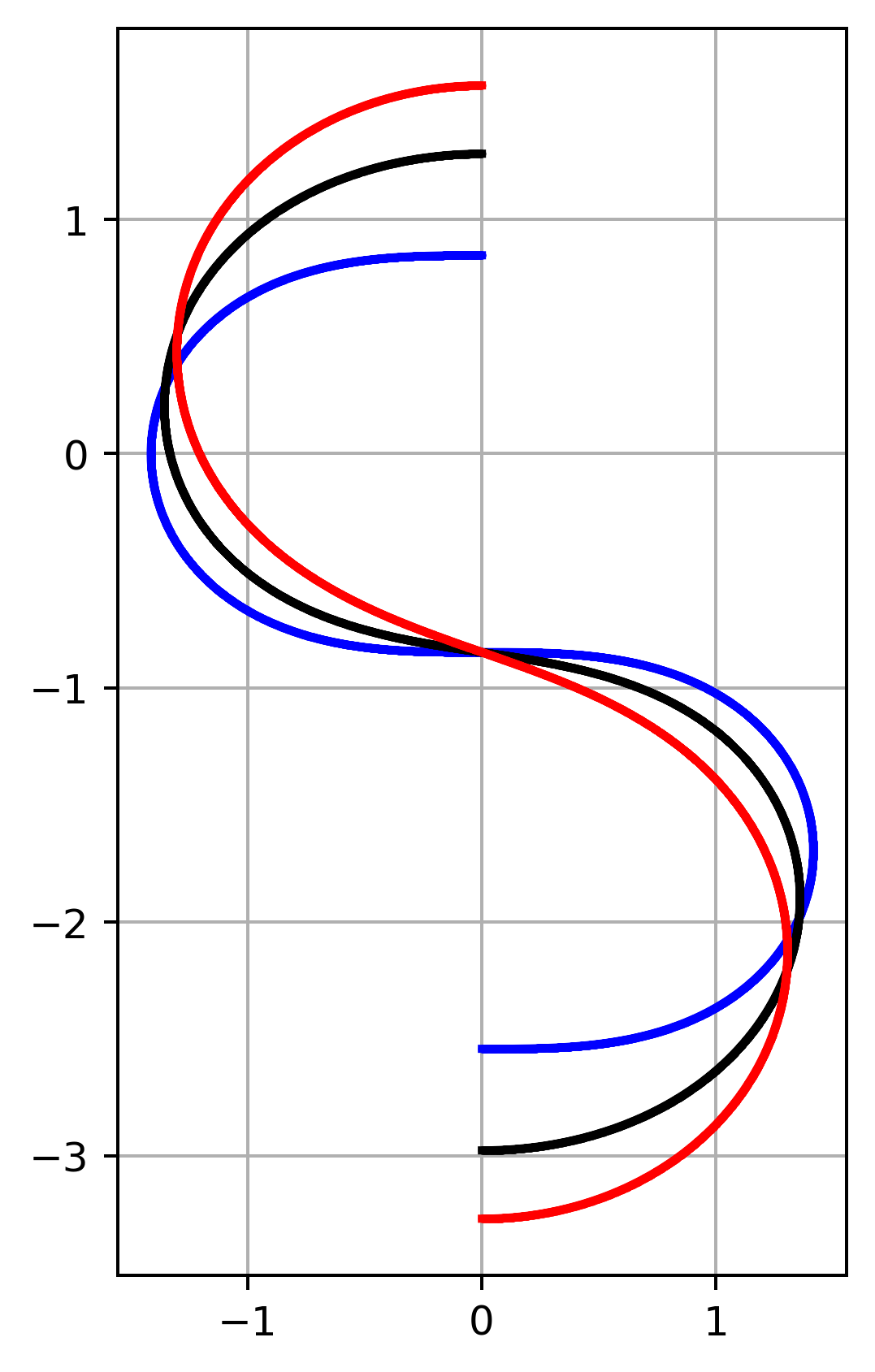}
\includegraphics[angle=-0,width=0.24\textwidth]{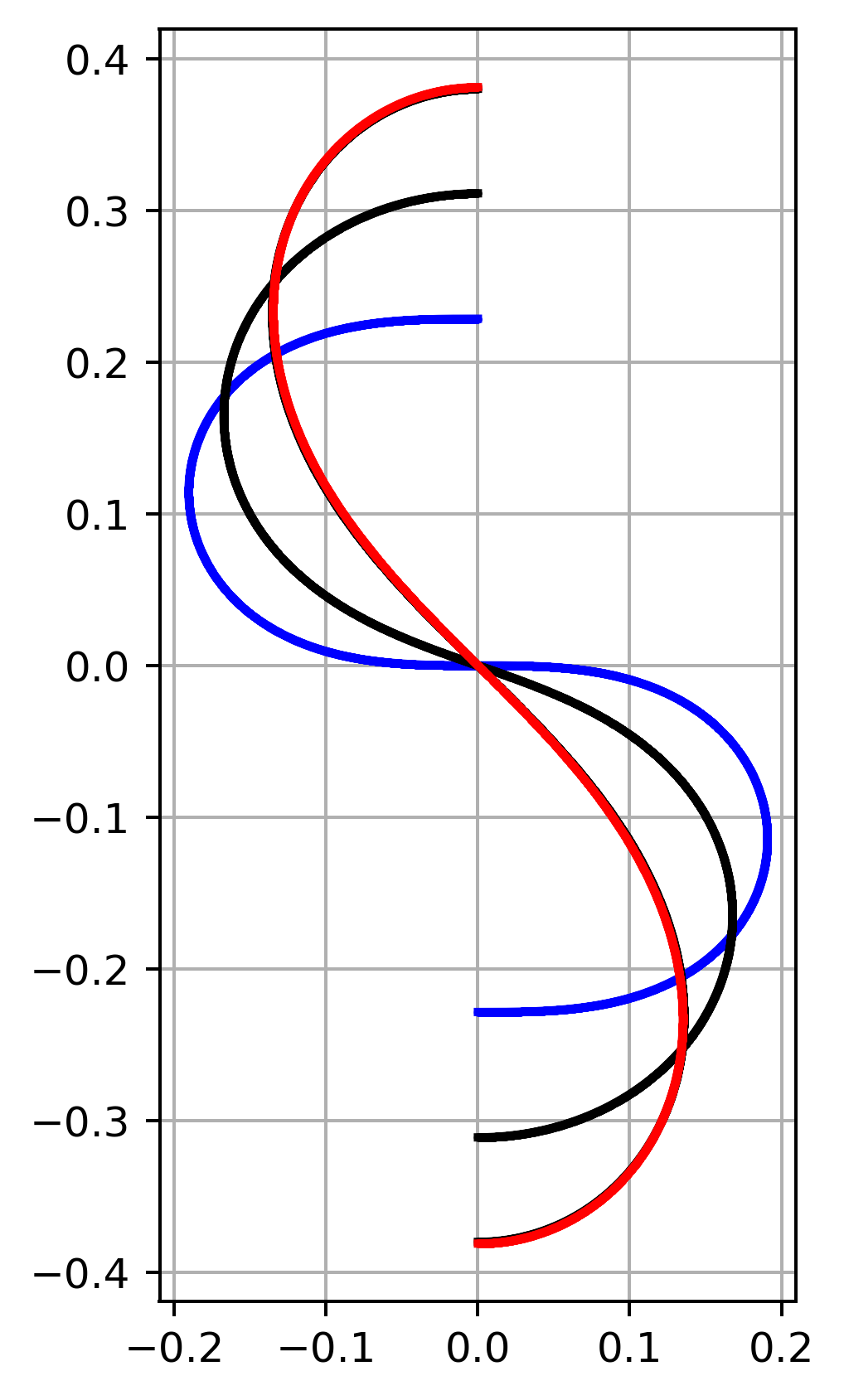}
\qquad
\includegraphics[angle=-0,width=0.18\textwidth]{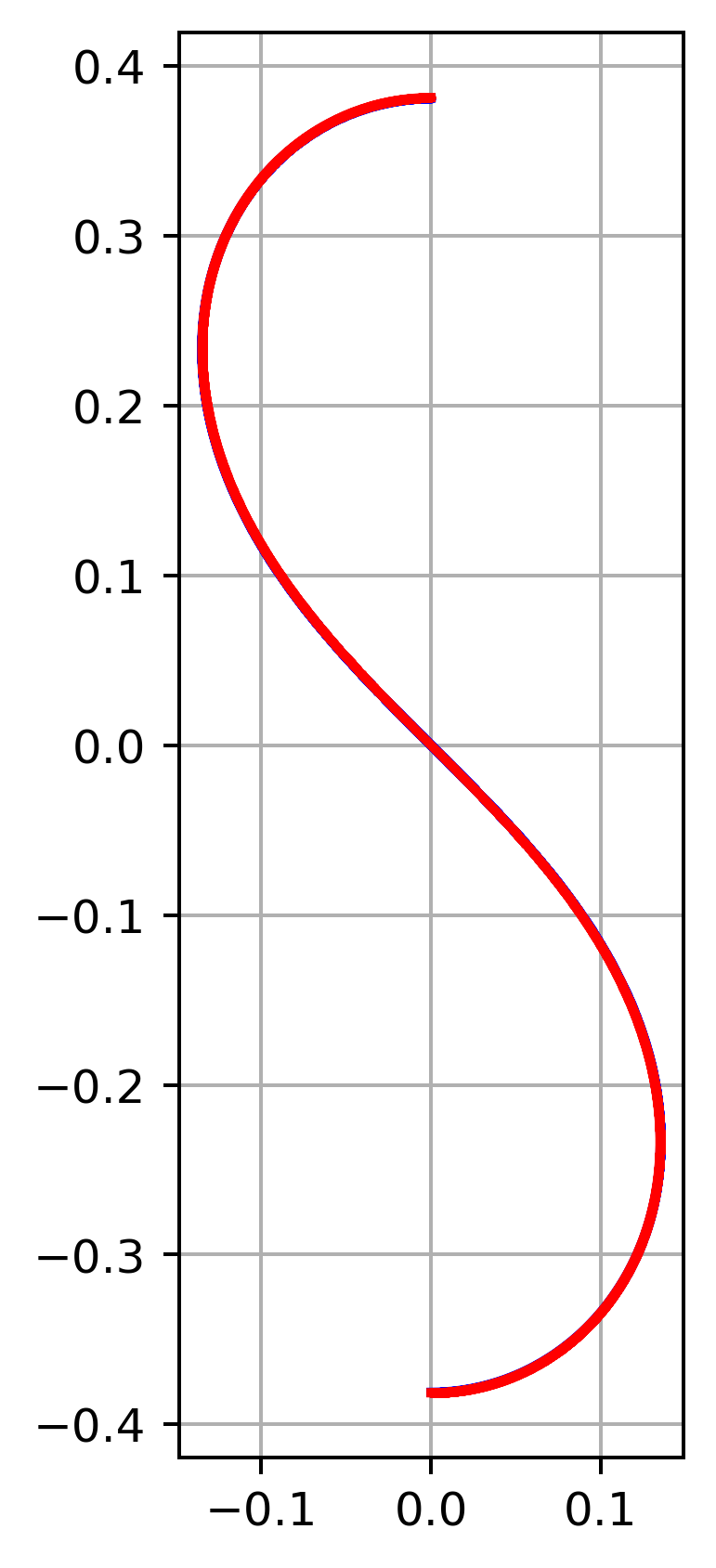}
\caption{Free boundary free elastic flow for a rectangular elastica. We show
$x_h^m$ at times $t=0,0.5,1$ (left), 
and at times $t=0,1,10,20$ rescaled by length (middle).
On the right we show a comparison between $x^h_m$ at time $t=20$ (blue)
and half a lemniscate (red, overlapping), both rescaled by length.
}
\label{fig:elastica}
\end{figure}%

As a final confirmation that the limiting shape is indeed that of half a
lemniscate, we show on the right of Figure~\ref{fig:elastica} a comparison
between the rescaled curve $x^h_m$ at the final time $t=20$, and a
rescaled version of the lemniscate curve defined by the parametrisation
\[
x_0(\rho) =
\begin{cases}
(-\cos^\frac12(\frac{1+\rho}2 \pi) \sin(\frac{1+\rho}4 \pi), 
\cos^\frac12(\frac{1+\rho}2 \pi) \cos(\frac{1+\rho}4 \pi))^T
& \rho < 0, \\
(\cos^\frac12(\frac{1-\rho}2 \pi) \sin(\frac{1-\rho}4 \pi), 
-\cos^\frac12(\frac{1-\rho}2 \pi) \cos(\frac{1-\rho}4 \pi))^T
& \rho \geq 0. \\
\end{cases}
\]
As expected, the two shapes on the right of Figure~\ref{fig:elastica} 
overlap perfectly.


\section*{Acknowledgements}

The Otto-von-Guericke Universit\"at Magdeburg supported two visits of the fourth author and one visit of the third author, where this work was both initiated and in the majority completed.
They are grateful for the support and thank their hosts for their generous hospitality. The fifth author acknowledges support from ARC DECRA DE190100379.

\appendix

\section{On Euler's rectangular elastica}
\label{AA}

{In this appendix we derive the parametrisations used in Figure~\ref{figrect}, as well as some of their properties.}

{Throughout, we fix} the Jacobi modulus 
$$k=\tfrac1{\sqrt2}. 
$$
Let \(\sn,\cn,\dn\) denote the Jacobi elliptic functions \textcolor{black}{from \eqref{eq:2.1.04} and \eqref{eq:2.1.05}}, and let 
$$
K=K(k)=K\left(\tfrac1{\sqrt2}\right), \quad E=E(k)=E\left(\tfrac1{\sqrt2}\right)
$$
be the complete elliptic integrals of the first and second kind {from \eqref{eq:2.1.01}
and \eqref{eq:2.1.02}.}
The turning angle \(\theta\) {of Euler's rectangular elastica} satisfies 
\begin{equation}\label{eq:thetaprime}
\theta'(s)=\kappa(s){=\sqrt2\,\mya\,\cn(\mya s,k)},
\end{equation} 
{where $\mya>0$ is a parameter that will be fixed later. Using}
\(
\mya{\int_0^s \cn(\mya {r},k)\,d{r}}=\sqrt2\arcsin\!\big({k}\,\sn(\mya s,{k})\big)  
\)
{we have that}
\(
\theta(s)=2\arcsin\!\big({k}\,\sn(\mya s,{k})\big),
\)
{and}
consequently
\[
\cos\theta(s)=\cn^2(\mya s,{k}),\qquad \sin\theta(s)=\sqrt2\,\sn(\mya s,{k})\,\dn(\mya s,{k}).
\]
{An} explicit arclength parametrisation \({\gamma(s)=}(\gamma^1(s),\gamma^2(s))\) is obtained by integrating \({(\gamma^1)}'=\cos\theta\), \({(\gamma^2)}'=\sin\theta\):
\begin{equation}
\gamma^1(s)=x_0+\frac1\mya\Big(2\,E({\AM}(\mya s,{k}),{k})-\mya s\Big), \qquad 
\gamma^2(s)=y_0-\frac{\sqrt2}{\mya}\,\cn(\mya s,{k}), \label{eq:gammar}
\end{equation}
where \(E({\cdot,k})\) is the incomplete elliptic integral of the second kind {from \eqref{eq:2.1.02}
and $\AM$ is the Jacobi amplitude from \eqref{eq:2.1.03}}.
{By construction} \(\sqrt{(\gamma^1)'(s)^2+(\gamma^2)'(s)^2}=1\), {and} so \(s\) is arclength.

{We now determine $\mya$ so that we can fit $m$ half-periods between the two support lines $\eta_{\pm1}$. First, on recalling
\eqref{eq:thetaprime}, we have that}
\[
\kappa'(s)=-\sqrt2\,\mya^2\,\sn(\mya s,{k})\,\dn(\mya s,{k}).
\]
Hence \(\kappa'(s)=0\) at {$s=\frac{2jK}\mya$} {for $j\in{\mathbb Z}$} (since \(\sn(2jK,{k})=0\)), and there the tangent {$\gamma'(s)$} is {horizontal}.
For the arclength parametrisation
the horizontal span from {\(s=0\)} to {\(s=\frac{2K}\mya\)} is
\[
\Delta \gamma^1_{\text{half}}={\gamma^1}\!\left(\tfrac{2K}{\mya}\right)-{\gamma^1(0)}
=\frac{1}{\mya}\big(2E(\pi,{k})-2K{(k)}\big)
=\frac{2}{\mya}\big({2E-K}\big).
\]
If there are \(m\) half-periods between the vertical lines \(\eta_{\pm1}\) (total width \(2\)), we obtain
\[
2=m\,\Delta \gamma^1_{\text{half}}
\quad\Longrightarrow\quad
\mya=m\big({2E-K}\big).
\]
With this choice, {the length of the curve $\gamma$ is given by}
\[
L=\frac{2m{K}}{\mya}={\frac{2K}{\,2E-K\,}},
\]
independent of \(m\).
{Moreover, for the bending energies of the curve we obtain $\E[\gamma]=2m\mya (2E-K)$
and $\siE[\gamma]=4m^2 K (2E-K)$, respectively.
On recalling the} classical special values
\(
K(\tfrac1{\sqrt2})=\frac{\Gamma(\tfrac14)^2}{4\sqrt{\pi}}\)
and
\(E(\tfrac1{\sqrt2})=\frac{\Gamma(\tfrac14)^2}{8\sqrt{\pi}}+\frac{\Gamma(\tfrac34)^2}{2\sqrt{\pi}}
\),
{we have that}
\(
2E-K=\frac{\Gamma(\tfrac34)^2}{\sqrt{\pi}}\),
{and so $\mya=m \frac{\Gamma(\tfrac34)^2}{\sqrt{\pi}}$, as well as
$L=\tfrac12\big(\Gamma(\tfrac14)/\Gamma(\tfrac34)\big)^2$
and $\siE[\gamma] = 2\pi m^2$.}

\section{Linearisation of the flow about Euler's rectangular elastica}
\label{AB}
\newcommand{\myeta}{{\xi}}

In this appendix we show that Euler's rectangular elastica are linearly unstable for
the free boundary free elastic flow.

Consider
\[
\myeta(\myu,\varepsilon) = \gamma(\myu) + \varepsilon\, v(\myu)\,\normal(\myu),
\]
where for the moment $\gamma\in\SX$ is any curve satisfying the free boundary
conditions. Later, we will choose $\gamma=\gamma_r$, where $\gamma_r$ is Euler's
rectangular elastica.

For each $\varepsilon$, let $s_\varepsilon$ denote arclength along $\myeta(\cdot,\varepsilon)$,
and let $\tangent_\varepsilon,\normal_\varepsilon,\kappa_\varepsilon$ denote its
unit tangent, unit normal, and scalar curvature, respectively. We write
$s:=s_0$, $\tangent:=\tangent_0$, $\normal:=\normal_0$, $\kappa_0:=\kappa_{\varepsilon}\!\mid_{\varepsilon=0}$
for the corresponding quantities of the base curve $\gamma$.

In order for $\myeta(\cdot,\varepsilon)$ to remain in $\SX$, we require
\[
v_\myu(\pm1) = v_{\myu\myu\myu}(\pm1) = 0\,.
\]
To derive this, it suffices to calculate the tangent, normal, curvature, and their
derivatives for $\myeta$ in terms of the corresponding quantities for $\gamma$, and use
that $\gamma\in \SX$.

The commutator of $\partial_\varepsilon$ and $\partial_{s_\varepsilon}$ is
\[
[\partial_\varepsilon, \partial_{s_\varepsilon}]
 = v\,\kappa_\varepsilon\,\partial_{s_\varepsilon}\,,
\]
which follows from
\[
\partial_\varepsilon\partial_{s_\varepsilon}
= \partial_\varepsilon\left(\frac1{|\myeta_\myu|}\partial_\myu\right)
=
\partial_\varepsilon\left(\frac1{|\myeta_\myu|}\right)\partial_\myu
+ \frac1{|\myeta_\myu|}\partial_\varepsilon\partial_\myu
=
-\frac{\partial_\varepsilon|\myeta_\myu|}{|\myeta_\myu|^2}\partial_\myu
+ \partial_{s_\varepsilon}\partial_\varepsilon
=
-(\partial_\varepsilon\log |\myeta_\myu|)\,\partial_{s_\varepsilon}
+ \partial_{s_\varepsilon}\partial_\varepsilon
\]
and
\[
2|\myeta_\myu|\partial_\varepsilon|\myeta_\myu|
= \partial_\varepsilon|\myeta_\myu|^2
 = 2\partial_\myu(\partial_\varepsilon\myeta)\cdot \tangent_\varepsilon|\myeta_\myu|
 = 2\partial_{s_\varepsilon}(\partial_\varepsilon\myeta)\cdot \tangent_\varepsilon|\myeta_\myu|^2
 = -2\kappa_\varepsilon v |\myeta_\myu|^2
 \,.
\]

Evaluating at $\varepsilon=0$ (so that $\partial_{s_\varepsilon}=\partial_s$ and $\kappa_\varepsilon=\kappa_0$),
we obtain
\begin{align*}
\partial_\varepsilon \tangent_\varepsilon\Big|_{\varepsilon=0}
 &= \partial_s (v\normal)
 + [\partial_\varepsilon,\partial_{s_\varepsilon}]\myeta\Big|_{\varepsilon=0}
 = \partial_s (v\normal)
 + v\kappa_0 \tangent
 = v_s\normal\,,\\
\partial_\varepsilon \normal_\varepsilon\Big|_{\varepsilon=0}
& = (\partial_\varepsilon \normal_\varepsilon\cdot \tangent)\tangent
 = -(\partial_\varepsilon \tangent_\varepsilon\cdot \normal)\tangent
 = -v_s\tangent\,.
\end{align*}
Moreover, since $\kappa_\varepsilon \normal_\varepsilon = (\myeta(\cdot,\varepsilon))_{s_\varepsilon s_\varepsilon}$
is the curvature vector of $\myeta(\cdot,\varepsilon)$, we have
\[
\partial_\varepsilon (\kappa_\varepsilon \normal_\varepsilon)\Big|_{\varepsilon=0}
 = \partial_s\partial_\varepsilon \tangent_\varepsilon\Big|_{\varepsilon=0}
 +  [\partial_\varepsilon,\partial_{s_\varepsilon}]\tangent_\varepsilon\Big|_{\varepsilon=0}
 = \partial_s(v_s\normal)
 +  v\kappa_0^2 \normal
 = (v_{ss} + v\kappa_0^2)\normal
 - \kappa_0 v_s \tangent
 \,.
\]
Thus the first variation of the scalar curvature is
\[
\partial_\varepsilon \kappa_\varepsilon\Big|_{\varepsilon=0}
 = \normal\cdot \partial_\varepsilon (\kappa_\varepsilon \normal_\varepsilon)\Big|_{\varepsilon=0}
 = v_{ss} + v\kappa_0^2
 \,.
\]
Using the commutator twice more (and always evaluating at $\varepsilon=0$) gives
\begin{align*}
\partial_\varepsilon (\kappa_\varepsilon)_s\Big|_{\varepsilon=0}
 &= \partial_s(v_{ss} + v\kappa_0^2)
  + v\kappa_0(\kappa_0)_s
 = v_{sss} + v_s\kappa_0^2 + 3v\kappa_0(\kappa_0)_s
\,, \\
\partial_\varepsilon (\kappa_\varepsilon)_{ss}\Big|_{\varepsilon=0}
 &= \partial_s\!\left(
 v_{sss} + v_s\kappa_0^2 + 3v\kappa_0(\kappa_0)_s
 \right)
 + v\kappa_0(\kappa_0)_{ss}
\\ &=
 v_{ssss} + v_{ss}\kappa_0^2 + 5v_s\kappa_0(\kappa_0)_s
 + 3v(\kappa_0)_s^2 + 4v\kappa_0(\kappa_0)_{ss}
\,.
\end{align*}
Consequently,
\begin{align}
\partial_\varepsilon\!\left((\kappa_\varepsilon)_{ss} + \frac12\kappa_\varepsilon^3\right)\Big|_{\varepsilon=0}
&=
v_{ssss} + \textcolor{black}{v_{ss}\kappa_0^2} + 5v_s\kappa_0(\kappa_0)_s
 + 3v(\kappa_0)_s^2 + 4v\kappa_0(\kappa_0)_{ss}
+ \frac32\kappa_0^2(v_{ss} + v\kappa_0^2)
\notag\\
&=
v_{ssss} + \frac52(v_{s}\kappa_0^2)_s
 + v\left(3(\kappa_0)_s^2 + 4\kappa_0(\kappa_0)_{ss}
 + \frac32\kappa_0^4\right)
 \,.
 \label{EQlin}
\end{align}

A straightforward check is to linearise around the horizontal line $\gamma_\ell$.
Since $\kappa_0\equiv 0$ for $\gamma_\ell$, \eqref{EQlin} gives
\[
\partial_\varepsilon\!\left((\kappa_\varepsilon)_{ss} + \frac12\kappa_\varepsilon^3\right)\Big|_{\gamma=\gamma_\ell}(s)
= v_{ssss}\,.
\]
The spectrum contains a zero eigenvalue corresponding to vertical translations, and
then every other eigenvalue is strictly positive.

Now let us specialise to Euler's elastica. To this end, we recall (up to translation)
its arclength parametrisation from \eqref{eq:gammar} in Appendix~\ref{AA}:
\[
\gamma_r(s) = \begin{pmatrix}
    2E(\AM(s,k),k) - s\\[0.2em]
    -\sqrt2 \cn(s,k)
\end{pmatrix},
\qquad k=\frac1{\sqrt2},
\]
where $E(\,\cdot\,,k)$ is the incomplete elliptic integral of the second kind and
$K=K(k)$ is the complete elliptic integral of the first kind.
\textcolor{black}{The curvature of $\gamma_r$ has period $4K$; the curve itself is periodic only up to horizontal translation.}
Important here is that the half-period $\gamma_r:[0,2K]\to\R^2$ satisfies the boundary
conditions of the free boundary free elastic flow, and also satisfies
\[
(\kappa_0)_{ss} + \frac12\kappa_0^3 = 0,
\]
thus is a stationary solution to the free boundary free elastic flow.

We recall from \eqref{eq:thetaprime} that
\[
\kappa_0(s) = \sqrt2 \cn(s,k),
\]
and moreover one can show that $(\kappa_0)_s^2 + \frac14\kappa_0^4 = 1$.
Using these identities in \eqref{EQlin} gives
\begin{align}
\partial_\varepsilon\!\left((\kappa_\varepsilon)_{ss} + \frac12\kappa_\varepsilon^3\right)\Big|_{\gamma=\gamma_r}(s)
 &= 
v_{ssss}(s) + 5\big(v_{s}(s)\cn^2(s,k)\big)_s
 + v(s)\left(3 - \frac54\kappa_0^4(s)\right)
\notag\\
&=
v_{ssss}(s) + 5\big(v_{s}(s)\cn^2(s,k)\big)_s
 + v(s)\left(3 - 5\cn^4(s,k)\right).
\label{EQlin2}
\end{align}
The linearised operator determining the stability or otherwise of the half-period of
the rectangular elastica is thus
\[
\SL v = 
v_{ssss}(s) + 5\big(v_{s}(s)\cn^2(s,k)\big)_s
 + v(s)\left(3 - 5\cn^4(s,k)\right)
 \,.
\]
The self-adjoint quadratic form associated to $\SL$ is
\[
\Q[v] = \int_0^{2K}
    \Big(v_{ss}^2(s)
    - 5\cn^2(s,k)\,v_s^2(s)
    + \big(3-5\cn^4(s,k)\big)v^2(s)\Big)\,
    ds
\,.
\]
Take $v = v_{-1} = \cos\!\Big(\frac{\pi s}{2K}\Big)$.
Then
\[
\Q[v_{-1}] = \int_0^{2K}
    \bigg[\bigg(\frac{\pi}{2K}\bigg)^4 
    + \big(3-5\cn^4(s,k)\big) \bigg]v_{-1}^2
    \,
    ds
    - 5\int_0^{2K}\cn^2(s,k)\bigg(\frac{\pi}{2K}\bigg)^2
    \sin^2\!\Big(\tfrac{\pi s}{2K}\Big)
    \,
    ds
\,.
\]
Numerical evaluation of this integral gives $\Q[v_{-1}] = -0.1646\ldots < 0$.

Thus the self-adjoint operator $\SL$ has a negative eigenvalue, and hence the linearised evolution $\partial_t v=-\SL v$ admits an exponentially growing mode; therefore Euler's rectangular elastica are linearly unstable for the free boundary free elastic flow.

More than a single half-period of Euler's rectangular elastica gives additional
unstable directions, corresponding to variations across half-periods. In any case,
none of them are linearly stable.

\end{document}